\documentclass[a4paper,10pt]{article}
\usepackage{authblk}

\usepackage[english]{babel}
\usepackage[latin1]{inputenc}
\usepackage{csquotes}
\usepackage[normalem]{ulem}
\usepackage{amsfonts,amsmath,amsthm}
\usepackage{empheq}
\usepackage{cancel}
\usepackage[titletoc,title]{appendix}
\usepackage[backend=bibtex8,doi=false,eprint=false,giveninits=true,isbn=false,style=numeric-comp,url=false,maxnames=99]{biblatex}
\makeatletter
\def\blx@maxline{77}
\makeatother
\bibliography{bibl.bib}
\DeclareRedundantLanguages{english,german,french}{english,german,ngerman,french}

\usepackage[section]{placeins}
\usepackage{cases}
\usepackage{mathabx}
\usepackage{xfrac}
\usepackage{fancyhdr}
\usepackage{color}
\usepackage[colorlinks]{hyperref}
\definecolor{blue75}{rgb}{0,0,.75}
\definecolor{green75}{rgb}{0,.75,0}
\hypersetup{colorlinks=true, urlcolor=blue75,linkcolor=blue75,citecolor=green75,pdfstartview=FitB,bookmarksopen=true,bookmarksopenlevel=1}
\usepackage[a4paper, left=2.5cm, right=2.5cm, top=2.5cm,bottom=2cm]{geometry}
\usepackage{constants}
\newconstantfamily{C}{
	symbol=C,
	format=\parenthezises,
	reset={section}
}
\usepackage{enumerate}

\usepackage{graphicx}
\graphicspath{ {images/} }
\usepackage{wrapfig}
\usepackage{figbib}
\allowdisplaybreaks
\begin{document}
	\newcommand{\ua}{u^{\alpha}}
	\newcommand{\ub}{u^{\beta}}
	\newcommand{\wg}{w^{\gamma}}
	\newcommand{\io}{\int_{\Omega}}
	\newcommand{\iot}{\int_0^T}
	\newcommand{\mut}{\tilde{\mu}_3}
	\newcommand{\td}{\text{d}}
	\newcommand{\diam}{\text{diam}}
	\newcommand{\R}{\mathbb{R}}
	\newcommand{\N}{\mathbb{N}}
	\newcommand{\oa}{\overline{\Omega}}
	\newcommand{\ot}{\Omega\times (0,T)}
	\newcommand{\ota}{\oa\times [0,T]}
	\newcommand{\otm}{\Omega\times (0,T_{max})}
	\newcommand{\otma}{\oa\times [0,T_{max}]}
	\newcommand{\Tmax}{T_{max,\varepsilon}}
	\newcommand{\bu}{\overline{u}}
	\newcommand{\bh}{\overline{h}}
	\newcommand{\bw}{\overline{w}}
	\newcommand{\ct}{C^{\vartheta, \frac{\vartheta}{2}}(\oa \times [0,T])}
	\newcommand{\cet}{C^{1+\vartheta, \frac{1+\vartheta}{2}}(\oa \times [0,T])}
	\newcommand{\czt}{C^{2+\vartheta, 1+\frac{\vartheta}{2}}(\oa \times [0,T])}
	\newcommand{\ckt}{C^{\kappa, \frac{\kappa}{2}}(\oa \times [0,T])}
	\newcommand{\czkt}{C^{2+\kappa, 1+\frac{\kappa}{2}}(\oa \times [0,T])}
	\newcommand{\ue}{u_{\varepsilon}}
	\newcommand{\we}{w_{\varepsilon}}
	\newcommand{\he}{h_{\varepsilon}}
	\newcommand{\uek}{u_{\varepsilon_k}}
	\newcommand{\wek}{w_{\varepsilon_k}}
	\newcommand{\hek}{h_{\varepsilon_k}}
	\newcommand{\veps}{v_{\varepsilon}}
	\newcommand{\uea}{\ue^{\alpha}}
	\newcommand{\ueb}{\ue^{\beta}}
	\newcommand{\uqa}{\ue^{q+\alpha-1}}
	\newcommand{\weg}{\we^{\gamma}}
	\newcommand{\uoe}{u_{0\varepsilon}}
	\newcommand{\woe}{w_{0\varepsilon}}
	\newcommand{\hoe}{h_{0\varepsilon}}
	\newcommand{\He}{H_{\varepsilon}}
	\newcommand{\qk}{q_k}
	\newcommand{\qmm}{q_{m-1}}
	
	\newcommand{\cb}{\color{blue}}
	\newcommand{\cmg}{\color{magenta}}
	
	\newtheorem{Theorem}{Theorem}[section]
	\newtheorem{Assumptions}[Theorem]{Assumptions}
	\newtheorem{Corollary}[Theorem]{Corollary}
	\newtheorem{Convention}[Theorem]{Convention}
	\newtheorem{Definition}[Theorem]{Definition}
	
	\newtheorem{Lemma}[Theorem]{Lemma}
	\newtheorem{Notation}[Theorem]{Notation}
	\newtheorem{Remark}[Theorem]{Remark}
	\theoremstyle{definition}
	\newtheorem{Example}[Theorem]{Example}
	
	\theoremstyle{definition}
	\newtheorem{proofpart}{Step}
	\makeatletter
	\@addtoreset{proofpart}{Theorem}
	\makeatother
	\numberwithin{equation}{section}
	\title{On a PDE-ODE-PDE model for two interacting cell populations under the influence of an acidic environment and with nonlocal intra- and interspecific growth limitation}
	\author{Maria Eckardt\footnote{eckardt@mathematik.uni-kl.de}\quad  and\quad  Christina Surulescu\footnote{surulescu@mathematik.uni-kl.de}\\
		{\small RPTU Kaiserslautern-Landau, Felix-Klein-Zentrum f\"{u}r Mathematik,} \\
		{\small Paul-Ehrlich-Str. 31, 67663 Kaiserslautern, Germany}}
	
	\date{\today}
	\maketitle
	
	\begin{abstract}
		We consider a model for the dynamics of active cells interacting with their quiescent counterparts under the influence of acidity characterized by proton concentration. The active cells perform nonlinear diffusion and infer proliferation or decay, according to the strength of spatially nonlocal intra- and interspecific interactions. The two cell phenotypes are interchangeable, the transitions depending on the environmental acidity. We prove global existence of a weak solution to the considered PDE-ODE-PDE system and perform numerical simulations in 1D to informally investigate boundedness and patterning behavior in dependence of the system's parameters and kernels involved in the nonlocal terms.
	\end{abstract}

	\section{Introduction}
	
	Tumor heterogeneity is a well established fact \cite{Hanahan2011}. The neoplastic tissue is -among others- composed of several cell phenotypes, all of which are related to the stage within the cell cycle. To simply, of this vast variety we only consider here two phenotypes: active and quiescent cells. The former are supposed to be motile and proliferate, while the latter just infer transitions toward or from activity. While competing with their active counterparts, quiescent cells can also be degraded. Furthermore, the advancement through the cell cycle and the corresponding phenotypic switch is influenced, inter alia, by biochemical factors in the peritumoral space, see \cite{Hanahan2011} and references therein. In particular, pH regulation is a key feature in tumor cell cycle progression, which it can delay or even inhibit \cite{Boedtkjer,Flinck2018,Flinck2018a,putney2003h}. \\[-2ex]
	
	\noindent
	The interactions of cells with their environment occur not only locally, but cells can percieve their surroundings in a far more extensive manner, by way of  protrusions like cytonemes/filopodia/invadopodia, tunneling nanotubes etc.  \cite{CasasTinto2019,SenzdeSantaMara2017,Valdebenito2021,Nazari2023}. This motivated the introduction of mathematical models for cell migration, proliferation, and spread. Most of them are of the reaction-diffusion-transport type, with spatial nonlocalities occuring in the advection terms, mainly to model cell-cell and/or cell-tissue adhesions, or nonlocal taxis see e.g. \cite{Armstrong2006,Hillen2007,Eckardt2020,Othmer2002,Domschke2014,ZRModelling,Carrillo2019,EckardtZhigun24}, or in the source terms, to describe intraspecific interactions over a whole sensing range  \cite{LiChSu,EckardtSu}. We refer to e.g.,  \cite{AndreguettoMaciel2021,Simoy2023,Ni2018} for settings also involving nonlocal interspecific competition in different, but related contexts, where the focus is on global stability and pattern issues. The work \cite{SZYMASKA2009} also considered spatially nonlocal interspecific interactions, but of cancer cells with extracellular matrix and both species featured such terms. For a recent review on nonlocal models for cell migration see \cite{CPSZ}; for more comprehensive reviews of nonlocal models in a  broader context refer to \cite{Eftimie2018,Kavallaris2018}. \\[-2ex]
	
	\noindent
	Nonlocal models can be obtained, thus far still in a non-rigorous manner, from space- or velocity-jump descriptions on the mesoscopic level (also including the kinetic theory of active particles framework \cite{Bellomo2017}), possibly also accounting for microscale dynamics like binding of transmembrane units to soluble or unsoluble ligands. We refer to \cite{ZRModelling,LiChSu,EckardtSu,Domschke2017,Buttenschn2017} for such deductions.\\[-2ex]
	
	\noindent
	The remainder of this paper is structured as follows: in Section \ref{sec:model} we present the model consisting of a PDE-ODE-PDE system, along with requirements for the involved parameters and functions. Sections \ref{sec:analysis1} and \ref{sec:analysis2} are dedicated to proving global existence of a nonnegative weak solution to the system, in the sense specified therein. In Section \ref{sec:simulations} we perform numerical simulations in 1D within various scenarios, to get some insight into boundedness and patterning behavior under the influence of different choices of relevant parameters, interaction kernels, and phenotypic switch triggered by acidity. Finally, Section \ref{sec:discussion} provides some concluding remarks and an outlook.  The appendix Sections \ref{sec:appA} and \ref{sec:appB} contain several auxiliary results needed for the proofs in Section \ref{sec:analysis1}.
	
	\section{Model}\label{sec:model}
	
	In the following $u$ and $w$ represent the densities of active and of quiescent cells, respectively, whereas $h$ is the concentration of protons in the extracellular space. By 'active' we mean here cells which are migrating and proliferating. On the other hand, 'quiescent' means cells which only interact with their active counterparts and with the environment, without moving nor being able to proliferate. We consider the IBVP within a bounded domain $\Omega \subset \R^d$ having a sufficiently regular boundary $\partial \Omega $ and with no-flux boundary conditions
	\begin{align}\label{IBVP}
		\begin{cases}
			\partial_t u = \nabla \cdot \left( \psi(w,h)\nabla u\right) + \mu_1 \ua \left(1- J_1(x,h)\ast \ub - J_2(x,h) \ast \wg \right) + \mut(h) F(w), & x\in \Omega, t>0,\\
			\partial_t w = \mu_2(h)(1-w) u - \mu_3(h) {F(w)}, & x \in \Omega, t>0,\\
			\partial_t h = D_H \Delta h + g(u,w) - \lambda h, & x \in \Omega, t>0\\
			\partial_{\nu} u = \partial_{\nu} h = 0, & x \in \partial \Omega, t>0\\
			u(x,0) = u_0(x),\, w(x,0) = w_0(x) ,\, h(x,0) = h_0(x), & x\in \Omega.
		\end{cases}
	\end{align}
	The first term on the right hand side of the first PDE in \eqref{IBVP} describes nonlinear diffusion of active cells. The diffusion coefficient $\psi $ can thereby depend on $w$ and $h$: a large amount of $w$-cells can increase the population pressure, thus leading to faster diffusion; too many quiescent cells would, however, impede migration  (e.g., due to lack of space). Large $h$-values are also supposed to enhance motility, as the active cells tend to leave such areas faster than more favorable places. The next term describes proliferation of active cells, which is limited by spatially nonlocal intra- and interspecific interactions. As in previous works \cite{LiChSu,EckardtSu} we consider the exponents $\alpha, \beta, \gamma$ in the weak Allee effect and the competition/crowding terms with the interaction kernels $J_1$ and $J_2$. The latter can be seen as weighting the influence of either interactions on the dynamics of $u$ over a whole region. This description enables a more flexible characterization of the interaction strengths and is related to the size of $u$- and $w-$cell clusters exchanging information with (bunches of) active cells. Eventually, the last term describes phenotypic switch from quiescent to active cells; this transition is happening with a {certain saturation} and its rate $\tilde \mu_3$ depends on the concentration $h$ of protons. Indeed, less acidic environments favor exit from the quiescent phase and advancement towards activity \cite{Taylor1984,Butturini2019}. \\[-2ex]
	
	\noindent
	The second equation in \eqref{IBVP} is an ODE describing the dynamics of quiescent cells. These are supposed to be non-motile, to be produced by active cells with a rate $\mu_2$ which depends on the acidity in the peritumoral space, and to infer a transition to activity, again with an acidity-dependent rate $\mu_3$, which might differ from $\tilde \mu_3$. We also include a kind of acidity-triggered competition between active and quiescent cells; it might have an own rate, but to keep the number of model coefficients as low as possible we take it to be $\mu_2(h)$, too.\\[-2ex]
	
	\noindent
	The third equation in system  \eqref{IBVP} is again a reaction-diffusion PDE and models the dynamics of proton concentration $h$. Protons are very small in comparison with cells and accordingly able to diffuse quite fastly. They are produced by both tumor cell phenotypes (primarily by active cells and to a lesser amount by quiescent ones) and infer natural decay (e.g., by proton buffering). \\[-2ex]
	
	\noindent Concrete choices of motility, transition, and proliferation coefficients will be provided in Section \ref{sec:simulations}.
	We define the convolution over $\Omega$ for an $h$-dependent kernel as
	\begin{align}
		\left (J(\cdot ,h) \ast u\right )(x) =  \io J(x-y,h(y))u(y) \td y. \label{deffalt}
	\end{align}

\noindent
This setting extends our macroscopic model from \cite{EckardtSu} in the sense that we consider here two interacting populations, the dynamics of both being influenced by that of the acidity in their surroundings. Instead of the tumor diffusion tensor depending only on space we have here a dependency on two of the solution components, however we do not include any repellent pH-taxis, but focus instead on the nonlocal interactions and on the phenotypic switch. It also extends the model in \cite{SZYMASKA2009}, where the two interacting species are not influenced by a third one, the diffusion of cells is of the linear type, and there are no transitions from one species to the other, although all interactions therein are nonlocal in space.  \\[-2ex]

\noindent
The model can be obtained in a way similar to the meso-to-macro deduction performed in \cite{EckardtSu}, if the dynamics of $w$ and $h$ is given as in the second and third equations of \eqref{IBVP}, respectively. Although it is not clear how to obtain nonlinear diffusion in general, this can be achieved if the diffusion coefficient is only depending on macroscopic quantities other than $u$. If only linear diffusion is considered, then the method provides a space-dependent (myopic) diffusion tensor of $u$-cells, which by an adequate choice of the cell velocity distribution leads to classical Fickian diffusion. \\[-2ex]



\noindent
We make the following assumptions on the involved parameters and functions:
\begin{itemize}
	\item $\alpha,\beta,\gamma \geq 1$ satisfy
	\begin{align} \label{bedalphbet}
		\alpha <
		\begin{cases}
			1+\beta, & d=1,2\\
			1+\frac{2\beta}{d}, & d>2,
		\end{cases}
	\end{align}
	\item $\mu_1,D_H, \lambda>0$,
	\item
	\begin{subequations}
		\begin{align}
			&\psi(w,h) \geq \delta >0 \text{ for } h,w \geq 0, \label{psilb}\\
			&{ \psi \in C^1(\R_0^+\times\R_0^+)}, \label{psilip}
		\end{align}
	\end{subequations}
	\item $\mu_2,\mu_3 \in C^1(\R_0^+)$ with $\partial_h \mu_2, \partial_h \mu_3 \in L^{\infty}(\R_0^+)$, $\mut$ Lipschitz with Lipschitz constant $L_{\mut} \geq 0$, $\mu_2,\mu_3,\mut \geq 0$,
	\item { $F(w) = w$ or $F(w) = \frac{w}{1+w}$ and set $\tilde{F}(w): = 1$ if $F(w) = w$ and $\tilde{F}(w) := \frac{1}{1+w}$ if $F(w) = \frac{w}{1+w}$,}
	\item $g$ is Lipschitz with constant $L_g > 0$ and satisfies
	\begin{subequations}
		\begin{align}
			0 \leq g(u,w) \leq G  \label{boundg}
		\end{align}
	\end{subequations}
	for $G \in(0,\infty)$ s.t. $\frac{G}{\lambda} \leq 1$,
	\item for $i = 1,2$ and $B:= B_{\diam(\Omega)}(0)$ it holds that
	\begin{subequations}\label{cond-kernels}
		\begin{align}
			&J_i(x,\cdot) \text{ is Lipschitz continuous on }  \R_0^+ \text{ with constant } L_{J_i}(x)\geq 0, \label{estj}\\
			&L_{J_i}, \, J_i(\cdot,0) \in L^{p_i}(B) \text{ for some } p_i \in (1,\infty), \label{lpji}\\
			&J_2 \geq 0, J_1 \geq \eta > 0 \text{ for } 0 \leq h\leq 1, \label{lbj1}
		\end{align}
	\end{subequations}
	\item  $u_0 \in C(\oa)$, $w_0,h_0 \in H^1(\Omega)$ and $0\leq u_0,h_0,w_0 \leq 1$.
\end{itemize}
\noindent
Furthermore, $C_i$, $i\in\N$, denotes by convention a positive constant or a positive function of its arguments.\\[-2ex]

\noindent
These assumptions are primarily made out of technical reasons, in order to support the analysis in Subsections \ref{sec:analysis1} and \ref{sec:analysis2}, however most of them are reasonable from the application viewpoint: all parameters should be nonnegative and the interactions should involve at least one cell on either side; the diffusion of active and motile cells should be nondegenerate; there should be an effective, but uniformly limited production of protons, which should not dominate the natural decay in a too substantial manner; the interaction kernels should be nonnegative and there should be genuine intraspecific interactions, while the proton concentration remains reasonably bounded, and the initial conditions should be nonnegative and uniformly bounded.

\section{Global existence of a classical solution to an approximate problem} \label{sec:analysis1}
Let $\vartheta \in (0,1)$. There are sequences of initial values $\left(\uoe\right)_{\varepsilon\in(0,1)}$, $\left(\woe\right)_{\varepsilon\in(0,1)}$, $\left(\hoe\right)_{\varepsilon\in(0,1)} \subset C^{2+\vartheta}(\oa)$ s.t.
\begin{subequations}
	\begin{align}
		&0 \leq \uoe \leq 1,\label{bounduoe}\\
		&0 \leq \woe \leq 1,\label{boundwoe}\\ 
		&0 \leq \hoe \leq 1, \label{boundhoe}\\
		&\partial_{\nu} \uoe = \partial_{\nu} \woe = \partial_{\nu} \hoe = 0 \text{ on } \partial \Omega,\nonumber\\
		&\uoe \underset{\varepsilon\to 0}{\rightarrow} u_0 \text{ in } C(\oa),\label{convuoe}\\
		&\woe \underset{\varepsilon\to 0}{\rightarrow} w_0 \text{ in } H^1(\Omega) ,\label{convwoe}\\
		&\hoe \underset{\varepsilon\to 0}{\rightarrow} h_0 \text{ in } H^1(\Omega).\label{convhoe}
	\end{align}
\end{subequations}
Throughout this chapter we consider for $\varepsilon\in(0,1)$ the approximate IBVP
\begin{align}\label{IBVPwid}
	\begin{cases}
		\partial_t \ue = \nabla \cdot \left( \psi(\we,\he)\nabla \ue\right) + \mu_1 \uea \left(1- J_1(x,\he)\ast \ueb - J_2(x,\he) \ast \weg \right) + \mut(\he) F(\we), & x\in \Omega, t>0,\\
		\partial_t \we = \varepsilon \Delta \we + \mu_2(\he) (1-\we) \ue - \mu_3(\he) F(\we), & x \in \Omega, t>0,\\
		\partial_t \he = D_H \Delta \he + g(\ue,\we) - \lambda \he, & x \in \Omega, t>0,\\
		\partial_{\nu} \ue = \partial_{\nu} \we = \partial_{\nu} \he = 0, & x \in \partial \Omega, t>0\\
		\ue(x,0) = \uoe(x),\, \we(x,0) = \woe(x) ,\, \he(x,0) = \hoe(x), & x\in \Omega.
	\end{cases}
\end{align}

\noindent We show local existence of a solution with a  fixed-point argument.
\begin{Lemma} \label{lemlocex}
	For all $\varepsilon\in(0,1)$ there is $\Tmax\in (0,\infty]$ and a solution $(\ue,\we,\he) \in \left(C^{2,1}(\oa \times [0,\Tmax))\right)^3$ of \eqref{IBVPwid} with $0\leq \ue$ and $0 \leq \we, \he \leq 1$ s.t. either $\Tmax = \infty$ or $\Tmax<\infty$ and 
	\begin{align} \label{limczt}
		\lim\limits_{t\nearrow \Tmax} \left( \| \ue(\cdot, t)\|_{C^{2+\vartheta}(\oa)} + \|\we(\cdot, t)\|_{C^{2+\vartheta}(\oa)} + \| \he(\cdot, t)\|_{C^{2+\vartheta}(\oa)}\right) = \infty.
	\end{align}
\end{Lemma}

\begin{proof}
	Let $\varepsilon\in (0,1)$ and $T\in (0,1)$ small enough. For $h<0$ we set $\mu_2(h) := \mu_2(-h),\  \mu_3(h) := \mu_3(-h),\ \mut(h) := \mut(-h)$. We will perform a fixed point argument in
	\begin{align*}
		S := &\left\{(\bu,\bw) \in \left(\ct \right)^2 \ :\ \bu,\bw\geq 0,\ \|\bu\|_{\ct} + \|\bw\|_{\ct} \leq M+1 \right\}
	\end{align*}
	for $M := \|\uoe\|_{C^{\vartheta}(\oa)} + \|\woe\|_{C^{\vartheta}(\oa)} + 1$. For $(\bu,\bw) \in S$, we consider the three decoupled IBVPs
	\begin{align} \label{IBVPu}
		\begin{cases}
			\partial_t \ue = \nabla \cdot \left( \psi(\we,\he)\nabla \ue\right) - \mu_1 \bu^{\alpha-1} \left( J_1(x,\he)\ast \bu^{\beta} + J_2(x,\he) \ast \weg \right) \ue \\
			\phantom{\partial_t \ue =}+ \mu_1 \bu^{\alpha} + \mut(\he) \tilde{F}(\bw)\we, & x\in \Omega,\ t>0,\\
			\partial_{\nu} \ue = 0, & x \in \partial \Omega,\ t>0\\
			\ue(x,0) = \uoe(x), & x\in \Omega,
		\end{cases}
	\end{align}
	\begin{align} \label{IBVPw}
		\begin{cases}
			\partial_t \we = \varepsilon \Delta \we + \mu_2(\he)(1 - \we) \bu - \mu_3(\he) \tilde{F}(\bw)\we, & x \in \Omega,\ t>0,\\
			\we(x,0) = \woe(x) , & x\in \Omega,
		\end{cases}
	\end{align}
	and
	\begin{align} \label{IBVPh}
		\begin{cases}
			\partial_t \he = D_H \Delta \he + g(\bu,\bw) - \lambda\he & x \in \Omega,\  t>0,\\
			\partial_{\nu} \he = 0, & x \in \partial \Omega,\ t>0\\
			\he(x,0) = \hoe(x), & x\in \Omega.
		\end{cases}
	\end{align} 	
	We start with \eqref{IBVPh}. Due to the H\"older continuity of $\bu$ and $\bw$ and the Lipschitz continuity of $g$ we can apply Theorem IV.5.3 from \cite{Lady} with the coefficients
	\begin{align*}
		a_{ii}(x,t) := D_H,\ 
		a_i (x,t) := 0,\ 
		a(x,t) := \lambda,\ 
		b_i(x,t) := \nu_i,\ 
		b(x,t) := 0,\ 
		f(x,t) := g(\bu,\bw)
	\end{align*}
	for $i \in \{1,\dots,d\}$ to \eqref{IBVPh} and obtain a unique solution $\he \in \czt$ s.t.
	\begin{align*} 
		\|\he\|_{\czt} 
		\leq \Cl{CLady53} \left(\|g(\bu,\bw)\|_{\ct} + \|\hoe\|_{C^{2+\vartheta}(\oa)}\right) 
		\leq \Cl{boundhL53}\left(M,\|\hoe\|_{C^{2+\vartheta}(\oa)}\right).
	\end{align*}
	Moreover, due to the Lipschitz continuity of $\mu_2,\ \mu_3$ on $[0,\|\he\|_{L^{\infty}(\Omega\times (0,T))}]$ and the H\"older continuity of $\bu, \bw$ and Lemma \ref{lemholprod} we conclude again from Theorem IV.5.3 in \cite{Lady}  with the coefficients
	\begin{align*}
		a_{ii}(x,t) := \varepsilon,\ 
		a_i (x,t) := 0,\ 
		a(x,t) := - \mu_2(h)\bu + \mu_3(\he)\tilde{F}(\bw)\,
		b_i(x,t) := \nu_i,\ 
		b(x,t) := 0,\ 
		f(x,t) := \mu_2(\he)\bu
	\end{align*} 
	for $i \in \{1,\dots,d\}$ that there is a unique solution $\we \in \czt$ to \eqref{IBVPw} satisfying
	\begin{align} \label{boundwczt}
		\|\we\|_{\czt} &\leq \Cl{CLady53w} \left(\|\mu_2(\he)\bu\|_{\ct} + \|\woe\|_{C^{2+\vartheta}(\oa)}\right) \\
		&\leq \Cl{boundwL53}\left(M,\|\woe\|_{C^{2+\vartheta}(\oa)}, \|\hoe\|_{C^{2+\vartheta}(\oa)}\right).
	\end{align}
	Then, we can estimate that
	\begin{align*}
		-(\we)_t + \varepsilon \Delta \we - \left(\mu_2(\he)\bu + \mu_3(\he)\tilde{F}(\bw)\right) \we = -\mu_2(\he)\bar{u} \leq 0
	\end{align*}
	and 
	\begin{align*}
		- G \leq -(\he)_t + D_H \Delta \he - \lambda \he = - g(\bar{u},\bar{w}) \leq 0,
	\end{align*}
	due to \eqref{boundg} for $\He := \max\{\|\hoe\|_{L^{\infty}(\Omega)}, \frac{G}{\lambda}\} \leq 1$. Hence, from a parabolic comparison principle (see e.g. Theorem 2.9 in \cite{Lieberman1996}) it follows that $0 \leq \we$ and due to \eqref{boundg} that $0 \leq \he \leq \He \leq 1$. Further, we set $\veps:=1-\we$ and estimate
	\begin{align*}
		- (\veps)_t + \varepsilon \Delta \veps - \mu_2(\he)\bu \veps = - \mu_3(\he) \tilde{F}(\bw)\we \leq 0
	\end{align*} 
	and combining this with \eqref{boundwoe} we conclude that $\veps \geq 0$ and consequently, $\we \leq 1$. 
	Now, we set
	\begin{align*}
		a_{ii}(x,t) &:= \psi(\we,\he),\\
		a_i (x,t) &:= - \partial_h \psi(\we,\he) \nabla \he - \partial_w \psi(\we,\he) \nabla \we,\\
		a(x,t) &:= \mu_1 \bu^{\alpha-1} \left(J_1(x,\he)\ast \bu^{\beta}(x,t) + J_2(x,\he) \ast \weg \right),\\
		b_i(x,t) &:= \nu_i,\\
		b(x,t) &:= 0,\\
		f(x,t) &:= \mu_1 \bu^{\alpha} + \mut(\he) \tilde{F}(\bw)\we
	\end{align*}
	for $i \in \{1,\dots,d\}$ to apply again Theorem IV.5.3 in \cite{Lady} to the equation corresponding to \eqref{IBVPu} in nondivergence form. From this theorem, due to \eqref{psilip}, the H\"older continuity of $\he$, $\we$ and its gradients, the Lipschitz continuity of $\mut$, and Lemmas \ref{lemholprod} and \ref{lemholj}, it follows that there is a unique solution $\ue\in C^{2 + \kappa, 1 
		+ \frac{\kappa}{2}} (\ota)$ to \eqref{IBVPu} for $\kappa := \min\{1,\alpha-1\} \vartheta$, due to the possibility that $\alpha \in (1,2)$, and satisfying
	\begin{align} \label{bounduczt}
		\|\ue\|_{\czkt} 
		\leq& \Cl{CLady532}\left(\|a_{ii}\|_{\ckt},\|a_{i}\|_{\ckt},\|a\|_{\ckt},\|b_{i}\|_{C^{1+\kappa,\frac{1+\kappa}{2}}}\right)\\
		&\cdot\left(\|f\|_{\ckt} + \|u_0\|_{C^{2+\kappa}(\oa)} \right)\nonumber\\ 
		\leq& \Cl{boundu53}\left(M, \|\uoe\|_{C^{2+\vartheta}(\oa)}, \|\woe\|_{C^{2+\vartheta}(\oa)}, \|\hoe\|_{C^{2+\vartheta}(\oa)} \right),
	\end{align}
	due to the embedding of H\"older spaces. Further, we estimate 
	\begin{align*}
		-(\ue)_t + \nabla \cdot \left( \psi(\we,\he)\nabla \ue\right) - \mu_1 \bu^{\alpha-1} \left( J_1(x,\he)\ast \bu^{\beta} + J_2(x,\he) \ast \weg \right) \ue =
		- \mu_1 \bu^{\alpha} - \mut(\he) \tilde{F}(\bw)\we\leq 0
	\end{align*}
	and conclude from the comparison principle in Theorem 2.9 in \cite{Lieberman1996} that $\ue\geq 0$. Now, we estimate with \eqref{boundwczt} and \eqref{bounduczt} and Lemma \ref{lemholuu0}:
	\begin{align*}
		&\|\ue\|_{\ct} + \|\we\|_{\ct} \\
		\leq& \|\ue - \uoe\|_{\ct} + \|\we-\woe\|_{\ct} + \|\uoe\|_{C^{\vartheta}(\oa)} + \|\woe\|_{C^{\vartheta}(\oa)}\\
		\leq&  \Cl{einbsob} T^{\frac{1}{2}\min\left\{1+\kappa,1-\vartheta\right\}}\|\ue\|_{\czkt} + \Cr{einbsob} T^{\frac{1}{2}\min\left\{\vartheta,1-\vartheta\right\}} \|\we\|_{\czt}\\ 
		&+ \|\uoe\|_{C^{\vartheta}(\oa)} + \|\hoe\|_{C^{\vartheta}(\oa)}\\
		\leq& \Cr{einbsob} T^{\frac{1}{2}\left\{\vartheta,1-\vartheta\right\}}(\Cr{boundu53} + \Cr{boundwL53}) + \|\uoe\|_{C^{\vartheta}(\oa)} + \|\woe\|_{C^{\vartheta}(\oa)} \leq M+1
	\end{align*}
	for $T<(1/(\Cr{einbsob}(\Cr{boundu53} + \Cr{boundhL53}) ))^{\frac{2}{\min\left\{\vartheta,1-\vartheta\right\}}}$. \\[-2ex]
	
	\noindent
	Hence, { $(\ue,\we)\in S$} and the operator
	\begin{align*}
		K:S \rightarrow S,\quad 
		(\bu,\bw) \mapsto (\ue,\we)
	\end{align*}
	is well-defined. Due to { the} continuous dependence of the solution on the coefficients { and to} Theorem IV.5.3 { from \cite{Lady},} the operator $K$ is continuous. Moreover, \eqref{boundwczt} and \eqref{bounduczt} imply that $K$ maps bounded sets of $\left(\ct\right)^2$ on bounded sets of $\left(\czkt\right)^2$. Hence, from the compact embedding $\czkt \subset\subset \ct$ we conclude that $K$ is a compact operator. Schauder's fixed point theorem implies that $F$ has a fixed point $(\ue,\we) \in (\ct)^2$, where as well $\ue \in \czkt$ and $\we \in \czt$, as was shown above.\\[-2ex]
	
	\noindent
	 Applying { again} Theorem IV.5.3 from \cite{Lady} to $\ue$, { this time} with
	\begin{align*}
		a(x,t) &:= 0,\\
		f(x,t) &:= \mu_1 \ue^{\alpha} \left( 1- J_1(x,\he)\ast \ue^{\beta}(x,t) - J_2(x,\he) \ast \weg \right) + \mut(\he)\tilde{F}(\we)\we,
	\end{align*}
	we conclude that also $\ue \in \czt$.
	{ By } extending the solution to its maximal existence time $\Tmax$, \eqref{limczt} follows.
\end{proof}

\noindent We show the global boundedness of our solution by adapting the estimates from Step 1 and 2 of the proof of Theorem 1 in \cite{LiChSu} similar{ ly} to \cite{EckardtSu}.
\begin{Lemma} \label{lemglobbound}
	There is $\Cl{boundu} >0$ s.t. $\|\ue\|_{L^{\infty}(\Omega\times (0,\Tmax))} \leq \Cr{boundu}$ for all $\varepsilon \in (0,1)$.
\end{Lemma}
\begin{proof}
	\begin{proofpart}
		Let $\varepsilon \in (0,1)$ and $q > \max\{1, \beta + \alpha -1\}$. Consider $t\in(0,\Tmax)$. We multiply the first equation of \eqref{IBVPwid} by $qu^{q-1}$ and integrate over $\Omega$ to obtain
		\begin{align*}
			\frac{d}{dt} \io \ue^q \td x
			=& - q(q-1) \io \psi(\we,\he) \left|\nabla \ue\right|^2 u^{q-2} \td x \\
			&+ q \mu_1 \io \uqa \left(1- J_1(x,\he)\ast \ueb - J_2(x,\he) \ast \weg \right) \td x+ \io \mut(\he)F(\we)\ue^{q-1} \td x,
		\end{align*}
		{ upon} using partial integration. Hence, we conclude from \eqref{psilb}, \eqref{lbj1}, the Lipschitz continuity of $\mut$, the boundedness of $h$, and Young's inequality that
		\begin{align*}
			&\frac{d}{dt} \io \ue^q \td x + \frac{4(q-1)}{q} \delta \io \left|\nabla u^{\frac{q}{2}}\right|^2 \td x + q \mu_1 \eta \io \ueb \td x \io \uqa \td x\\
			\leq& q \mu_1 \io \uqa \td x + \|\mut\|_{L^{\infty}(0,1)} \io \ue^{q-1} \td x.
		\end{align*}
		Setting $\Cl{c1} := \mu_1 + \|\mut\|_{L^{\infty}(0,1)}$ and adding $q\Cr{c1}\|\ue\|_{L^q(\Omega)}^q$ on both sides of the above equation, due to \eqref{boundhoe}, the fact that $F(\we) \leq 1$, and using Young's inequality, we arrive at
		\begin{align}\label{estglobex1}
			&\frac{d}{dt} \io \ue^q \td x + \frac{4(q-1)}{q} \delta\io  \left|\nabla u^{\frac{q}{2}}\right|^2 \td x + q \mu_1 \eta \io \ueb \td x \io \uqa \td x + q \Cr{c1} \io \ue^q \td x \nonumber\\
			\leq& 2q \Cr{c1} \left( \io \uqa \td x + |\Omega|\right).
		\end{align}
		From Lemma \ref{lemLiSuChThm1} it follows for $K_1 = \frac{\Cr{c1}}{\delta}$ and $K_2 = \frac{2\Cr{c1}}{\mu_1\eta}$ that
		\begin{align} \label{estglobex2}
			\io \ue^{q+\alpha-1} \td x 
			\leq \frac{2(q-1)}{q^2\Cr{c1}}\delta\io |\nabla \ue^{\frac{q}{2}}|^2 \td x 
			+ \frac{\mu_1\eta}{2\Cr{c1}} \io \ueb \td x \io \ue^{q+\alpha-1} \td x + \Cl{constLSC2}\left(q\right),
		\end{align}
		where
		\begin{align*}
			s = \begin{cases}
				\infty, &d=1\\
				\left(\frac{2(q+\alpha-1+\beta)}{q-\alpha+1+\beta}, \infty\right), &d=2\\
				\frac{2d}{d-2}, &d>2,
			\end{cases}
		\end{align*}
		\begin{align*}
			\Cr{constLSC2}(q) :=& \left(2\left(\frac{\Cl{constsp}^2q^2\Cr{c1}}{(q-1)\delta}\right)^{\frac{q+\alpha-1-\beta}{q-\alpha+1+\beta-2\frac{q+\alpha-1+\beta}{s}}} + \Cl{constpp2}(q)^{\frac{q+\alpha-\beta-1}{q-\frac{q+\alpha-1+\beta}{s}}}\right)^{\frac{q-\alpha+1+\beta - \frac{2(q+\alpha-1+\beta)}{s}}{\beta+1-\alpha - \frac{2\beta}{s}}}\\
			&\cdot\left(\frac{2\Cr{c1}}{\mu_1\eta}\right)^{\frac{q-\frac{2(q+\alpha-1)}{s}}{\beta+1-\alpha-\frac{2\beta}{s}}}
			+ \Cr{constpp2}(q)^{\frac{q+\alpha-\beta-1}{q-\frac{q+\alpha-1+\beta}{s}}},
		\end{align*}
		and 
		\begin{align*}
			\Cr{constsp} &:= 2C_S\left(1+2C_P\right),\\
			\Cr{constpp2}(q) &:= 4C_S|\Omega|^{\frac{1}{2}-\frac{q}{q+\alpha-1+\beta}}.
		\end{align*}
		Here, $C_S$ denotes the Sobolev embedding constant from $H^1(\Omega)$ into $L^s(\Omega)$ and $C_P$ denotes the constant from the Poincar\'e inequality.
		Hence, inserting \eqref{estglobex2} into \eqref{estglobex1} we obtain
		\begin{align*}
			\frac{d}{dt} \io \ue^q \td x + q\Cr{c1} \io \ue^q \td x
			\leq 2 q\Cr{c1}\left( \Cr{constLSC2}\left(q\right) + |\Omega| \right).
		\end{align*}
		We conclude from Lemma \ref{lemodeest} with $K_1 = q \Cr{c1}$ and $K_2 = 2 \Cr{constLSC2}\left(q\right) + |\Omega| $, and due to \eqref{bounduoe} that
		\begin{align} \label{bounduk}
			\|\ue\|_{L^q(\Omega)} \leq \sqrt[q]{2 \Cr{constLSC2}\left(q\right) + |\Omega|  + \|\uoe\|_{L^q(\Omega)}^q} \leq \sqrt[q]{2 \Cr{constLSC2}\left(q\right) + |\Omega|\left(1  + \|\uoe\|_{L^{\infty}(\Omega)}^q\right)} \underset{q\to\infty}{\rightarrow} \infty,
		\end{align}
		due to
		\begin{align*}
			\left(\left(\frac{q^2}{q-1}\right)^{\frac{q+\alpha-1-\beta}{q-\alpha+1+\beta-2\frac{q+\alpha-1+\beta}{s}}\cdot \frac{q-\alpha+1+\beta - \frac{2(q+\alpha-1+\beta)}{s}}{\beta+1-\alpha - \frac{2\beta}{s}}}\right)^{\frac{1}{q}} 
			\geq q^{\frac{1}{q} \cdot \frac{q+\alpha-1-\beta}{\beta+1-\alpha - \frac{2\beta}{s}}} 
			= \left(q^{1 + \frac{\alpha-1-\beta}{q}}\right)^{\frac{1}{\beta+1-\alpha - \frac{2\beta}{s}}}
			\underset{q\to\infty}{\rightarrow} \infty.
		\end{align*}
		Due to \eqref{bounduoe} we can also find an upper bound, { which is} independent from $\varepsilon$, namely
		\begin{align} \label{boundueinde}
			\|\ue\|_{L^q(\Omega)} \leq \sqrt[q]{2 \Cr{constLSC2}\left(q\right) + 2|\Omega|}.
		\end{align}
	\end{proofpart}
	\begin{proofpart}
		We proceed with a Moser iteration. For $d=2$, we consider $s \in (\max\{\frac{2\beta}{\beta + 1- \alpha},3\}, \infty)$ (for $d\neq 2$ we consider $s$ as above) and set $h := \frac{2(s-1)(\alpha-1)}{s-2}$ and $q_k := 2^k + h$ for $k\in\N$ s.t. $k \geq \log_2(\alpha-1) -2$. Then, using Lemma \ref{lemLiSuChThm1} we conclude with $K_1 = \frac{\Cr{c1}}{\delta}$ and $r = r_k := \frac{2q_{k-1}}{q_k}$ that
		\begin{align}
			&\frac{d}{dt} \io \ue^{\qk} \td x + \frac{4(\qk-1)}{\qk} \delta\io  \left|\nabla u^{\frac{\qk}{2}}\right|^2 \td x + \qk \mu_1 \eta \io \ueb \td x \io \ue^{q_k+\alpha-1} \td x + \qk \Cr{c1} \io \ue^{\qk} \td x \nonumber\\
			\leq& 2 \qk \Cr{c1} \left(\io \ue^{q_k+\alpha-1} \td x + |\Omega|\right)\nonumber\\
			\leq& \frac{4(\qk-1)}{\qk^2} \delta \io |\nabla u^{\frac{\qk}{2}}|^2 \td x
			+ 2\qk \Cr{c1}\left(\left(2\Cl{constLSC12}(k) + \Cl{constpp4}(k)\right)\|u^{\frac{\qk}{2}}\|_{L^{r_k}(\Omega)}^{2r_k\frac{\frac{2(q_k+\alpha-1)}{s}-q_k}{2q_{k-1}\left(\frac{2}{s}-1\right)+2(\alpha-1)}}
			+ \Cr{constpp4}(k) +  |\Omega|\right), \label{estMos1}
		\end{align}
		where
		\begin{align*}
			\Cr{constLSC12}(k) &:= \left(\frac{\Cr{constsp}^2q_k^2\Cr{c1}}{(q_k-1)\delta}\right)^{\frac{2q_{k-1}-2(q_k+\alpha-1)}{2q_{k-1}\left(1-\frac{2}{s}\right) +2(\alpha-1)}} ,\\
			\Cr{constpp4}(k) &:= \Cl{constpp3}(k)^{\frac{2q_{k-1}-2(q_k+\alpha-1)}{\frac{2q_{k-1}}{s}-q_k}},\\
			\Cr{constpp3}(k) &:= 4C_S(s)|\Omega|^{\frac{q_{k-1}-q_k}{2q_{k-1}}},
		\end{align*}
		We know from Lemma \ref{lemLiSuChThm12} that 
		\begin{align*}
			\frac{\frac{2(q_k+\alpha-1)}{s}-q_k}{2q_{k-1}\left(\frac{2}{s}-1\right)+2(\alpha-1)} = 1
		\end{align*}
		and consequently,
		\begin{align*}
			\|u^{\frac{\qk}{2}}\|_{L^{r_k}(\Omega)}^{2r_k\frac{\frac{2(q_k+\alpha-1)}{s}-q_k}{2q_{k-1}\left(\frac{2}{s}-1\right)+2(\alpha-1)}} = \left(\io u^{q_{k-1}} \td x\right)^2.
		\end{align*}
		Inserting this into \eqref{estMos1} we conclude that
		\begin{align}
			&\frac{d}{dt} \io \ue^{\qk} \td x + \qk \Cr{c1} \io \ue^{\qk} \td x \nonumber \\
			\leq& 2\qk \Cr{c1} \left(2\Cr{constLSC12}(k) + 2 \Cr{constpp4}(k) + |\Omega|\right) \max\left\{1,\left(\io u^{q_{k-1}} \td x\right)^2\right\}. \label{estMos2}
		\end{align}
		Lemma \ref{lemLiSuChThm12} implies that
		\begin{align*}
			\frac{2q_{k-1}-2(q_k+\alpha-1)}{2q_{k-1}\left(1-\frac{2}{s}\right) +2(\alpha-1)} = \frac{s}{s-2} \text{ and }
			\frac{2q_{k-1}-2(q_k+\alpha-1)}{\frac{2q_{k-1}}{s}-1} \leq \alpha + 1.
		\end{align*} 
		Moreover, we can estimate
		\begin{align*}
			\frac{q_{k-1}-q_k}{2q_{k-1}} \in \left(-\frac{1}{2},0\right).
		\end{align*}
		Hence, we can estimate
		\begin{align*}
			2\Cr{constLSC12}(k) + 2 \Cr{constpp4}(k) + |\Omega|
			&= 2\left(\frac{\Cr{constsp}^2q_k^2\Cr{c1}}{(q_k-1)\delta}\right)^{\frac{s}{s-2}} +2\left(4C_S|\Omega|^{\frac{q_{k-1}-q_k}{2q_{k-1}}}\right)^{\frac{2q_{k-1}-2(q_k+\alpha-1)}{\frac{2q_{k-1}}{s}-1}} + |\Omega|\\
			&\leq 2\left(\frac{\Cr{constsp}^2\Cr{c1}}{\delta}q_k\right)^{\frac{s}{s-2}}
			+ 2\left(4\max\left\{1,C_S\right\}\max\left\{1,|\Omega|^{-\frac{1}{2}} \right\}\right)^{\alpha+1} +|\Omega|\\
			&\leq \bar{a}2^{\frac{s}{s-2}k},
		\end{align*}
		where
		\begin{align*}
			\bar{a}:= 2\left(\frac{\Cr{constsp}^2\Cr{c1}}{\delta}(1+h)\right)^{\frac{s}{s-2}}
			+ 2\left(4\max\left\{1,C_S\right\}\max\left\{1,|\Omega|^{-\frac{1}{2}} \right\}\right)^{\alpha+1} +|\Omega|.
		\end{align*}
		Inserting this into \eqref{estMos2} we obtain
		\begin{align*}
			\frac{d}{dt} \io \ue^{\qk} \td x + \qk \Cr{c1} \io \ue^{\qk} \td x
			\leq 2\qk \Cr{c1}  \bar{a} 2^{\frac{s}{s-2}k}\max\left\{1,\sup_{t\geq 0}\left(\io u^{q_{k-1}} \td x\right)^2\right\}.
		\end{align*}
		For $k \geq 1$ we can estimate that $\io \uoe^{\qk} \td x \leq \|\uoe\|_{L^{\infty}(\Omega)}^{q_k}|\Omega|\leq |\Omega|$ due to \eqref{bounduoe}.
		Hence, Lemma \ref{lemLSC21} implies that for $k\geq m\geq 1$ large enough, i.e. s.t. $\bar{a}2^{\frac{s}{s-2}m} \geq 1$., it holds that
		\begin{align*}
			\left(\io \ue^{\qk} \td x\right)^{\frac{1}{q_k}} \leq (4\bar{a})^{\frac{2^{k-m+1}-1}{q_k}} 2^{\frac{\frac{s}{s-2}\left(2(2^{k-m}-1) + m 2^{k-m+1} -k\right)}{q_k} } \max\left\{\sup_{t \geq 0} \left(\io \ue^{q_{m-1}} \td x\right)^{\frac{2^{k-m+1}}{q_k}}, |\Omega|^{\frac{2^{k-m}}{q_k}}, 1 \right\}
		\end{align*}
		For $k \to \infty$ we obtain
		\begin{align}
			\|\ue\|_{L^{\infty}(\Omega)} &\leq (4\bar{a})^{\frac{1}{2^{m-1}}} 2^{\frac{s(m+1)}{(s-2)2^{m-1}}}
			\max\left\{\sup_{t \geq 0} \left(\io \ue^{q_{m-1}} \td x\right)^{\frac{1}{2^{m-1}}}, |\Omega|^{2^{-m}}, 1 \right\} \label{boundum}
		\end{align}
		in $(0,\Tmax)$. We already know from Step 1 that the right-hand side is bounded above by a constant indepent from $\varepsilon$. Consequently, $\ue \in L^{\infty}(\Omega \times (0,\Tmax))$ is bounded above by a constant independent from $\varepsilon$.
	\end{proofpart}
\end{proof}

\noindent We can also perform a quasi-maximum principle as in Step 3 of the proof of Theorem 1 in \cite{LiChSu}.
\begin{Corollary} \label{corbb}
	We find $K>1$ and 'small' enough parameters s.t.
	\begin{align}
		\|\ue\|_{L^{\infty}(\Omega\times (0,\Tmax))} 
		\leq& K\max\left\{1,\left(4C_S |\Omega|^{-\frac{1}{2}}\right)^{\frac{1-\frac{2}{s}}{\left(1-\frac{1}{s}\right)\left(\beta+1-\alpha-\frac{2\beta}{s}\right)}} \left(\frac{2\Cr{c1}}{\mu_1\eta}\right)^{\frac{1-\frac{2}{s}}{\beta+1-\alpha-\frac{2\beta}{s}}}\right\}. \label{bounduek}
	\end{align}
\end{Corollary}
\begin{proof}
	We want to consider the limit $m \rightarrow \infty$ in \eqref{boundum} for 'small' enough parameters. We already know from \eqref{bounduk} in Step 1 that for $t \in (0,\Tmax)$ it holds that
	\begin{align*}
		\io \ue^{\qmm} \td x 
		\leq& 2 \Cr{constLSC2}\left(\qmm\right) + 2|\Omega|.
	\end{align*}
	We assume that
	\begin{align} \label{bedmu1}
		\Cr{c1} = \mu_1 + \|\mut\|_{L^{\infty}(0,1)} \leq \frac{\delta}{\Cr{constsp}^2\qmm^2}
	\end{align}
	holds. Then, we conclude
	\begin{align*}
		&\io \ue^{\qmm} \td x \\
		\leq& 6\max\left\{\left(2\left(\frac{1}{2^{m-1} +h -1}\right)^{\frac{\qmm+\alpha-1-\beta}{\qmm-\alpha+1+\beta-2\frac{\qmm+\alpha-1+\beta}{s}}} + \Cr{constpp2}(\qmm)^{\frac{\qmm+\alpha-\beta-1}{\qmm-\frac{\qmm+\alpha-1+\beta}{s}}}\right)^{\frac{\qmm-\alpha+1+\beta - \frac{2(\qmm+\alpha-1+\beta)}{s}}{\beta+1-\alpha - \frac{2\beta}{s}}}\right.\\ 
		&\phantom{8\max\{\}}\left.\cdot \left(\frac{2\Cr{c1}}{\mu_1\eta}\right)^{\frac{\qmm-\frac{2(\qmm+\alpha-1)}{s}}{\beta+1-\alpha-\frac{2\beta}{s}}},\Cr{constpp2}(\qmm)^{\frac{\qmm+\alpha-\beta-1}{\qmm-\frac{\qmm+\alpha-1+\beta}{s}}}, |\Omega|\right\} =: H(m).
	\end{align*}
	With
	\begin{align*}
		\lim\limits_{m\to\infty} \left(H(m)\right)^{\frac{1}{2^{m-1}}} = \max\left\{1, \left(4C_S |\Omega|^{-\frac{1}{2}}\right)^{\frac{1-\frac{2}{s}}{\left(1-\frac{1}{s}\right)\left(\beta+1-\alpha-\frac{2\beta}{s}\right)}} \left(\frac{2\Cr{c1}}{\mu_1\eta}\right)^{\frac{1-\frac{2}{s}}{\beta+1-\alpha-\frac{2\beta}{s}}}\right\}.
	\end{align*}
	and
	\begin{align*}
		\lim\limits_{m\to \infty} (4\bar{a})^{\frac{1}{2^{m-1}}} 2^{\frac{s(m+1)}{(s-2)2^{m-1}}} = 1
	\end{align*}
	and $|\Omega|^{2^{-m}}\underset{m\to\infty}{\rightarrow} 1$
	we conclude from \eqref{boundueinde} that 
	\begin{align*}
		\|\ue\|_{L^{\infty}(\Omega\times (0,\Tmax))} 
		\leq& \max\left\{1,\left(4C_S |\Omega|^{-\frac{1}{2}}\right)^{\frac{1-\frac{2}{s}}{\left(1-\frac{1}{s}\right)\left(\beta+1-\alpha-\frac{2\beta}{s}\right)}} \left(\frac{2\Cr{c1}}{\mu_1\eta}\right)^{\frac{1-\frac{2}{s}}{\beta+1-\alpha-\frac{2\beta}{s}}}\right\}.
	\end{align*}
	Obviously, we do not find parameters satisfying $\eqref{bedmu1}$ for $m$ tending to infinity. Nevertheless, for any $K>1$ we find an $m^*$ depending only on $K$ s.t. if \eqref{bedmu1} is satisfied for $m^*$ then, \eqref{bounduek} holds.
\end{proof}

\noindent
In the following remark we give an exact formula the Sobolev constant $C_S$ that only depends on the domain $\Omega$ and the dimension $d$, { in order to} get an impression of the upper bound of $\ue$.
\begin{Remark}
	If $\Omega$ is convex the upper bound from Lemma \ref{corbb} can be given in terms of $K$ and our parameters as due to Lemma \ref{lemMizuguchi} the Sobolev embedding constant $C_S(s)$ is given by
	\begin{align} \label{defcsex}
		C_S(s) = \begin{cases}
			\max\left\{1, \frac{\diam(\Omega)|V|}{|\Omega|}\right\} & d=1,\\
			\sqrt{2}\max\left\{|\Omega|^{\frac{1}{s}-\frac{1}{2}},\frac{\diam(\Omega)^{1+\frac{s+2}{s}}\pi^{\frac{s+2}{2s}}}{2|\Omega|}\frac{\Gamma\left(\frac{s-2}{2s}\right)}{\Gamma\left(\frac{s+2}{2s}\right)}\right\}\sqrt{\frac{\Gamma\left(\frac{2}{s}\right)}{\Gamma\left(2\frac{s-1}{s}\right)}}\left(\frac{\Gamma(2)}{\Gamma\left(1\right)}\right)^{\frac{s-2}{2s}} & d= 2,\\
			\sqrt{2}\max\left\{|\Omega|^{-\frac{1}{d}},\frac{\diam(\Omega)^{d}\pi^{\frac{d-1}{2}}}{d|\Omega|}\frac{\Gamma\left(\frac{1}{4}\right)}{\Gamma\left(\frac{d-1}{2}\right)}\right\}\sqrt{\frac{\Gamma\left(\frac{d-2}{2}\right)}{\Gamma\left(\frac{d+2}{2}\right)}}\left(\frac{\Gamma(d)}{\Gamma\left(\frac{d}{2}\right)}\right)^{\frac{1}{d}} & d\geq 3,
		\end{cases}
	\end{align}
	where $V := \bigcup_{x\in \Omega} \Omega_x$, where $\Omega_x := \left\{y-x : y\in\Omega\right\}$ for $x \in \Omega$, and $\Gamma$ denotes the Gamma Function given by $\Gamma(x) = \int_0^{\infty} t^{x-1} e^{-t} \td t$ for $x > 0$.
\end{Remark}

\noindent Global existence of our solution follows from the last lemma. 
\begin{Theorem} \label{globex}
	For $\varepsilon \in (0,1)$ there is a bounded solution $(\ue,\we,\he) \in \left(C^{2+\vartheta,1+\frac{\vartheta}{2}}\left(\oa \times [0,\infty)\right)\right)^3$ of \eqref{IBVPwid} with $0\leq \ue \leq \Cr{boundu}$ and $0 \leq \we,\he \leq 1$.
\end{Theorem}

\begin{proof}
	\noindent Let $\varepsilon \in (0,1)$. We have shown in Lemmas \ref{lemlocex} and \ref{lemglobbound} that $\ue,\we,\he\in L^{\infty}(\otm)$. 
	Assume $\Tmax<\infty$. Then, $g(\ue,\we) \in L^{\infty}(\otm)$ follows from the Lipschitz continuity of $g$. Putting together Theorem III.5.1. from \cite{Lady}, the boundedness of $\he$, and
	Theorem 4 from \cite{DiBenedetto} with $a(\nabla \he) := D_H \nabla \he$ and $b(x,t,\he) := \lambda \he - g(\ue, \we)$, we conclude that there is some  $\kappa_1 \in (0,1)$ s.t.
	\begin{align}\label{boundhct}
		\|\he\|_{C^{\kappa_1,\frac{\kappa_1}{2}}(\otm)} \leq \Cl{bck}.
	\end{align}
	Analogously, with $a(\nabla \we) := \varepsilon \nabla \we$ and $b(x,t,\we) := \mu_3(\he)F(\we) + \mu_2(\he)(\we - K)\ue$ we conclude from Theorem 4 in \cite{DiBenedetto}, { by} using the boundedness of our solution and the Lipschitz continuity of $\mu_2$ and $\mu_3$ on $[0,H]$, that there { is some} $\kappa_2 \in (0,1)$ s.t.
	\begin{align}\label{boundwct}
		\|\we\|_{C^{\kappa_2,\frac{\kappa_2}{2}}(\otm)} \leq \Cl{bck2}.
	\end{align}
	Moreover, due to the boundedness of our solution, \eqref{psilb} and \eqref{lpji}, we conclude with 
	\begin{align*}
		a(x,t,\nabla \ue) &:= \psi(\we,\he)\nabla \ue,\\
		b(x,t,\ue) &:= \mu_1 \uea \left(-1 + J_1(x,\he)\ast \ueb + J_2(x,\he) \ast \weg \right) - \mut(\he) F(\we)
	\end{align*}
	from Theorem 4 in \cite{DiBenedetto} that there is some $\kappa_3 \in (0,1)$ s.t.
	\begin{align}\label{bounduct}
		\|\ue\|_{C^{\kappa_3,\frac{\kappa_3}{2}}(\otm)} \leq \Cl{bck3}.
	\end{align}
	Again, as in Lemma \ref{lemlocex} (applying Theorem IV.5.3 from \cite{Lady} twice to every function if necessary) it follows that $u,w,h \in C^{2+\vartheta,1+\frac{\vartheta}{2}}(\oa \times [0,\Tmax])$. This contradicts \eqref{limczt} and consequently, $\Tmax = \infty$.
	%
\end{proof}

\noindent We show the uniqueness of this solution as in \cite{EckardtSu}, where we need to restrict $p_1,p_2$ in \eqref{lpji} if $d\geq 3$.
\begin{Lemma}
	Assume that $p_1, p_2$ from \eqref{lpji} satisfy $p_1,p_2 \geq \frac{2d}{d+2}$ if $d\geq 3$ and $p_1,p_2 \in (1,\infty)$ as before if $d=1,2$. { T}hen the classical solution from Theorem \ref{globex} is unique.
\end{Lemma}
\begin{proof}
	Let $\varepsilon \in (0,1)$ and $T \in (0,\infty)$. Assume that there are two solutions $(u_1,w_1,h_1), (u_1,w_1,h_1) \in \left(C^{2+\vartheta,1+\frac{\vartheta}{2}}\left(\oa \times [0,\infty)\right)\right)^3$ to \eqref{IBVPwid}. Then, after subtracting the equations for $h_1$ and $h_2$ from another, {  we obtain} that
	\begin{align*}
		(h_1-h_2)_t = D_H \Delta(h_1-h_2) + g(u_1,w_1) - g(u_2,w_2) - \lambda(h_1-h_2)
	\end{align*}
	holds in $\Omega\times (0,T)$.
	Now, we multiply the above equation with $h_1-h_2$, integrate over $\Omega$, and obtain for $t \in (0,T)$ { by} using partial integration, the Lipschitz continuity of $g$, and Young's inequality, that
	\begin{align}
		&\frac{d}{dt} \io |h_1-h_2|^2 \td x + \lambda \io |h_1-h_2|^2 \td x + D_H \io |\nabla(h_1-h_2)|^2 \td x\nonumber\\
		=& \io (g(u_1,w_1) - g(u_2,w_2))(h_1-h_2) \td x\nonumber\\
		\leq& 2 L_g \io |u_1-u_2||h_1-h_2| \td x + 2 L_g\io |w_1-w_2||h_1-h_2|\td x\nonumber\\
		\leq& \Cl{abschh1h2} \left(\io |u_1-u_2|^2 \td x + \io |w_1-w_2|^2 \td x \right) + \lambda \io |h_1-h_2|^2 \td x. \label{esth1h2}
	\end{align}
	Subtracting the equations for $w_1$ and $w_2$ from another we obtain { that} 
	\begin{align*}
		(w_1-w_2)_t =& \varepsilon \Delta (w_1-w_2) + \mu_2(h_1)(1-w_1) u_1 - \mu_2(h_2)(1-w_2)u_2 + \mu_3(h_2)F(w_2) - \mu_3(h_1)F(w_1) \\
		=& \varepsilon \Delta (w_1-w_2) + \mu_2(h_1) (1-w_1) (u_1 - u_2) + \mu_2(h_1)(w_2-w_1)u_2\\ 
		&+ (\mu_2(h_1) - \mu_2(h_2)) (1-w_2)u_2 + \mu_3(h_2)(w_2 - w_1)\tilde{F}(w_1)\tilde{F}(w_2)  
		+ (\mu_3(h_2) - \mu_3(h_1))F(w_1)
	\end{align*}
	holds in $\Omega \times (0,T)$. { Analogously to above, by using the Lipschitz continuity of $\mu_2$ and $\mu_3$ and the boundedness of the solutions, we get} that for $t \in (0,T)$ it holds:
	\begin{align}
		&\frac{d}{dt} \io |w_1-w_2|^2 \td x + \varepsilon \io |\nabla(w_1-w_2)|^2 \td x \nonumber \\
		\leq& \|\mu_2\|_{L^{\infty}(0,1)} \io |u_1-u_2||w_1-w_2| \td x + (\|\mu_2\|_{L^{\infty}(0,1)} \Cr{boundu} +\|\mu_3\|_{L^{\infty}(0,1)}) \io |w_1-w_2|^2\td x\\ 
		&+ (\|\mu_2'\|_{L^{\infty}(0,1)} \Cr{boundu} + \|\mu_2'\|_{L^{\infty}(0,1)}) \io |h_1-h_2||w_1-w_2| \td x\\
		\leq& \Cl{abschw1w2} \left(\io |u_1-u_2|^2 \td x + \io |w_1-w_2|^2 \td x + \io |h_1-h_2|^2 \td x\right).\label{estw1w2}
	\end{align}
	Further, by subtracting the equations for $u_1$ and $u_2$ from another { we obtain} that
	\begin{align} 
		(u_1-u_2)_t =& \nabla \cdot \left( \psi(w_1,h_1)\nabla u_1 - \psi(w_2,h_2)\nabla u_2\right) + \mu_1 (\ua_1 - \ua_2) + J_1(x,h_2)\ast \ub_2 - J_1(x,h_1)\ast \ub_1 \nonumber\\ 
		&+ J_2(x,h_2) \ast \wg_2 - J_2(x,h_1) \ast \wg_1 + \mut(h_1) F(w_1) - \mut(h_2) F(w_2)\nonumber\\
		=& \nabla \cdot \left( \psi(w_1,h_1)\nabla (u_1 - u_2) + (\psi(w_2,h_2) - \psi(w_2,h_2))\nabla u_2\right) + \mu_1 (\ua_1 - \ua_2)\nonumber\\ 
		&+ J_1(x,h_2)\ast \ub_2 - J_1(x,h_1)\ast \ub_1 + J_2(x,h_2) \ast \wg_2 - J_2(x,h_1) \ast \wg_1\nonumber\\
		&+ (\mut(h_1) - \mut(h_2)) F(w_1) + \mut(h_2) (w_1 - w_2)\tilde{F}(w_1)\tilde{F}(w_2)  \label{absch0u1u2}
	\end{align}
	holds in $\Omega \times (0,T)$. Using the boundedness of $u_2$, the mean value theorem, H\"older's inequality, and the Sobolev embedding theorem (with the help of our additional assumptions on $p_1, p_2$), we can estimate
	\begin{align*}
		&\left|J_1(x,h_2)\ast \ub_2 - J_1(x,h_1)\ast \ub_1\right| \nonumber\\
		\leq& |(J_1(x,h_2) - J_1(x,h_1) )\ast \ub_2| + | J_1(x,h_1)\ast (\ub_1 - \ub_2)|\nonumber\\
		\leq& \Cr{boundu} \io L_{J_1}(x-y)|h_1(y)-h_2(y)| \td y + \beta \Cr{boundu} \io J_1(x-y,h_1)(y)|u_1(y) - u_2(y)| \td y\nonumber\\
		\leq& \Cr{boundu} \|L_{J_1}\|_{L^{p_1}(B)} \|h_1-h_2\|_{L^{\frac{p_1}{p_1-1}}(\Omega)} + \beta \Cr{boundu} (\|L_{J_1}\|_{L^{p_1}(B)}+ \|J_1(x,0)\|_{L^{p_1}(B)}) \|u_1-u_2\|_{L^{\frac{p_1}{p_1-1}}(\Omega)}\nonumber\\
		\leq& \Cl{absch1u1u2} \left(\|h_1-h_2\|_{H^1(\Omega)} + \|u_1-u_2\|_{H^1(\Omega)}\right) 
	\end{align*}
	and conclude using H\"older's and Young's inequalit{ies} that
	\begin{align}
		&\io \left|J_1(x,h_2)\ast \ub_2 - J_1(x,h_1)\ast \ub_1\right| |u_1-u_2| \nonumber \\
		\leq& \Cr{absch1u1u2} \left(\|h_1-h_2\|_{H^1(\Omega)} + \|u_1-u_2\|_{H^1(\Omega)}\right) \|u_1-u_2\|_{L^1(\Omega)} \nonumber \\
		\leq& \Cl{absch2u1u2}\left(D_H^{-1}, \delta^{-1}\right) \left(\|h_1-h_2\|_{L^2(\Omega)}^2 + \|u_1-u_2\|_{L^2(\Omega)}^2 \right) + \frac{D_H}{2} \|\nabla(h_1-h_2)\|_{(L^2(\Omega))^n}^2 + \frac{\delta}{2} \|\nabla(u_1-u_2)\|_{(L^2(\Omega))^n}^2. \label{estj1u1u2}
	\end{align} 
	Analogously, we conclude
	\begin{align}
		&\io \left|J_2(x,h_2)\ast \wg_2 - J_2(x,h_1)\ast \wg_1\right| |u_1-u_2| \td x\nonumber\\
		\leq& \Cl{absch2u1u2}\left(D_H^{-1}, \varepsilon^{-1}\right) \left(\|h_1-h_2\|_{L^2(\Omega)}^2 + \|w_1-w_2\|_{L^2(\Omega)}^2 \right) + \frac{D_H}{2} \|\nabla(h_1-h_2)\|_{(L^2(\Omega))^n}^2 + \frac{\varepsilon}{2} \|\nabla(u_1-u_2)\|_{(L^2(\Omega))^n}^2. \label{estj2u1u2}
	\end{align} 
	Multiplying \eqref{absch0u1u2} by $u_1-u_2$, integrating over $\Omega$, using partial integration, \eqref{psilb}, the Lipschitz continuity of $\psi$, the boundedness of $u_1,u_2,\nabla u_2$ $h_2$, the mean value theorem, \eqref{estj1u1u2}, \eqref{estj2u1u2}, and Young's inequality, we conclude that
	\begin{align}
		&\frac{d}{dt} \io |u_1-u_2|^2 \td x + \delta \io |\nabla(u_1-u_2)|^2 \td x \nonumber\\
		\leq& \|\nabla u_2\|_{(L^{\infty}(\Omega))^n} \io \left(\|\partial_w \psi\|_{L^{\infty}((0,1)^2)}|w_1-w_2| + \|\partial_h \psi\|_{L^{\infty}((0,1)^2)}|h_1-h_2| \right) |\nabla (u_1 - u_2)| \td x \nonumber\\ 
		&+ \mu_1 \alpha \Cr{boundu}^{\alpha-1} \io |u_1 - u_2|^2 \td x
		+ \mu_1\Cr{boundu}^{\alpha}\io |J_1(x,h_2)\ast \ub_2 - J_1(x,h_1)\ast \ub_1||u_1 - u_2| \td x\nonumber\\
		&+ \mu_1\Cr{boundu}^{\alpha}\io |J_2(x,h_2) \ast \wg_2 - J_2(x,h_1) \ast \wg_1||u_1-u_2| \td x \nonumber\\
		&+ L_{\mut}\io |h_1 - h_2||u_1-u_2| \td x + \|\mut\|_{L^{\infty}(0,1)} \io |w_1-w_2||u_1 -u_2| \td x\nonumber\\
		\leq& \Cl{abschu1u2} \left(\|u_1-u_2\|_{L^2(\Omega)}^2 + \|h_1-h_2\|_{L^2(\Omega)}^2 + \|w_1-w_2\|_{L^2(\Omega)}^2\right)\nonumber\\ 
		&+ \delta \|\nabla(u_1-u_2)\|_{(L^2(\Omega))^n}^2 + D_H \|\nabla(h_1-h_2)\|_{L^2(\Omega)}^2 + \varepsilon \|\nabla(w_1-w_2)\|_{(L^2(\Omega))^n}^2. \label{estu1u2}
	\end{align}
	Adding up \eqref{esth1h2}, \eqref{estw1w2} and \eqref{estu1u2} we conclude that for $t\in(0,T)$ it holds: 
	\begin{align*}
		&\frac{d}{dt} \left(\|u_1-u_2\|_{L^2(\Omega)}^2 + \|h_1-h_2\|_{L^2(\Omega)}^2 + \|w_1-w_2\|_{L^2(\Omega)}^2\right)\\ 
		\leq& (\Cr{abschh1h2} + \Cr{abschw1w2} + \Cr{abschu1u2}) \left(\|u_1-u_2\|_{L^2(\Omega)}^2 + \|h_1-h_2\|_{L^2(\Omega)}^2 + \|w_1-w_2\|_{L^2(\Omega)}^2\right)
	\end{align*}
	Finally, we obtain $u_1 \equiv u_2$, $h_1 \equiv h_2$ and $w_1 \equiv w_2$ on $[0,T]$ from Gronwall's inequality. As this holds for all $T\in(0,\infty)$, uniqueness of the classical solution to \eqref{IBVPwid} follows.
\end{proof}

\section{Existence of a weak solution to the original problem}\label{sec:analysis2}
\begin{Definition} \label{defweaksol}
	By a weak solution to \eqref{IBVP} we mean a tuple $(u,w,h)$ of nonnegative functions s.t. for all $T\in (0,\infty)$ it holds that $u\in L^{\infty}(\ot)\cap L^2(0,T;H^1(\Omega))$, $w\in L^{\infty}(\ot)\cap L^{\infty}(0,T;H^1(\Omega))$ with $\partial_t w \in L^2(\ot)$ and $h \in L^{\infty}(\ot) \cap W^{2,1}_2(\Omega\times (0,T))$
	satisfying 
	\begin{align}
		&- \iot \io u \eta_t \td x \td t - \io u_0 \eta(\cdot,0) \td x\nonumber\\
		=& - \iot \io \psi(w,h)\nabla u \cdot\nabla \eta \td x \td t
		+  \mu_1 \iot \io \ua \left(1- J_1(x,h)\ast \ub - J_2(x,h) \ast \wg \right) \eta \td x \td t\nonumber\\ 
		&+ \iot \io \mut(h) F(w)\eta \td x \td t, \label{weakequ}
	\end{align}
	for all $\eta \in W_2^{1,1}(\ot)$ with $\eta(T) = 0$ and
	\begin{align} 
		w_t &= \mu_2(h)u - \mu_3(h)F(w), \label{weakeqw}\\
		h_t &= D_H \Delta h + g(u,w) -\lambda h \label{weakeqh}
	\end{align}
	a.e. in $\Omega\times (0,T)$, $w(\cdot, 0) = w_0(\cdot)$ and $h(\cdot, 0) = h_0(\cdot)$ a.e. in $\Omega$ and $\partial_{\nu} h = 0$ a.e. on $\partial \Omega \times (0,\infty)$.
\end{Definition}

\begin{Lemma}\label{lemconve}
	There are $u\in L^{\infty}(\ot)\cap L^2(0,T;H^1(\Omega))$, $w\in L^{\infty}(\ot) \cap L^{\infty}(0,T;H^1(\Omega))$ and $h \in L^{\infty}(\ot) \cap W^{2,1}_2(\Omega\times (0,T))$ with $ u \leq \Cr{boundu}$ and $w,h\leq 1$ s.t. for a subsequence
	\begin{subequations}
		\begin{align}
			\nabla \uek &\underset{k \to \infty}{\rightharpoonup} \nabla u &\text{ in } L^2(\ot), \label{convnuek}\\
			\uek &\underset{k \to \infty}{\rightarrow} u &\text{ in } L^2(\ot) \text{ and a.e. in }\ot, \label{convuek}\\
			\nabla \wek &\overset{*}{\underset{k \to \infty}{\rightharpoonup}} \nabla w &\text{ in } L^{\infty}(0,T;L^2(\Omega)),\label{convnwek}\\
			\wek &\underset{k \to \infty}{\rightarrow} w &\text{ in } C([0,T];L^2(\Omega)) \text{ and a.e. in }\ot, \label{convwek},\\
			\partial_t \wek &\underset{k \to \infty}{\rightharpoonup} \partial_t w &\text{ in } L^2(\ot), \label{convptwek}\\
			\hek &\underset{k \to \infty}{\rightharpoonup} h &\text{ in } L^2(0,T;H^2(\Omega)), \label{convdhek}\\
			\hek &\underset{k \to \infty}{\rightarrow} h &\text{ in } L^2(0,T;H^1(\Omega)) \text{ and a.e. in }\ot, \label{convnhek}\\
			\partial_t \hek &\underset{k \to \infty}{\rightharpoonup} \partial_t h &\text{ in } L^2(\ot), \label{convpthek}
		\end{align}
	\end{subequations}
\end{Lemma}
\begin{proof}
	Let $T>0$. First, Theorem IV.9.1 (and the remark at the end of the chapter) from \cite{Lady} implies that
	\begin{align}
		\|\he\|_{W^{2,1}_2(\ot)} \leq \C \left(\|g(\ue,\we)\|_{L^2(\ot)} + \|\hoe\|_{H^1(\Omega)}\right). \label{boundhw212}
	\end{align}
	Due to the Lipschitz continuity of $g$, the uniform boundedness of $(\ue)$ and $(\we)$, and the convergence in \eqref{convhoe}, the right-hand side of \eqref{boundhw212} is uniformly bounded for $\varepsilon\in (0,1)$. Hence, Lions-Aubin (with $H^2(\Omega) \subset\subset H^1(\Omega) \subset L^2(\Omega)$) and Banach-Alaoglu imply the existence of $h \in W^{2,1}_2(\ot)$ and a subsequence s.t. \eqref{convdhek} - \eqref{convpthek} hold.
	Hence, from Lemma \ref{lemlocex} if follows for a.e. $(x,t) \in \ot $ that
	\begin{align*}
		h(x,t) = \lim\limits_{k \to \infty} \hek(x,t) 
		\leq 1.
	\end{align*}
	Next, we want to obtain a uniform estimate on the norm of $\nabla \ue$. Therefore, we multiply the equation for $\ue$ in \eqref{IBVPwid} by $\ue$, integrate over $\Omega$, and use partial integration to obtain
	\begin{align*}
		\frac{d}{dt} \io \ue^2 \td x =& - \io \psi(\we,\he) |\nabla \ue|^2 \td x + \mu_1 \io \ue^{\alpha+1} \left(1- J_1(x,\he)\ast \ueb - J_2(x,\he) \ast \weg \right) \td x\\ 
		&+ \io \mut(\he) F(\we) \ue \td x.
	\end{align*}
	Then, we can estimate using \eqref{psilb}, \eqref{lbj1}, the uniform boundedness of $(\ue)$, $(\we)$, $(\he)$, $\alpha \geq 1$, the continuity of $\mut$, and Young's inequality that
	\begin{align*}
		&\frac{d}{dt} \io \ue^2 \td x +  \delta \io |\nabla \ue|^2 \td x \\
		\leq& \mu_1 \Cr{boundu}^{\alpha-1}\left(1+ \Cr{boundu}^{\beta}\|J_1(x,\he)\|_{L^1(B)} + \|J_2(x,\he)\|_{L^1(B)} \right) \io \ue^2 \td x + \|\mut\|_{L^{\infty}(0,1)} \io \ue \td x\\
		\leq& \C  \left(1+ \|\L_{J_1}\|_{L^1(B)} + \|J_1(\cdot,0)\|_{L^1(B)} + \|\L_{J_1}\|_{L^1(B)} + \|J_2(\cdot,0)\|_{L^1(B)}\right)\io \ue^2 \td x + \C
	\end{align*}
	Hence, we conclude from Gronwall's inequality that for all $\varepsilon \in(0,1)$ it holds:
	\begin{align}\label{boundnu}
		\|\nabla \ue\|_{L^2(\ot)} \leq \C(T).
	\end{align}
	Further, we multiply the equation for $\ue$ in \eqref{IBVPwid} by $\varphi \in H_0^1(\Omega)$ and obtain {\cb upon} using partial integration, the H\"older inequality, the continuity of $\psi$, the uniform boundedness of $(\ue)$, $(\we)$ and $(\he)$, \eqref{lbj1}, and the Lipschitz continuity of $\mut$ that
	\begin{align*}
		\left|\io \partial_t \ue \varphi \td x \right|
		\leq&  \io |\psi(\we,\he) \nabla \varphi| \td x + \mu_1 \io |\uea \left(1- J_1(x,\he)\ast \ueb - J_2(x,\he) \ast \weg \right) \varphi| \td x\\ 
		&+ \io |\mut(\he) F(\we) \varphi| \td x\\
		\leq& \| \psi \|_{L^{\infty}((0,1)^2)} \|\nabla \ue \|_{(L^2(\Omega))^d} \|\nabla \varphi\|_{(L^2(\Omega))^d}\\ 
		&+\left(\mu_1 \Cr{boundu}^{\alpha} \left(1 + \Cr{boundu}^{\beta}\|J_1(x,\he)\|_{L^1(B)} + \|J_2(x,\he)\|_{L^1(B)}\right) +  \|\mut\|_{L^{\infty}(0,1)}\right) \io |\varphi| \td x\\
		\leq& \Cl{boundptu1} ( \|\nabla \ue \|_{(L^2(\Omega))^n} \|\nabla \varphi\|_{(L^2(\Omega))^n} + \|\varphi\|_{L^2(\Omega)}).
	\end{align*}
	Hence, 
	\begin{align*}
		\|\partial_t \ue\|_{H^{-1}(\Omega)} \leq \Cr{boundptu1} (\|\nabla \ue \|_{(L^2(\Omega))^n}  + 1)
	\end{align*}
	and we conclude from \eqref{boundnu} that
	\begin{align}\label{boundptu}
		\|\partial_t \ue\|_{L^2(0,T;H^{-1}(\Omega))} \leq \C(T).
	\end{align}
	Combining \eqref{boundnu} and \eqref{boundptu} with the uniform boundedness of $(\ue)$ we conclude from Lions-Aubin (with $H^1(\Omega)\subset \subset L^2(\Omega) \subset H^{-1}(\Omega)$) and Banach-Alaoglu that there is $u\in C(0,T;L^2(\Omega)) \cap L^2(0,T;H^1(\Omega))$ and a subsequence s.t. \eqref{convnuek} and \eqref{convuek} hold.
	Moreover, it holds a.e. in $\ot$ that $u(x,t) \leq \lim\limits_{k \to \infty}\uek(x,t) \leq \Cr{boundu}$ due to the pointwise a.e. convergence.\\[-2ex] 
	
	\noindent To obtain a uniform estimate on the norm of $\nabla \we$ we multiply the equation for $\we$ by $\Delta \we$ and obtain after integration over $\Omega$ and partial integration, due to our boundary condition on $\we$, that
	\begin{align*}
		&\frac{d}{dt} \io |\nabla \we |^2 \td x = \io (\nabla \we)_t \cdot \nabla \we \td x = \io (\we)_t \Delta \we \td x\\
		=& - \varepsilon \io |\Delta\we|^2 \td x - \io \mu_2(\he)(1-\we)\ue \Delta \we \td x + \io \mu_3(\he)F(\we) \Delta \we \td x\\
		=&- \varepsilon \io |\Delta\we|^2 \td x + \io \left(\mu_2'(\he)(1-\we)\nabla \he \ue - \mu_2(\he)\nabla \we \ue +  \mu_2(\he)(1-\we)\nabla \ue\right) \cdot \nabla \we \td x\\ 
		&- \io \left(\mu_3'(\he)\nabla \he  F(\we) + \mu_3(\he)(\tilde{F}(\we))^2\nabla \we \ue\right) \nabla \we \td x.
	\end{align*}
	Hence, using the uniform boundedness of $(\ue), (\he)$, { the} continuity of $\mu_2',\mu_3'$, and Young's inequality, we obtain
	\begin{align*}
		&\frac{d}{dt} \io |\nabla \we |^2 \td x + \varepsilon \io |\Delta\we|^2 \td x\\
		\leq& (\|\mu_2'\|_{L^{\infty}(0,1)} \Cr{boundu} + \|\mu_3'\|_{L^{\infty}(0,1)}) \io |\nabla \he||\nabla \we| \td x + \|\mu_2\|_{L^{\infty}(0,1)} \io |\nabla \ue||\nabla \we| \td x\\ 
		&+ \left(\|\mu_2\|_{L^{\infty}(0,1)} \Cr{boundu} + \|\mu_3\|_{L^{\infty}(0,1)} \right) \io |\nabla\we|^2 \td x\\
		\leq& \Cl{boundnw} \left(\io |\nabla\ue|^2 \td x + \io |\nabla\we|^2 \td x + \io |\nabla\he|^2 \td x\right).
	\end{align*}
	Hence, Gronwall's inequality, the fact that $\|\nabla \woe\|_{L^2(\Omega)} \leq \Cl{boundnwoel2}$ due to the convergence in \eqref{convwoe}, { and} \eqref{boundhw212} imply that
	\begin{align*}
		\io |\nabla \we(t) |^2 \td x + \varepsilon \iot \io |\Delta \we |^2 \td x \td t &\leq e^{\Cr{boundnw}T}\left(\|\nabla \woe\|_{L^2(\Omega)} + \Cr{boundnw} \int_0^T \io |\nabla\ue|^2 \td x + \io |\nabla\he|^2 \td x  \td t \right) \\
		&\leq \Cl{boundnw2}(T)
	\end{align*}
	for $t\in(0,T)$ and $\varepsilon \in (0,1)$. Consequently, for all $\varepsilon \in (0,1)$ it holds that
	\begin{align} 
		\|\nabla \we \|_{L^{\infty}(0,T;(L^2(\Omega))^n)} \leq \Cr{boundnw2}(T), \label{boundnw}\\
		\varepsilon\|\Delta \we\|_{(L^2(\ot))^n} \leq \Cl{boundnw3}(T). \label{bounddivw}
	\end{align}
	To obtain a uniform estimate on some norm of the time derivative of $\we$, we multiply the equation for $\we$ from \eqref{IBVPwid} by $\varphi \in L^2(\Omega)$, integrate over $\Omega$, use the Lipschitz continuity of $\mu_2$ and $\mu_3$ on $(0,1)$, the uniform boundedness of $(\ue)$, $(\we)$ and $(\he)$, \eqref{boundnw}, and H\"older's inequality to conclude that
	\begin{align*}
		\left|\io \partial_t \we \varphi \td x\right| 
		&\leq \varepsilon \io |\Delta\we| |\varphi| \td x  + \io |\mu_2(\he)(1-\we)\ue \varphi| \td x + \io |\mu_3(\he)F(\we) \varphi| \td x\\
		&\leq \left( \varepsilon\|\Delta\we\|_{L^2(\Omega)} + (\|\mu_2\|_{L^{\infty}(0,1)}\Cr{boundu} + \|\mu_3\|_{L^{\infty}(0,1)})|\Omega|^{\frac{1}{2}}\right) \|\varphi\|_{L^2(\Omega)}.
	\end{align*}
	Consequently, we conclude from \eqref{bounddivw} as $(L^2(\Omega))^* = L^2(\Omega)$ that for all $\varepsilon \in (0,1)$ it holds that
	\begin{align} \label{boundptw}
		\|\partial_t \we \|_{L^2(\ot)} \leq \C(T).
	\end{align}
	Combining the uniform boundedness of $(\we)$ with \eqref{boundnw} - \eqref{boundptw} we conclude from Lions-Aubin (with $H^1(\Omega) \subset \subset L^2(\Omega) \subset L^2(\Omega)$) and Corollary 3.30 in \cite{Brezis} due to $L^{\infty}(0,T;L^2(\Omega)) = (L^1(0,T;L^2(\Omega)))^*$, 
	as $L^1(0,T;L^2(\Omega)(\Omega))$ is separable, that there is $w \in C(0,T;L^2(\Omega))\cap L^{\infty}(0,T;H^1(\Omega))$ and a subsequence 
	s.t. \eqref{convnwek} - \eqref{convptwek} hold.
	From the pointwise a.e. convergence we conclude that $w(x,t) \leq \lim\limits_{k \to \infty}\wek(x,t) \leq 1$ holds a.e. in $\ot$ due to \eqref{convwoe}. 
\end{proof}

\noindent Due to the pointwise convergence shown in the proof of the last lemma the following corollary is a direct consequence of Corollary \ref{corbb}.
\begin{Corollary}
	For $K>1$ and 'small' enough parameters it holds that
	\begin{align*}
		\|u\|_{L^{\infty}(\Omega\times (0,\infty))} 
		\leq& K\max\left\{1,\left(4C_S |\Omega|^{-\frac{1}{2}}\right)^{\frac{1-\frac{2}{s}}{\left(1-\frac{1}{s}\right)\left(\beta+1-\alpha-\frac{2\beta}{s}\right)}} \left(\frac{2\Cr{c1}}{\mu_1\eta}\right)^{\frac{1-\frac{2}{s}}{\beta+1-\alpha-\frac{2\beta}{s}}}\right\}. 
	\end{align*}
\end{Corollary}

\begin{Theorem}\label{thmexweaksol}
	There is a bounded nonnegative weak solution $(u,w,h)$ to \eqref{IBVP} in the sense of Definition \ref{defweaksol} satisfying $u\leq \Cr{boundu}, w, h \leq 1$ a.e. in $\Omega \times (0,\infty)$. 
\end{Theorem}
\begin{proof}
	Let $T>0$ and $\eta \in W_2^{1,1}(\ot)$ with $\eta(T)=0$. We consider the subsequence from Lemma \ref{lemconve}. Multiplying the equations from \eqref{IBVPwid} by $\eta$ and using partial integration we obtain the weak formulation
	\begin{align*}
		&- \iot \io \uek \eta_t \td x \td t - \io u_{0\varepsilon_k} \eta(\cdot,0) \td x\\
		=& - \iot \io \psi(\wek,\hek)\nabla \uek \cdot \nabla \eta \td x \td t
		+  \mu_1 \iot \io \uek^{\alpha} \left(1- J_1(x,\hek)\ast \uek^{\beta} - J_2(x,\hek) \ast \wek^{\gamma} \right) \eta \td x \td t\\ 
		&+ \iot \io \mut(\hek)  F(\wek) \eta \td x \td t
	\end{align*}
	and
	\begin{align*}
		\iot \io \partial_t \wek \eta \td x \td t
		&= - \varepsilon_k \iot \nabla \wek \cdot \nabla \eta \td x \td t + \iot \left(\mu_2(\hek)(1-\wek)\uek  - \mu_3(\hek)F(\wek)\right)\eta \td x \td t\\
		\iot \io \partial_t \hek \eta \td x \td t &= D_H \iot \Delta \hek \eta \td x \td t + \iot \io (g(\uek,\wek) - \lambda \hek)\eta \td x \td t.
	\end{align*}
	
	\noindent From the continuity of $\psi$ and $\mut$, \eqref{convwek}, \eqref{convnhek}, and the dominated convergence theorem, we conclude that
	\begin{align}
		\psi(\wek,\hek)\nabla \eta \underset{k\to\infty}{\rightarrow} \psi(w,h) \nabla \eta \text{ in } L^2(\ot),\nonumber\\
		\mut(\hek) F(\wek) \underset{k\to\infty}{\rightarrow} \mut(h) F(w) \text{ in } L^2(\ot). \label{convmut}
	\end{align}
	Hence, 
	\begin{align} \label{convpsi}
		\iot \io \psi(\wek,\hek)\nabla \uek \cdot \nabla \eta \td x \td t \underset{k\to\infty}{\rightarrow} \iot \io \psi(w,h)\nabla u \cdot \nabla \eta \td x \td t 
	\end{align}
	follows from \eqref{convnuek} and compensated compactness (see e.g. Proposition 3.13 (iv) in \cite{Brezis}). Further, due to \eqref{estj} and the uniform boundedness of $(\uek)$ for $(x,t) \in \ot$ we estimate
	\begin{align*}
		&\left|J_1(x,\hek) \ast \uek^{\beta}(t) - J_1(x,h)\ast u^{\beta}(t)\right|\\
		\leq& \io |J_1(x-y,\hek(y,t)) - J_1(x-y,h(y,t))|\uek^{\beta}(y,t) \td x
		+ \io J_1(x-y,h(y,t))|\uek^{\beta}(y,t)-u^{\beta}(y,t)| \td x\\
		\leq& \Cr{boundu}^{\beta} \io L_{J_1}(x-y)|\hek(y,t) - h(y,t)| \td x
		+ \io (L_{J_1}(x-y)h(y,t) + J_1(x-y,0))|\uek^{\beta}(y,t)-u^{\beta}(y,t)|  \td x\\
		&\underset{k\to\infty}{\rightarrow} 0,
	\end{align*}
	due to the dominated convergence theorem combined with \eqref{convwek} and \eqref{convnhek} and the uniform boundedness of $(\uek)$ and $(\hek)$. Consequently, the convergence
	\begin{align}\label{convj1pw}
		J_1(x,\hek) \ast \uek^{\beta} \underset{k\to\infty}{\rightarrow} J_1(x,h)\ast u^{\beta} \text{ pointwise a.e. in }\ot
	\end{align}
{ holds} and analogously,
	\begin{align}\label{convj2pw}
		J_2(x,\hek) \ast \wek^{\gamma} \underset{k\to\infty}{\rightarrow} J_2(x,h)\ast w^{\gamma} \text{ pointwise a.e. in }\ot
	\end{align}
	follows.
	Hence, we conclude from \eqref{convuek}, \eqref{convwek}, \eqref{convnhek}, \eqref{convj1pw}, and \eqref{convj2pw} that
	\begin{align*}
		\uek^{\alpha} \left(1- J_1(x,\hek)\ast \uek^{\beta} - J_2(x,\hek) \ast \wek^{\gamma}\right)
		\underset{k\to\infty}{\rightarrow} 
		u^{\alpha} \left(1 - J_1(x,h)\ast u^{\beta} - J_2(x,h) \ast w^{\gamma} \right)\\ \text{ pointwise a.e. in } \ot
	\end{align*}
	and from the uniform boundedness of $(\uek)$, $(\wek)$ and $(\hek)$ and Lemma \ref{lemholj} that for all $k\in\N$ it holds that
	\begin{align*}
		|\uek^{\alpha} \left(1- J_1(x,\hek)\ast \uek^{\beta} - J_2(x,\hek) \ast \wek^{\gamma}\right)| \leq \C.
	\end{align*}
	Hence, the dominated convergence theorem implies 
	\begin{align*}
		\uek^{\alpha} \left(1- J_1(x,\hek)\ast \uek^{\beta} - J_2(x,\hek) \ast \wek^{\gamma}\right)
		\underset{k\to\infty}{\rightarrow} 
		u^{\alpha} \left(1 - J_1(x,h)\ast u^{\beta} - J_2(x,h) \ast w^{\gamma} \right)\\ 
		\text{ in } L^2(\ot).
	\end{align*}
	Combining this with \eqref{convuoe}, \eqref{convuek}, \eqref{convmut} and \eqref{convpsi} we conclude that $u$ satisfies \eqref{weakequ}.\\[-2ex]
	
	\noindent With the help of H\"older's inequality and \eqref{boundnw} we obtain
	\begin{align}
		\varepsilon_k \left|\iot \io \nabla \wek \cdot \nabla \eta \td x \td t\right|
		\leq \varepsilon \|\nabla \wek \|_{(L^2(\ot))^n}\|\nabla \eta \|_{(L^2(\ot))^n}
		\leq \varepsilon \Cr{boundnw2} T^{\frac{1}{2}} \|\nabla \eta \|_{(L^2(\ot))^n}
		\underset{k \to \infty}{\rightarrow} 0 . \label{convnwe}
	\end{align}
	Further, we conclude from the continuity of $\mu_2$ and $\mu_3$, the pointwise convergences in \eqref{convuek}, \eqref{convwek} and \eqref{convnhek}, the uniform boundedness of $(\uek)$, $(\wek)$ and $(\hek)$ and the dominated convergence theorem that
	\begin{align}
		\mu_2(\hek)(1-\wek)\uek \underset{k \to \infty}{\rightarrow} \mu_2(h)(1-w)u &\text{ in } L^2(\ot), \label{convm2}\\
		\mu_3(\hek)F(\wek) \underset{k \to \infty}{\rightarrow} \mu_3(h)F(w) &\text{ in } L^2(\ot). \label{convm3}
	\end{align}
	Combining \eqref{convwoe} and \eqref{convwek} with \eqref{convnwe} - \eqref{convm3} we conclude that
	\begin{align*}
		\iot \io \partial_t w \eta \td x \td t
		= \iot \left(\mu_2(h)(1-w)u  - \mu_3(h)F(w)\right)\eta \td x \td t.
	\end{align*}
	Due to $C_c^{\infty}(\ot) \subset W_2^{1,1}(\ot)$, the fundamental lemma of calculus of variations implies that $w$ satisfies \eqref{weakeqw}. Moreover, using partial integration, we conclude due to \eqref{convwoe} and \eqref{convwek} that (especially also for $\eta \in C^1(0,T;H^1(\Omega)) \subset W^{1,1}_2(\ot)$ with $\eta(T) = 0$) it holds 
	\begin{align*}
		- \iot \io w \eta_t \td x \td t - \io w(\cdot,0) \eta(\cdot,0) \td x \td t 
		&= \iot \io \partial_t w \eta \td x \td t\\
		&= \lim\limits_{k\to \infty} \iot \io \partial_t \wek \eta \td x \td t\\
		&= \lim\limits_{k\to \infty} - \iot \io \wek \eta_t \td x \td t - \io w_{0\varepsilon_k}(\cdot) \eta(\cdot,0) \td x \td t\\
		&= - \iot \io w \eta_t \td x \td t - \io w_0(\cdot) \eta(\cdot,0) \td x \td t.
	\end{align*}
	Hence, due to $C_c^{\infty}(\oa) \subset H^2(\Omega)$ we conclude again from the fundamental lemma of calculus of variations that, indeed, $w(\cdot,0) = w_0(\cdot)$ a.e. in $\Omega$.\\[-2ex]
	
	\noindent Furthermore, we estimate with the help of H\"older's and Minkowski's inequalit{ies} and the Lipschitz continuity of $g$ that
	\begin{align*}
		\iot \io |g(\uek,\wek) - g(u,w)||\eta| \td x \td t \leq L_g\left(\|\uek - u\|_{L^2(\ot)} + \|\uek - u\|_{L^2(\ot)} \right) \|\varphi\|_{L^2(\ot)} \underset{k\to \infty}{\rightarrow} 0,
	\end{align*}
	due to \eqref{convuek} and \eqref{convwek}.
	Hence, it follows from \eqref{convdhek} - \eqref{convpthek} that
	\begin{align*}
		\iot \io \partial_t h \eta \td x \td t = D_H \iot \io \Delta h \eta \td x \td t + \iot \io (g(u,w) - \lambda h)\eta \td x \td t
	\end{align*}
	for all $\eta \in W_2^{1,1}(\ot)$. We conclude again from the fundamental lemma of calculus of variations that $h$ satisfies \eqref{weakeqh} a.e. in $\ot$. Finally, $h(\cdot, 0) = h_0(\cdot)$ a.e. in $\Omega$ follows as for $w$. We conclude similary that
	\begin{align*}
		\iot \int_{\partial \Omega} \nabla h \cdot \nu \eta \td \sigma(x) \td t = 0,
	\end{align*}
	which gives us $\nabla h \cdot \nu = 0$ a.e. on $\partial \Omega \times (0,T)$.
\end{proof}

\section{1D Simulations}\label{sec:simulations}
In this section we simulate the behavior of solutions to \eqref{IBVP} in one dimension. Thereby, we decompose the domain $\Omega = [-5,5]$ into an equidistant mesh $x_2,\dots,x_{N-1}$ with step size $dx = 0.05$ and the time interval $[0,50]$ with step size $dt = 0.0001$. For a simulation of the no-flux boundary condition we add points $x_1<x_2$ and $x_N > x_{N-1}$ outside of $\Omega$ and assume equality of the solutions on the neighboring points.
As in \cite{EckardtSu}, we use the method from \cite{Nadin2011} to discretize the nonlocal integral terms via a composite trapezoidal rule. Moreover, as in \cite{ZRModelling}, we recompute the convolution matrices only every 40 times steps to improve the runtime. Thereby, we assume that  changes in the values of the convolution matrices $phi\_mat_1$ and $phi\_mat_2$ (due to changes of $h$) are negligible within this time interval. Namely, for $i, j \in \{2,\dots,N-1\}$ the corresponding entry of the $k$th-convolution matrix 
is
\begin{align*}
	(phi\_mat_1)_{ij}^k = J_1(x_i-x_j,h_j^{40k}).
\end{align*}
and with the help of this we compute the $n+1$st convolution term at $x_i$, $i \in \{2,\dots,N-1\}$ as:
\begin{align*}
	(conv_1)_{i}^{n+1} = \sum_{j=3}^{N-2}  (phi\_mat_1)_{ij}^{\lfloor n/40\rfloor}(u_j^n)^{\beta} + \frac{1}{2}\left((phi\_mat_1)_{i2}^{\lfloor n/40\rfloor}(u_2^n)^{\beta} + (phi\_mat_1)_{i(N-1)}^{\lfloor n/40\rfloor}(u_2^n)^{\beta}\right)
\end{align*}
Analogously we compute $phi\_mat_2$ and $conv_2$.
For the discretization of the diffusion term we use finite differences and an upwind scheme.
The initial conditions are depicted in Figure \ref{fig:1} and are given by
\begin{align*}
	u_0(x) &= \begin{cases}
		0.3e^{-\frac{1}{5}(x+5)^2}, & x \in [-5,0],\\
		0.3e^{-5}\left(1-\frac{x}{5}\right), & x \in (0,5],
	\end{cases},\\
	w_0(x) &= \begin{cases}
		0.7e^{-(x+5)^2}, & x \in [-5,0],\\
		0.7e^{-25}\left(1-\frac{x}{5}\right), & x \in (0,5],
	\end{cases},\\
	h_0(x) &= \begin{cases}
		0.3e^{-5}, & x \in [-5,0],\\
		0.3e^{-5}\left(1-\frac{x}{5}\right), & x \in (0,5],
	\end{cases}.
\end{align*}
\begin{figure}[h!]
	\begin{center}
		\includegraphics[width=0.44\linewidth]{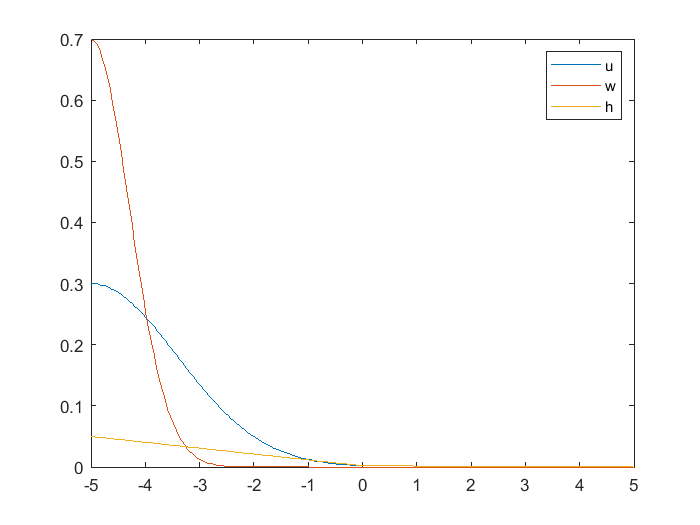}
		\caption{Initial conditions $u_0$, $w_0$, $h_0$.}\label{fig:1}	
	\end{center}
\end{figure}

\noindent We choose  the functions $\psi(h,w)=0.5$ (for simplicity), $\mu_2(h) = h$ (meaning that the net 'deactivation' of $u$-cells is directly proportional to the amount of protons available in the microtumor space), $\mu_3(h) = \mut(h) =\frac{h}{1+h}$ (there is no loss of $w$-cells when becoming $u$-cells, the transition - primarily to motility- is favored by acidity, but in a limited manner, quickly reaching saturation), $g(u,w) = \frac{u+w}{1+u+w}$ (both phenotypes are producing acid, also in a limited way), and the constants $D_H =0.1$ and $\lambda = 1$.\\[-2ex]

\noindent
First, we took $\beta=\gamma=\mu_1=1$ and explored the influence of the kernels on the minimal value $\alpha^*$ of $\alpha$ for which the solution ceases to exist globally in time (with accuracy to one decimal place). Thereby, we considered as in \cite{EckardtSu} the logistic kernel $J_L(x) = \frac{1}{2+e^x+e^{-x}}$, the uniform kernel $J_U(x) = \chi_{[-1,1]}(x)$ and, moreover, the $h$-dependent kernels 
\begin{align} 
	J_1(x,h) &= \frac{1}{\sqrt{2\pi}}e^{-\frac{x^2}{2}}\left(\frac{h}{1+h} + \frac{1}{10}\right), \label{j1}\\
	J_2(x,h) &= \frac{h^{2}}{2(1+h^2)}, \label{j2}
\end{align}
the first of which is a $h$- dependent shift of a Gaussian, while the latter is a Holling III-type function of $h$ suggesting a slower increase towards saturation, with a certain 'learning effect' as far as the response to more acidity is concerned: as $J_2$ stands for the interaction of the two cell phenotypes, it accounts for both of them extruding protons, along with the corresponding adaptation of $u$-cells to interspecific cues.\\[-2ex]

\begin{figure}[h!]
	\includegraphics[width=0.24\linewidth]{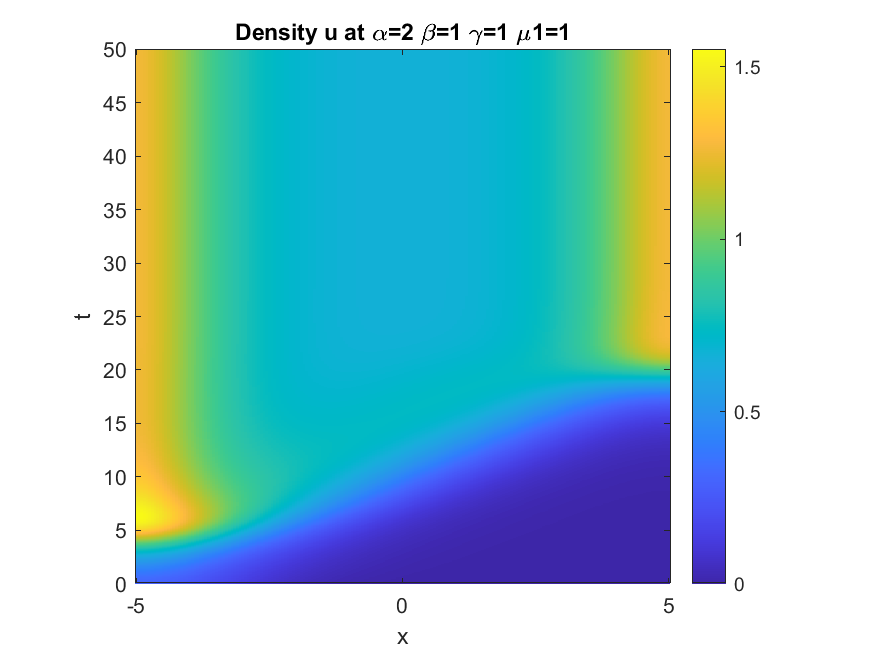}\  \includegraphics[width=0.24\linewidth]{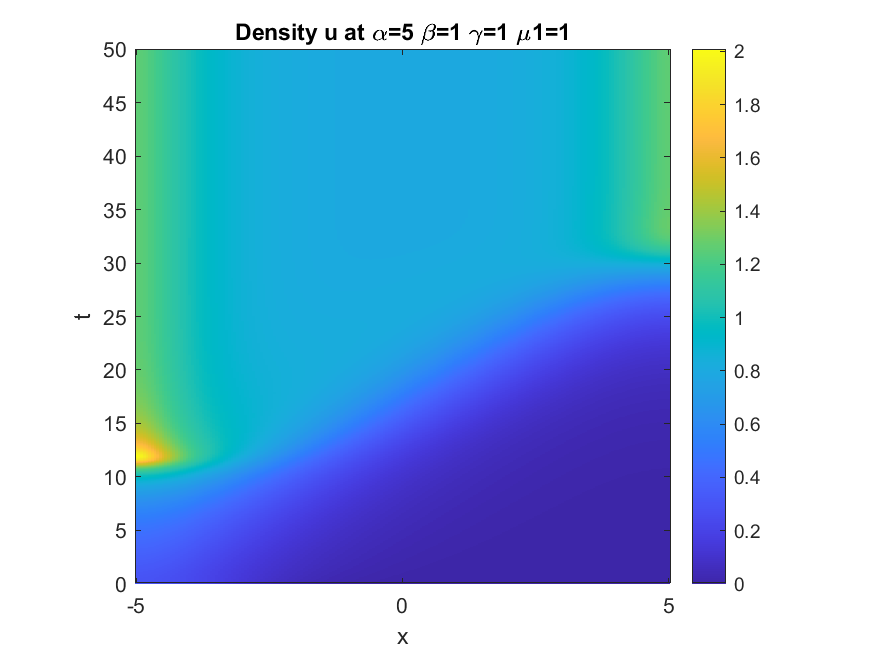}\  \includegraphics[width=0.24\linewidth]{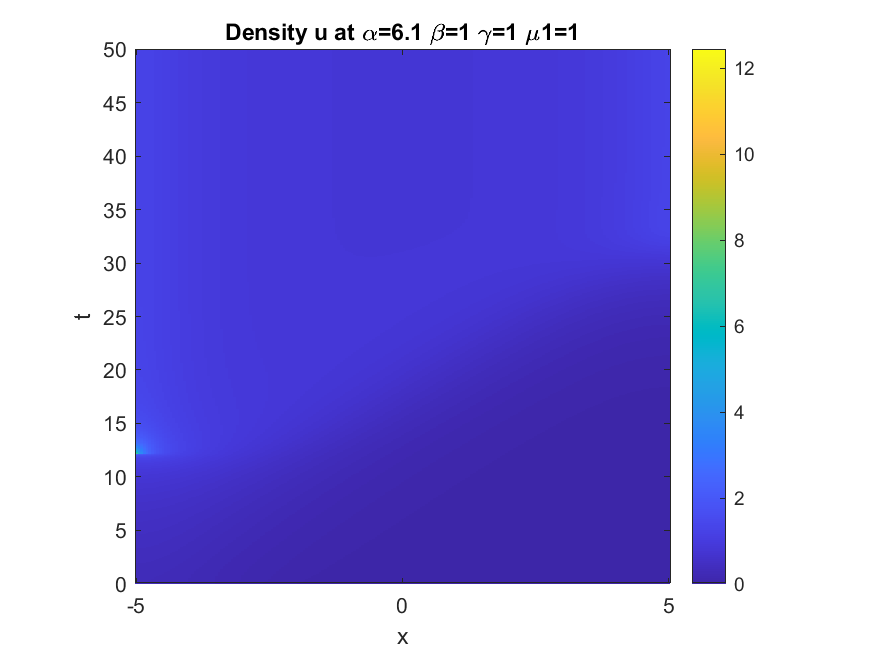}\ 
	\includegraphics[width=0.24\linewidth]{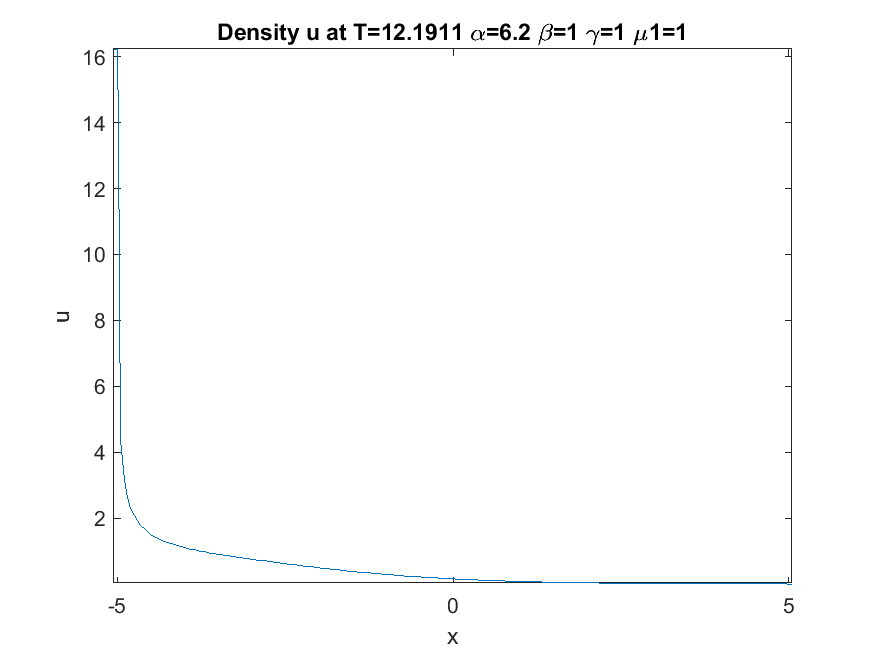}\\
	\includegraphics[width=0.25\linewidth]{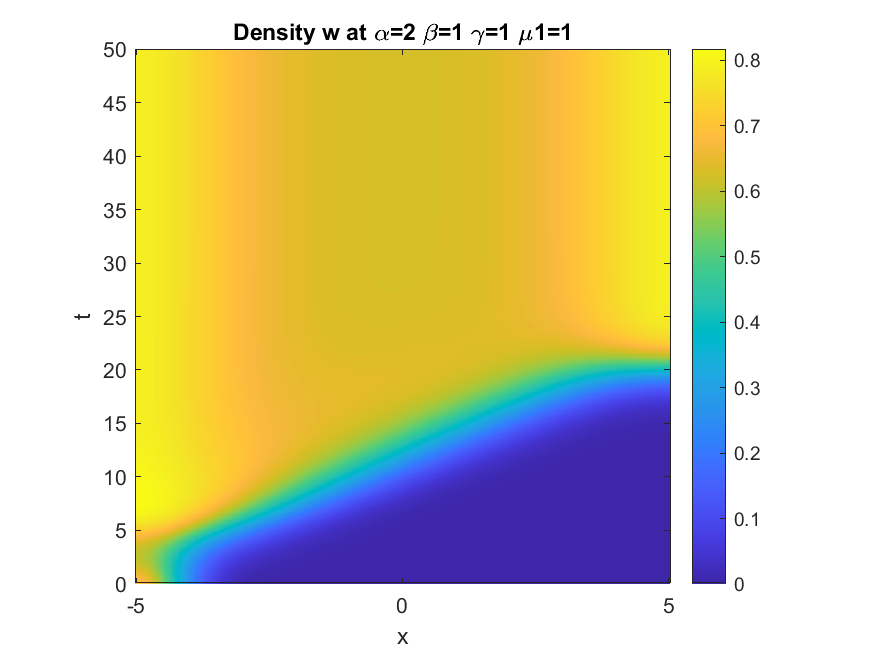}\  \includegraphics[width=0.24\linewidth]{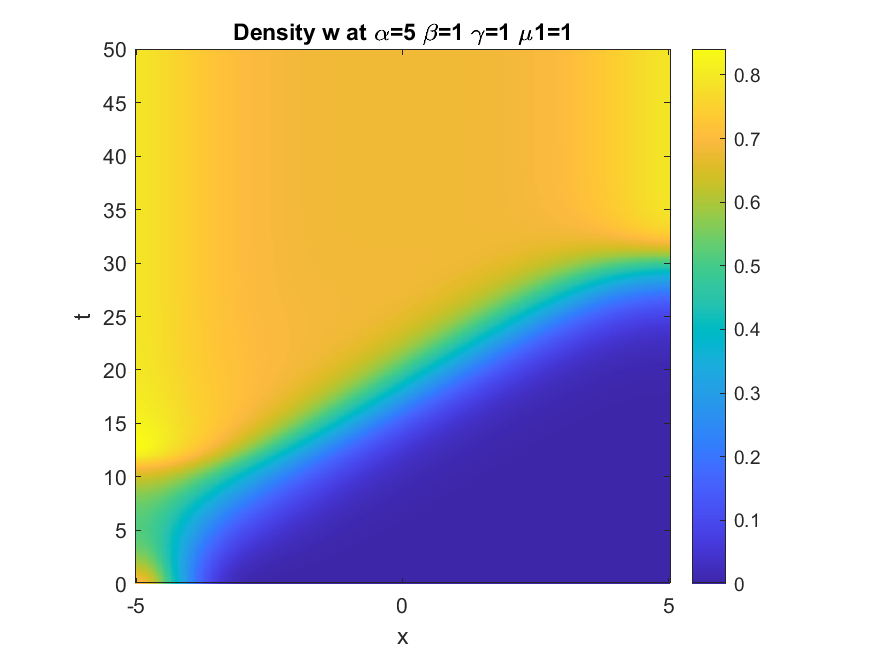}\  \includegraphics[width=0.24\linewidth]{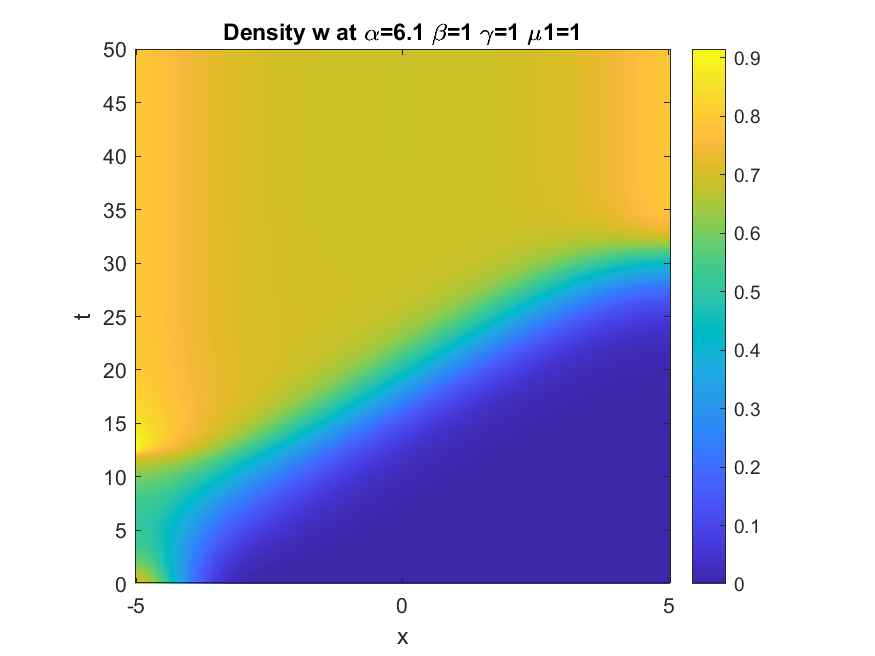}\ 
	\includegraphics[width=0.24\linewidth]{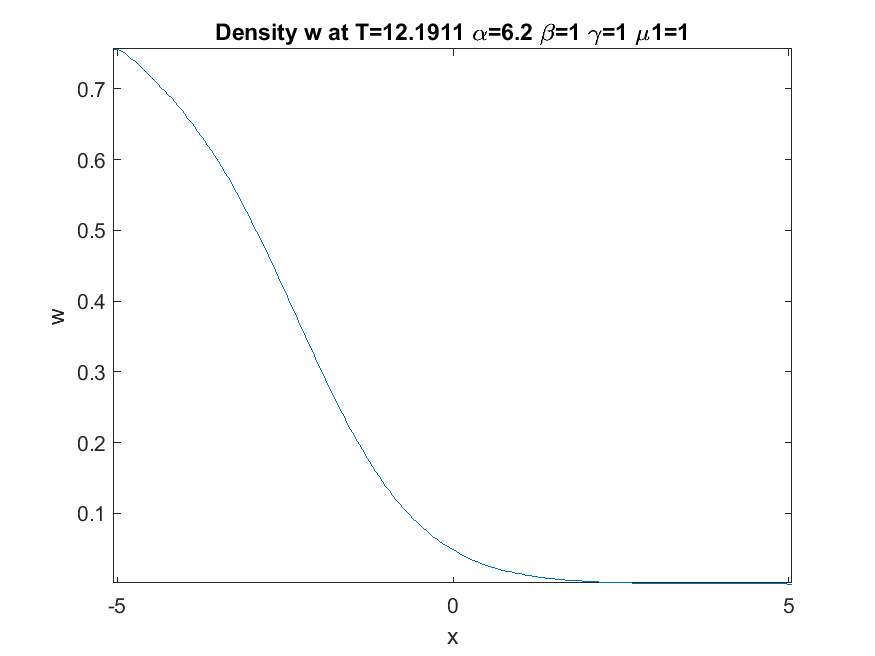}\\
	\includegraphics[width=0.25\linewidth]{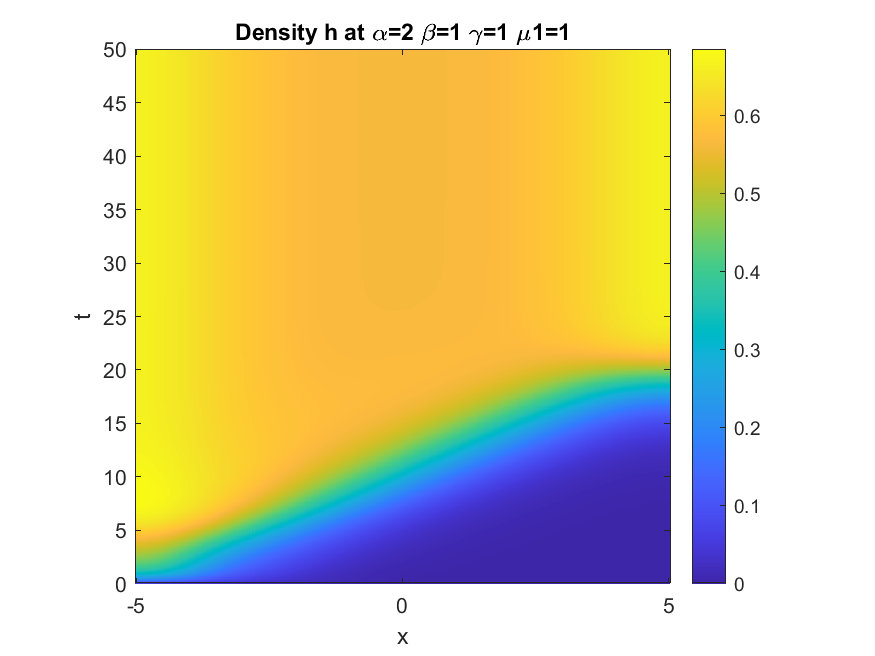}\  \includegraphics[width=0.24\linewidth]{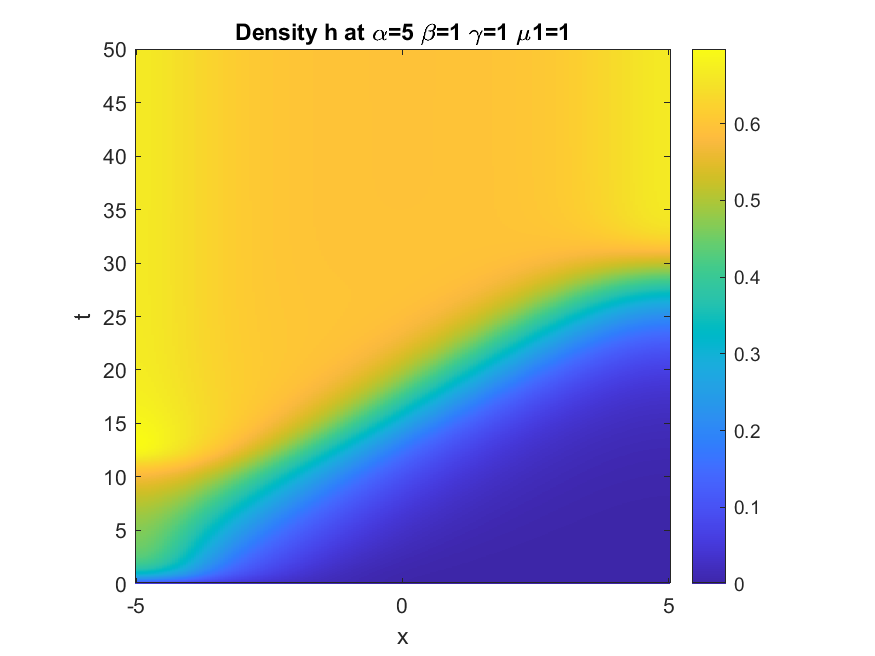}\  \includegraphics[width=0.24\linewidth]{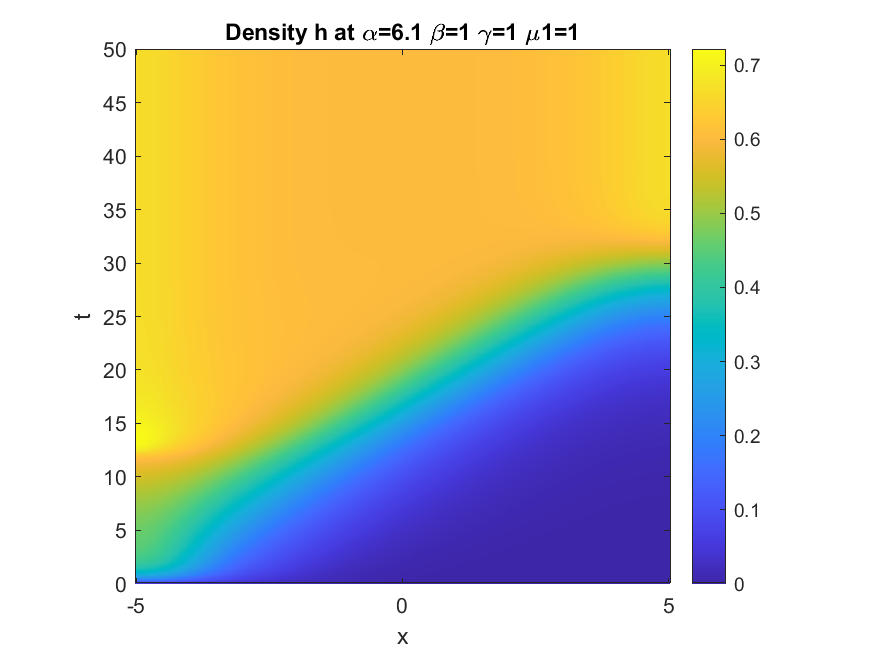}\ 
	\includegraphics[width=0.24\linewidth]{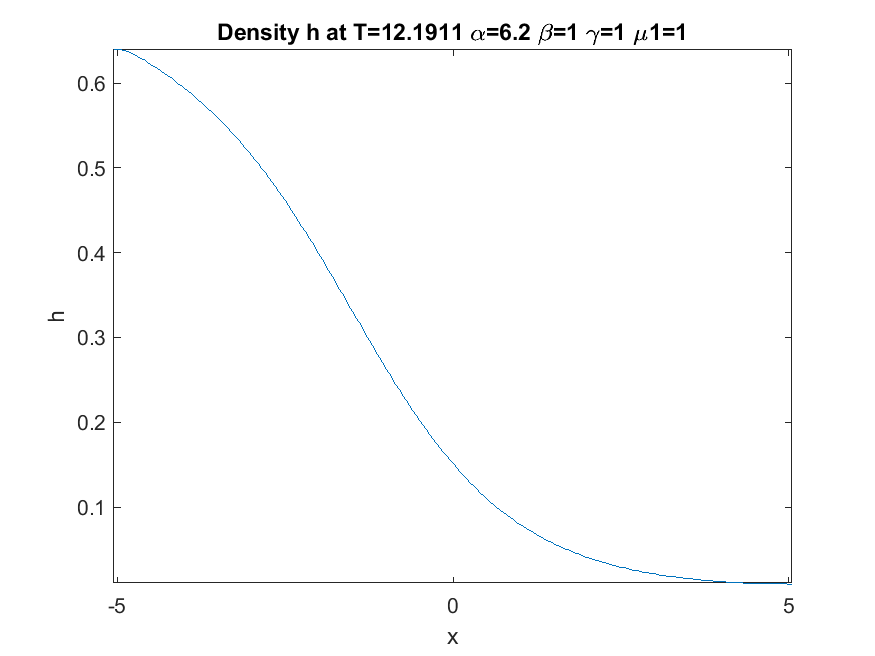}\\
	\caption{Simulation results of model \eqref{IBVP} with $J_1=J_2=J_L$, i.e. logistic kernels, $\beta=\gamma=\mu_1=1$, $\alpha =2,5,6.1,6.2$ (columns from left to right, respectively). Component $u$ of solution  starts to become unbounded near $\alpha =6.2$. In the rightmost column a blow-up occurs in the next time step.}\label{fig:2}
\end{figure}

\begin{figure}[h!]
	\includegraphics[width=0.24\linewidth]{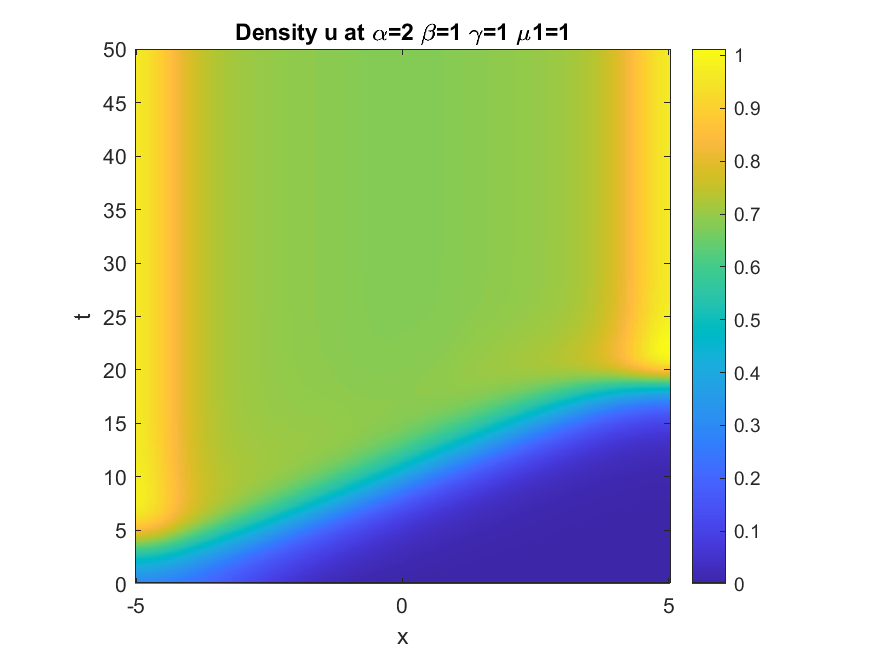}\  \includegraphics[width=0.24\linewidth]{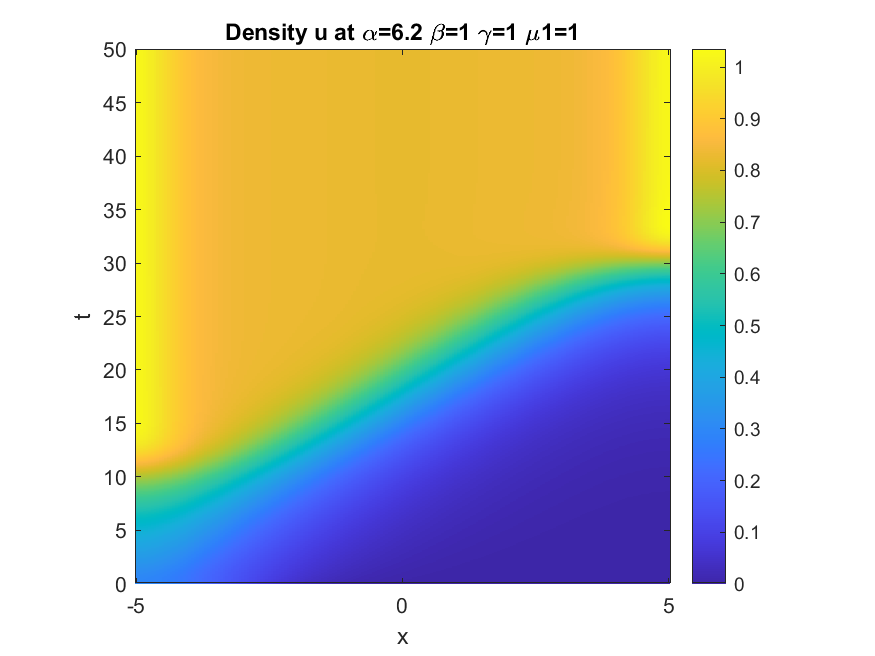}\  \includegraphics[width=0.24\linewidth]{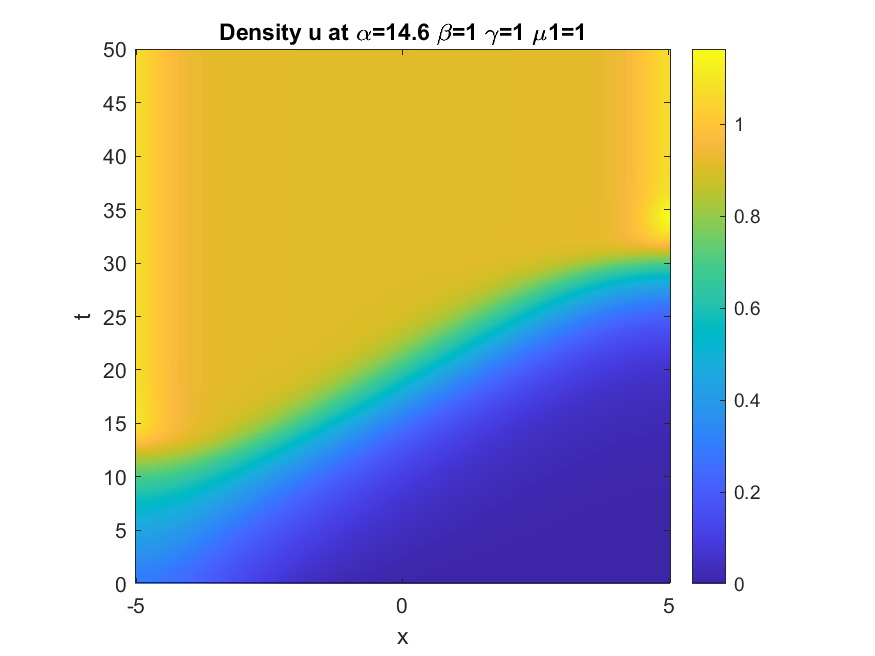}\ 
	\includegraphics[width=0.24\linewidth]{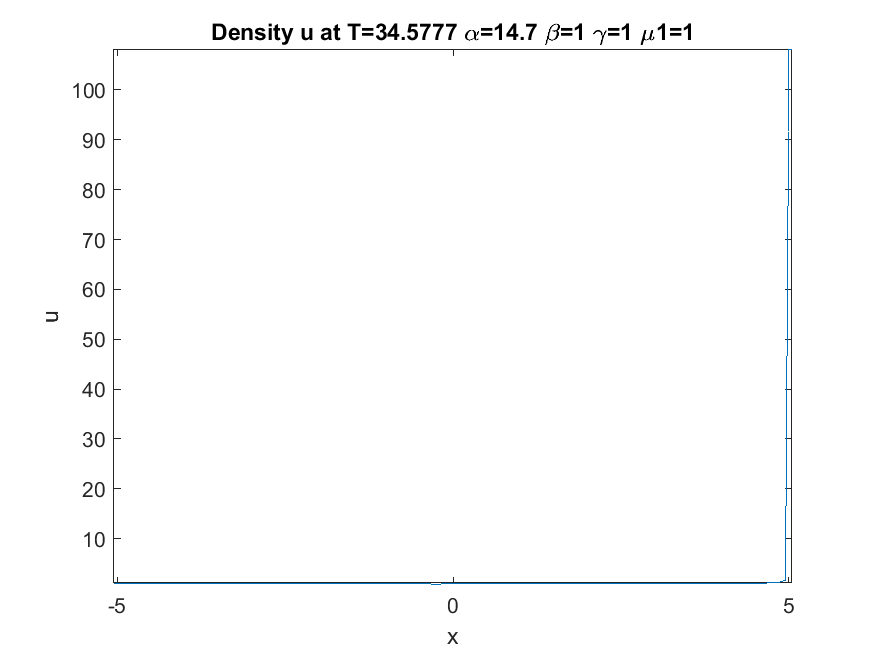}\\
	\includegraphics[width=0.25\linewidth]{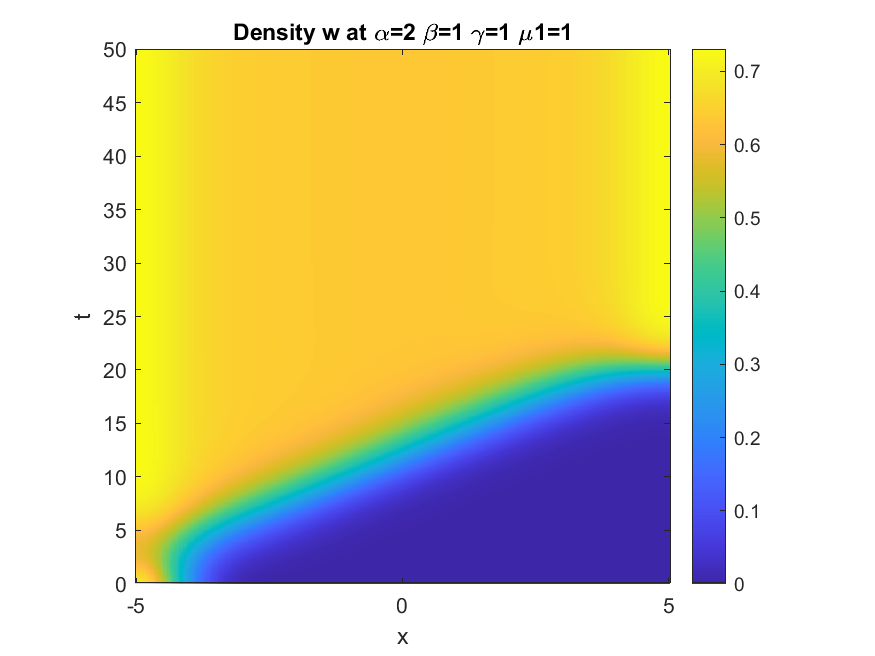}\  \includegraphics[width=0.24\linewidth]{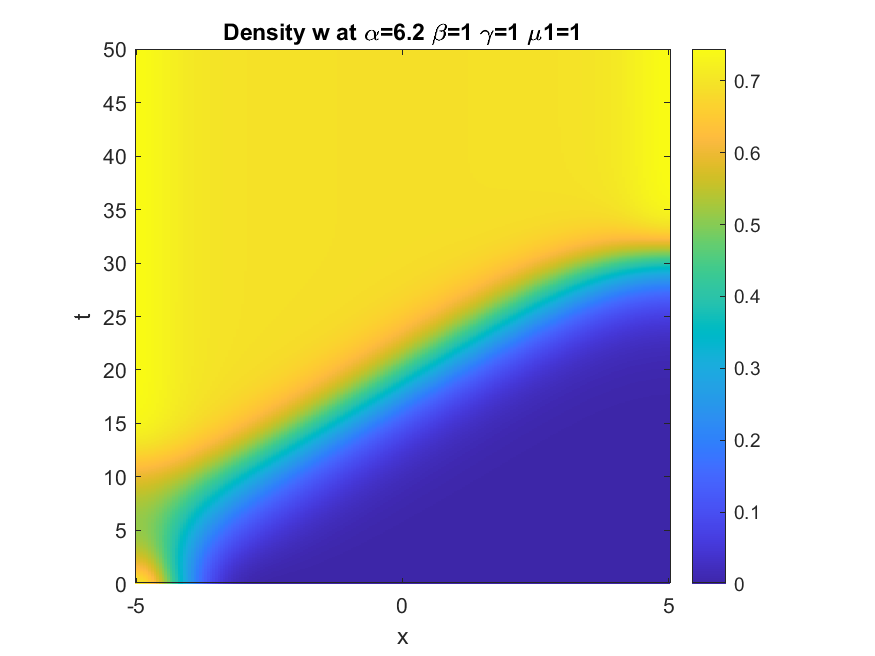}\  \includegraphics[width=0.24\linewidth]{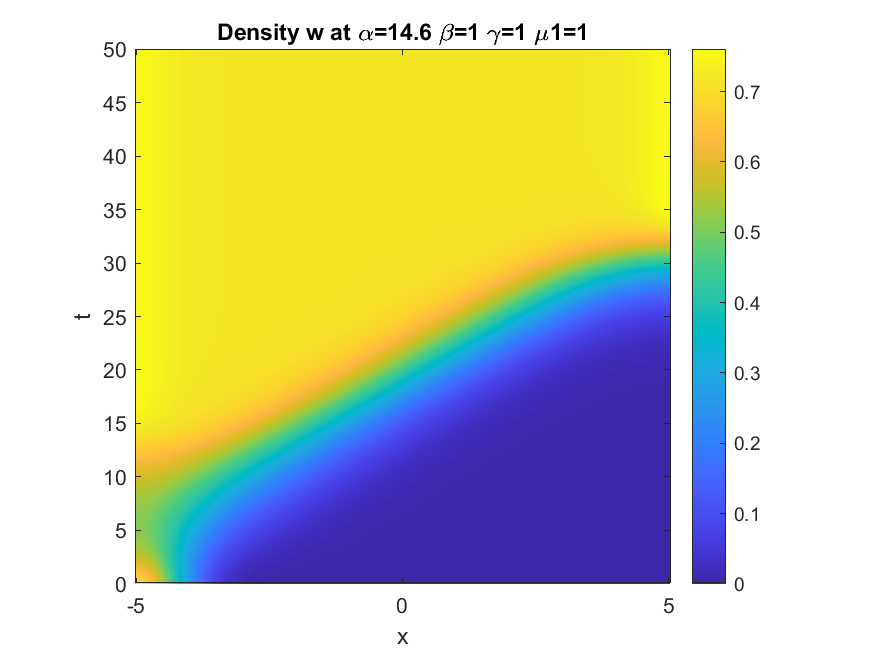}\ 
	\includegraphics[width=0.24\linewidth]{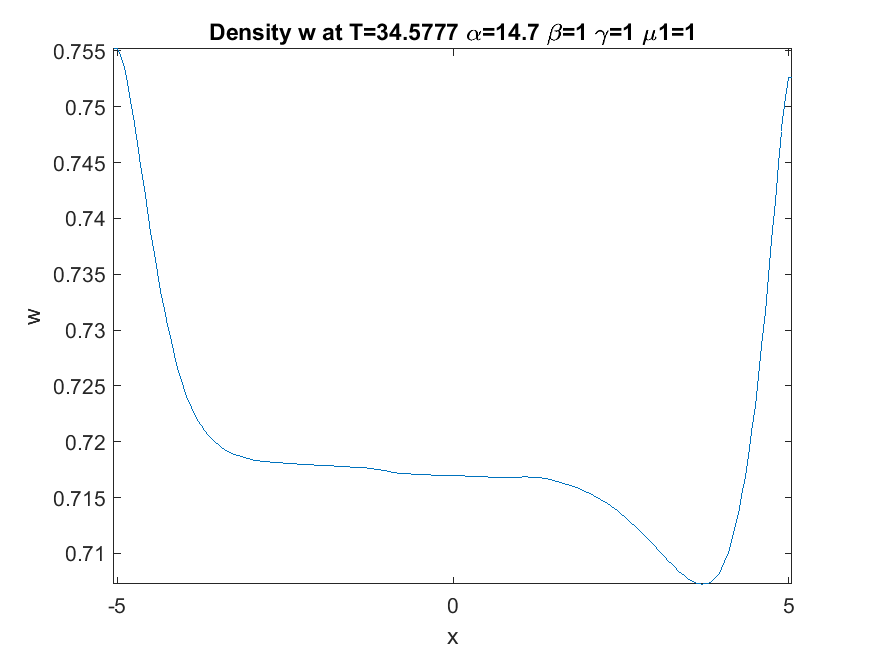}\\
	\includegraphics[width=0.25\linewidth]{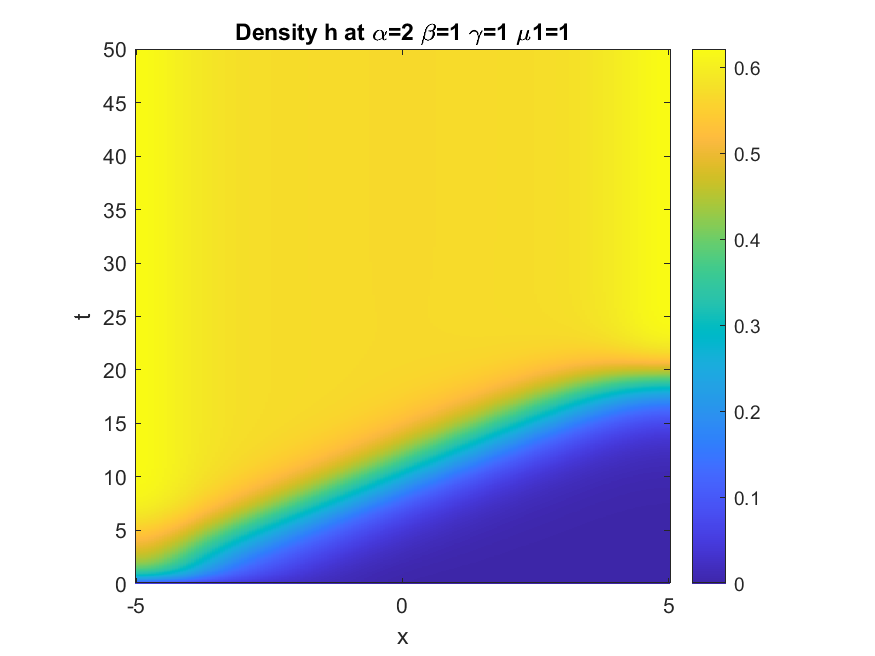}\  \includegraphics[width=0.24\linewidth]{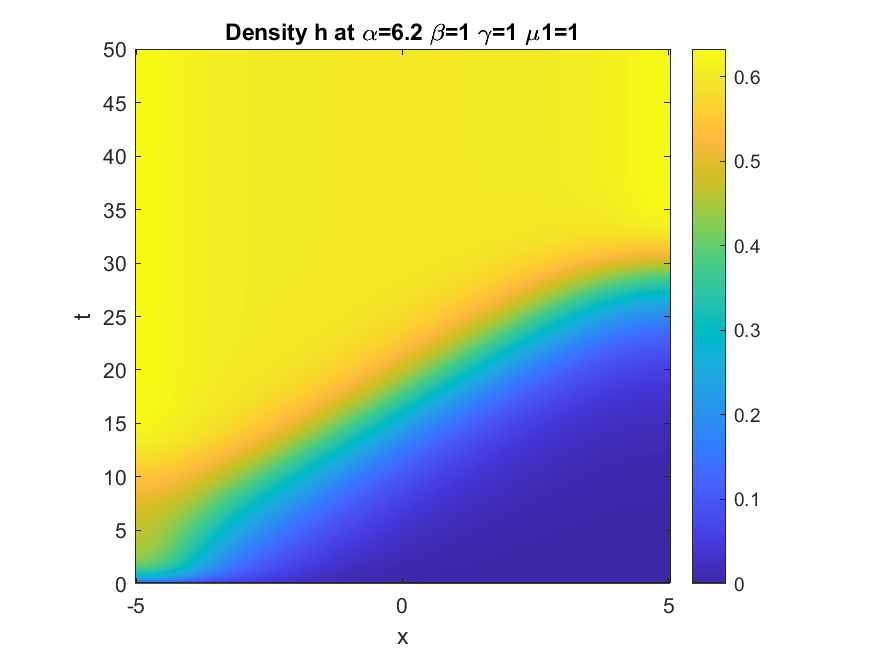}\  \includegraphics[width=0.24\linewidth]{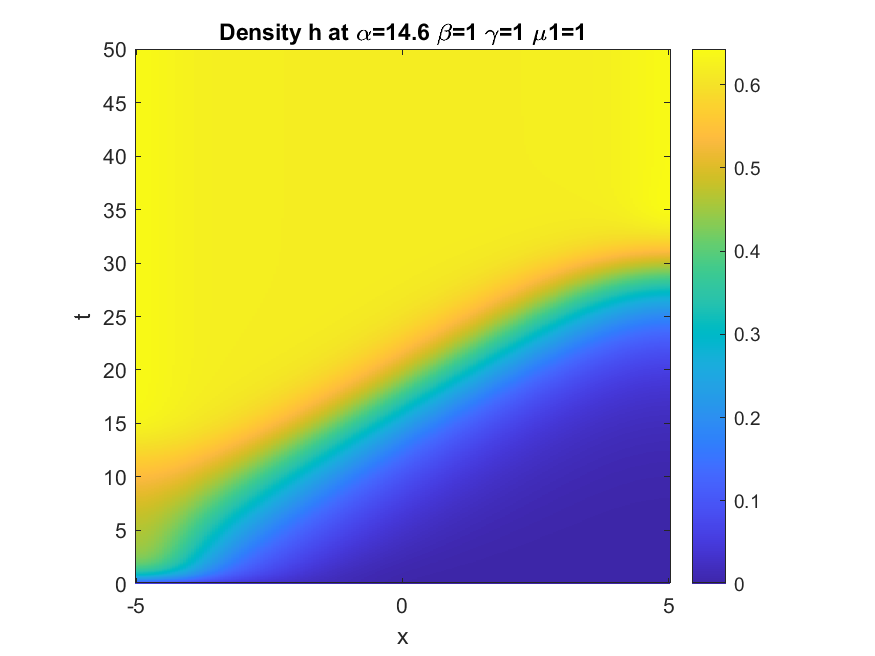}\ 
	\includegraphics[width=0.24\linewidth]{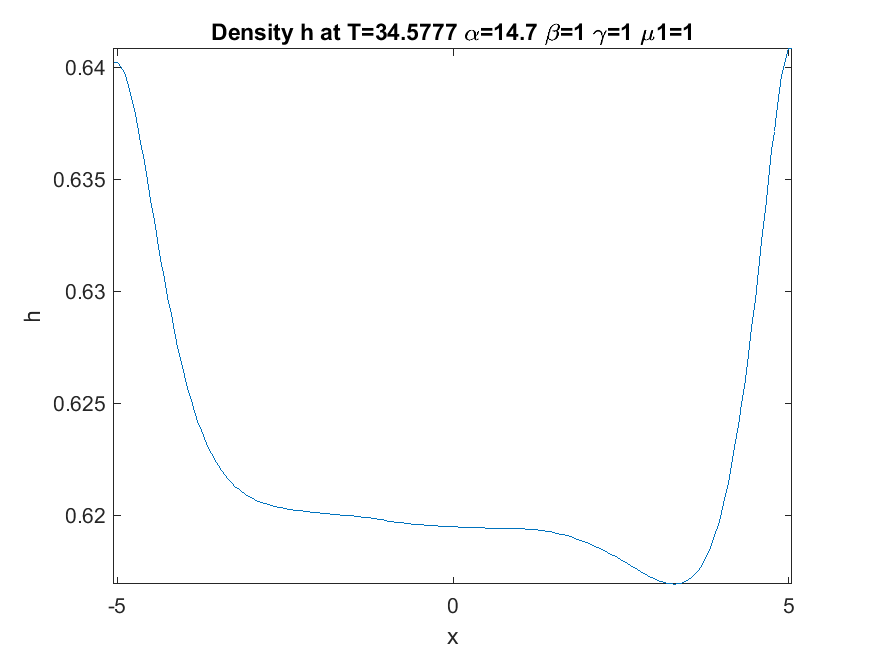}\\
	\caption{Simulation results of model \eqref{IBVP} with $J_1=J_2=J_U$, i.e.  uniform kernels, $\beta=\gamma=\mu_1=1$, $\alpha =2,6.2,14.6,14.7$ (columns from left to right, respectively). Component $u$ of solution  starts to become unbounded near $\alpha =14.7$. In the rightmost column a blow-up occurs in the next time step.} \label{fig:3}
\end{figure}

\noindent
The first columns of Figures \ref{fig:2} and \ref{fig:3} show the solution for the critical $\alpha$ from \eqref{bedalphbet}, when $J_1$ and $J_2$ are both logistic or uniform, respectively. The solution $u$ aggregates at the position of the initial accumulation of the active cells at the left boundary. In the case of two logistic kernels a stronger aggregation for increasing $\alpha$ values can be observed leading to a blow-up at the left boundary near $\alpha^* = 6.2$. On the other hand, in the case of two uniform kernels $u$ invades the whole domain and aggregates at the right boundary, leading to a blow-up there for approximately $\alpha^* = 14.7$. This invasive behavior can also be observed for all combinations of kernels and parameters $\alpha,\beta,\gamma,\mu$ as long as no blow-up at the left boundary occurs. An overview of the minimal values $\alpha^*$ depending on the kernels can be found in Table \ref{tab:1}.

\begin{table}[h!]
	\centering
	\begin{tabular}{r|c|r}
		& $\alpha^*$ & Figure\\
		\hline
		$J_1,J_2$ logistic & 6.2 & \ref{fig:2} \\
		$J_1$ logistic, $J_2$ uniform & 7.4 & \\
		$J_1$ uniform, $J_2$ logistic & 10.3 & \\
		$J_1,J_2$ uniform & 14.7 & \ref{fig:3}\\
		$J_1,J_2$ from \eqref{j1},\eqref{j2} & 4.1 & \ref{fig:5}
	\end{tabular}
	\caption{Minimal value $\alpha^*$ for which the solution ceases to exist for $\beta=\gamma=\mu_1=1$ depending on the kernels $J_1$ and $J_2$.}
	\label{tab:1}
\end{table}

\noindent
In further tests we investigated for logistic kernels the influence of $\beta$,$\gamma$ and the growth rate $\mu_1$ on the blow-up behavior. Higher values of $\beta$ lead to an increase of the minimal value $\alpha^*$ where blow-up occurs. In the case $\beta=10$ and $\gamma=\mu_1=1$ we observed that for $\alpha=26.9, 27.1, 27.3$ the solution ceases to exist, whereas  it exists globally in time for $\alpha = 27, 27.2, 27.4$. Hence, in contrast to \cite{EckardtSu,LiChSu} we cannot determine a value $\alpha^*$ s.t. for $\alpha<\alpha^*$ the solution is global, whereas it blows-up for $\alpha \geq \alpha^*$. It seems that for $\alpha\geq \alpha^{**}=27.5$ blow-up occurs but we cannot assure this. In contrast, higher values of $\gamma$ and/or  $\mu$ lead to a blow-up for lower $\alpha$'s, see Table \ref{tab:2} for an overview of the concrete values of $\alpha^*$ along with the respective parameter combinations.\\[-2ex]

\begin{table}[h!]
	\centering
	\begin{tabular}{r|c}
		Parameters & $\alpha^*$ \\
		\hline
		$\beta=10$, $\gamma=\mu_1=1$ & 26.9 \\
		$\beta=\gamma=10$, $\mu_1=1$ & 22.1  \\
		$\beta=10$, $\gamma = 0.1$, $\mu_1=1$ & 33.4 \\
		$\beta=1$, $\gamma=10$, $\mu_1=1$ & 4.6 \\
		$\beta=1$, $\gamma = 0.1$, $\mu_1=1$ & 25.3 \\
		$\beta = \gamma =1$, $\mu_1=10$ & 3.6
	\end{tabular}
	\caption{Minimal value $\alpha^*$ for which the solution ceases to exist, depending on parameters $\beta,\gamma,\mu_1$. Both convolution kernels are logistic: $J_1=J_2=J_L$.\\}\label{tab:2}
\end{table}

\begin{figure}[h!]
	\includegraphics[width=0.24\linewidth]{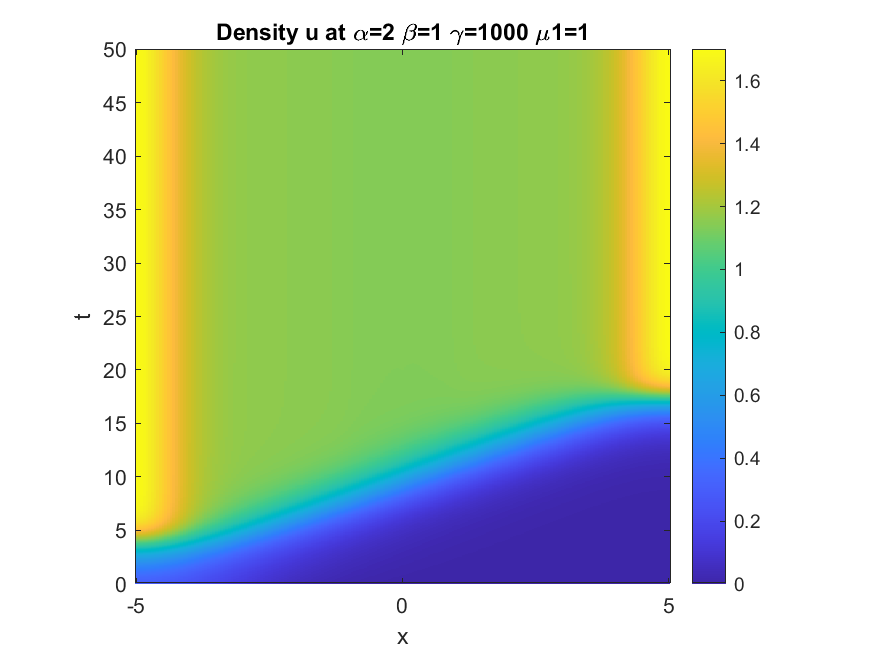}\  \includegraphics[width=0.24\linewidth]{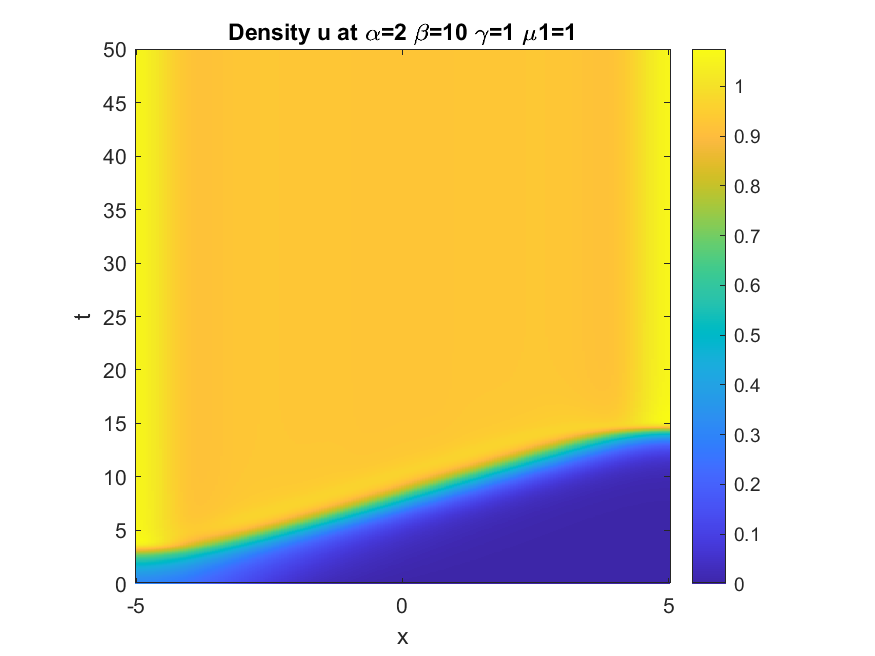}\  \includegraphics[width=0.24\linewidth]{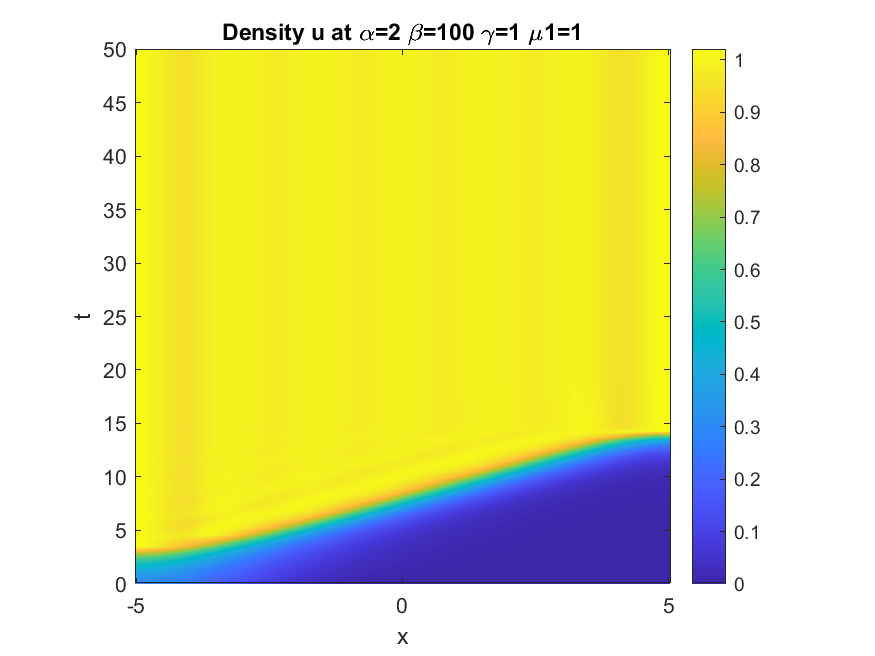}\ 
	\includegraphics[width=0.24\linewidth]{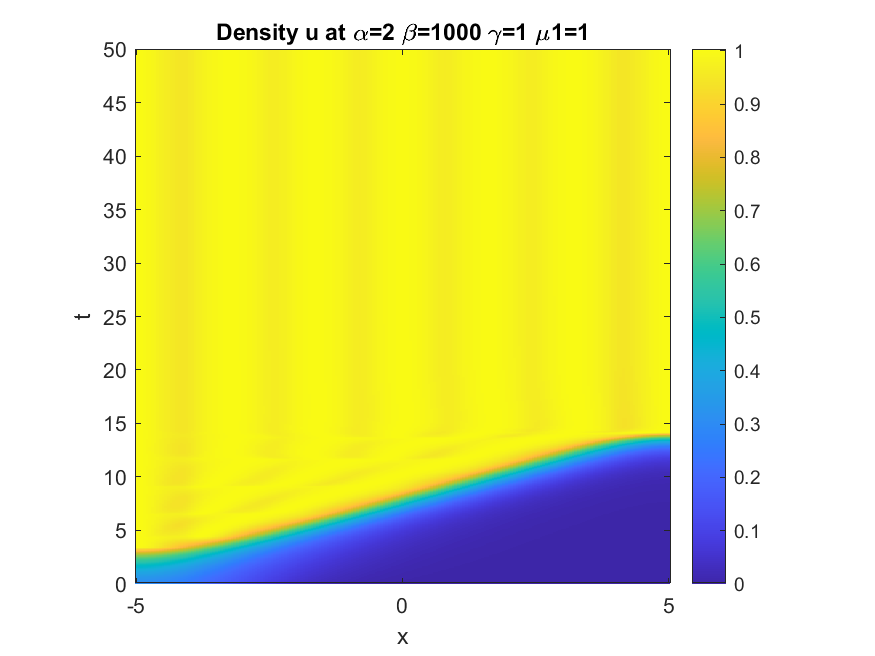}\\
	\includegraphics[width=0.24\linewidth]{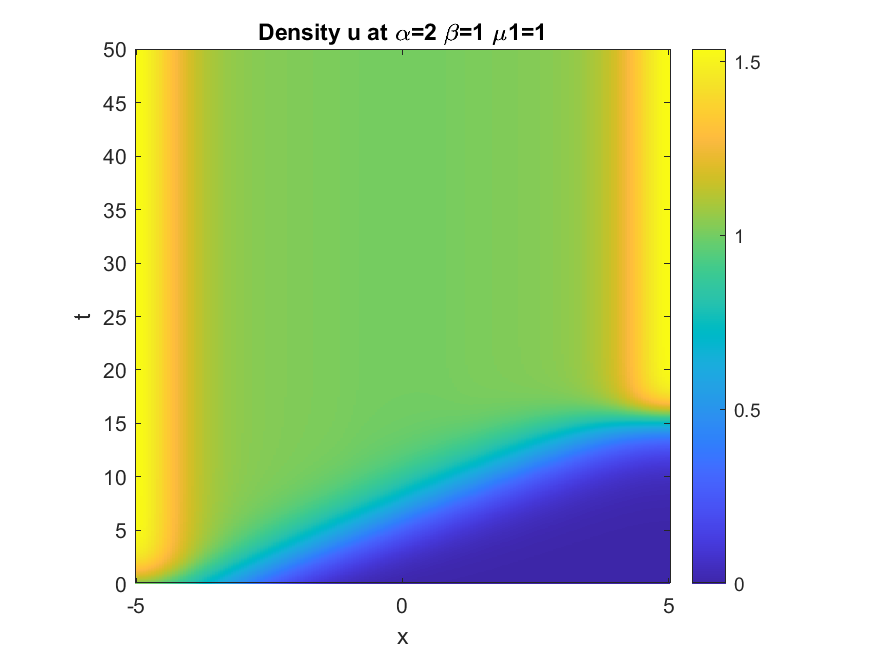}\  \includegraphics[width=0.24\linewidth]{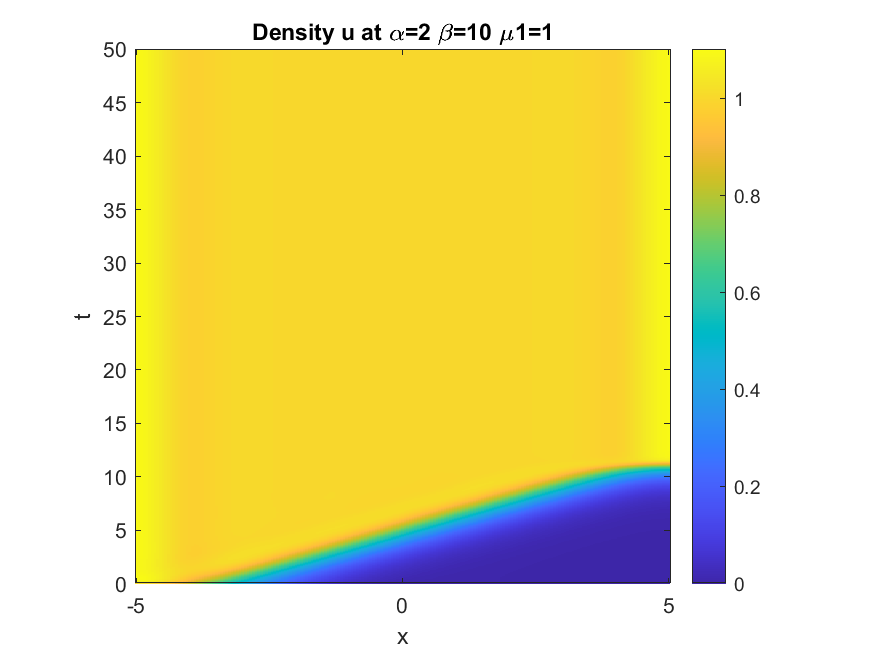}\  \includegraphics[width=0.24\linewidth]{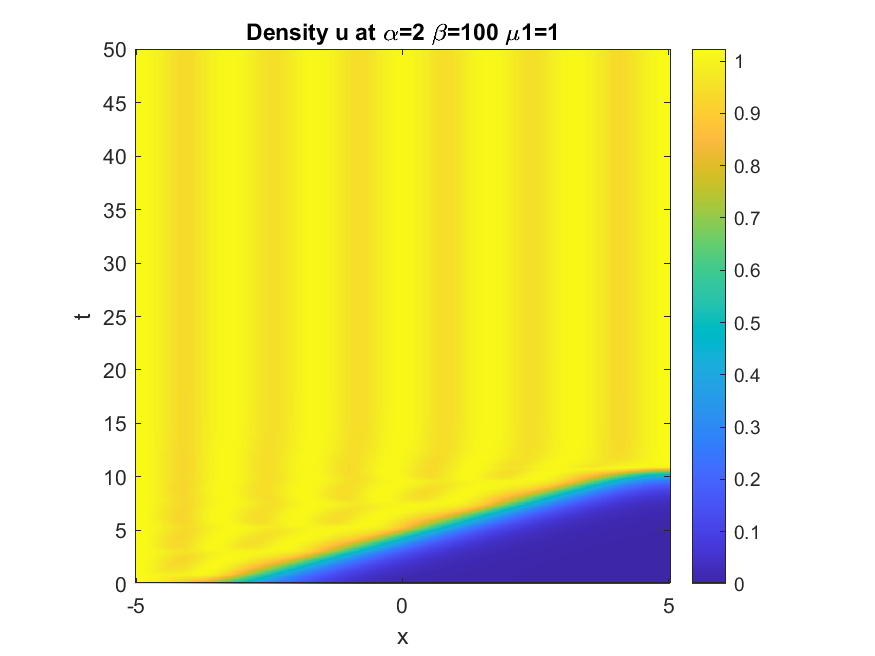}\ 
	\includegraphics[width=0.24\linewidth]{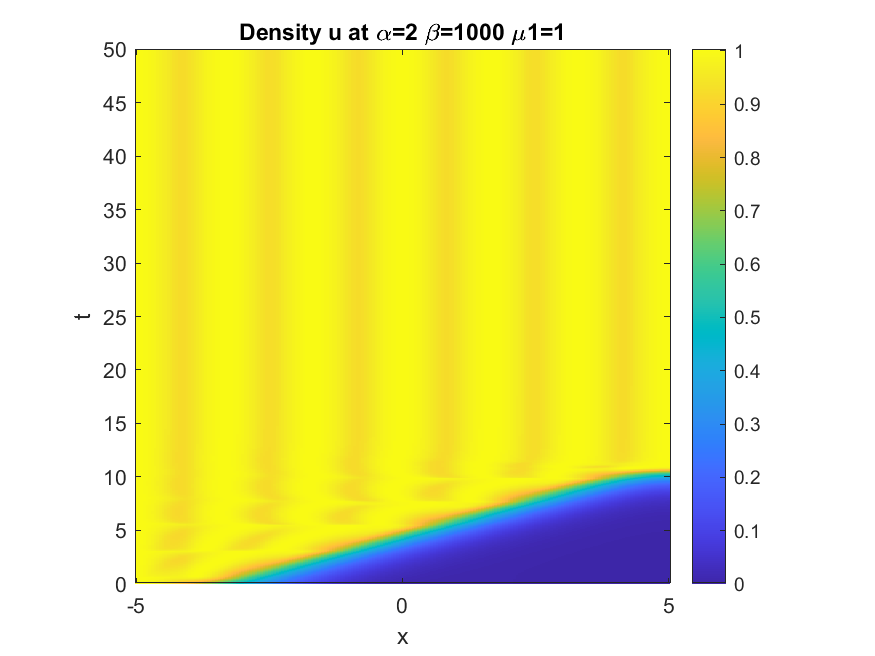}\\
	\includegraphics[width=0.24\linewidth]{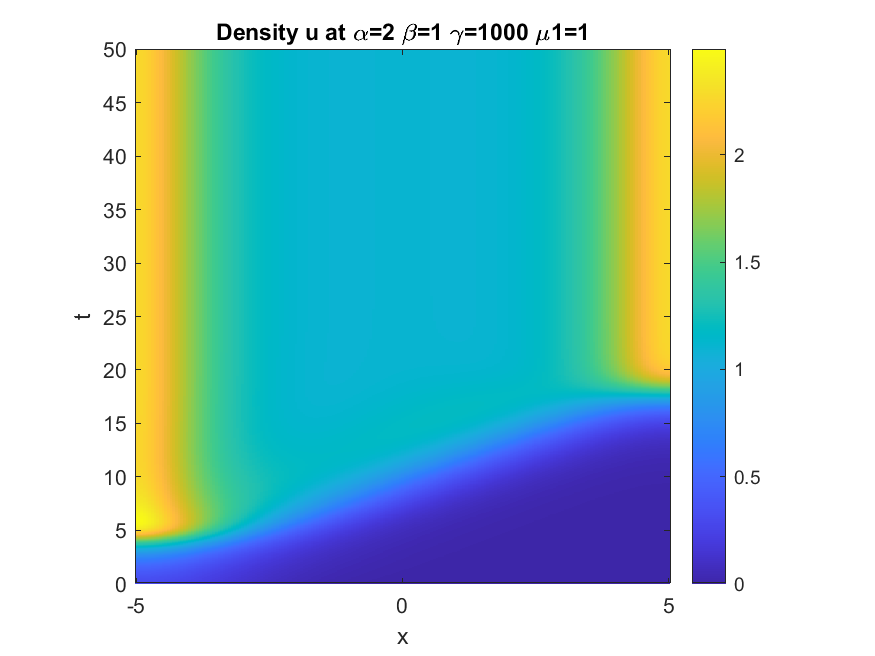}\  \includegraphics[width=0.24\linewidth]{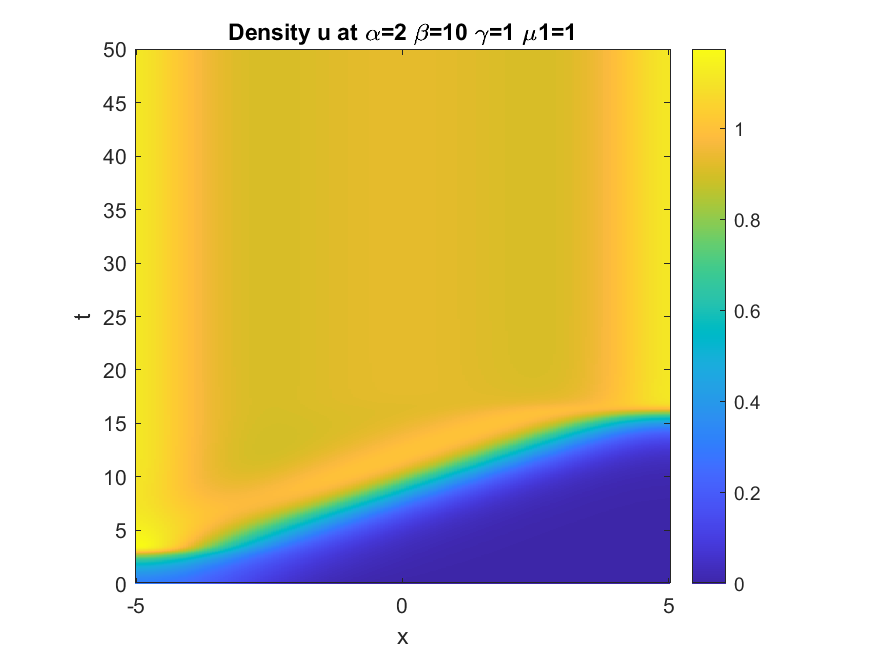}\  \includegraphics[width=0.24\linewidth]{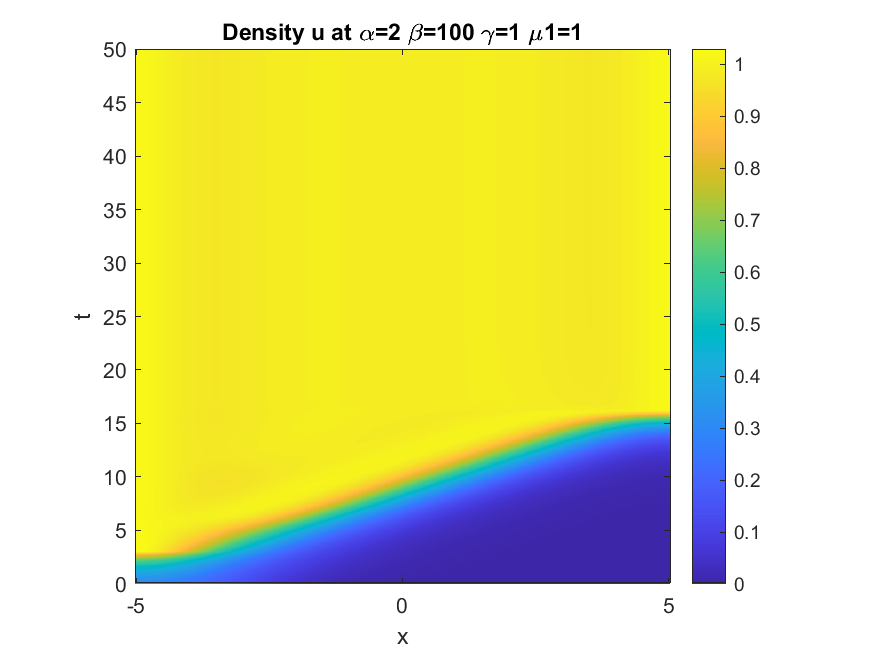}\ 
	\includegraphics[width=0.24\linewidth]{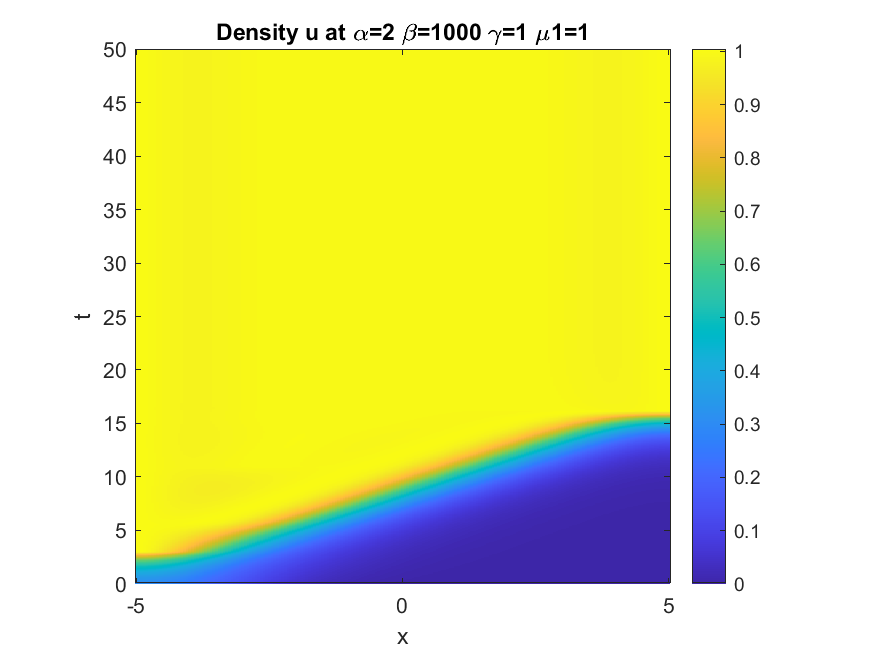}\\
	\includegraphics[width=0.24\linewidth]{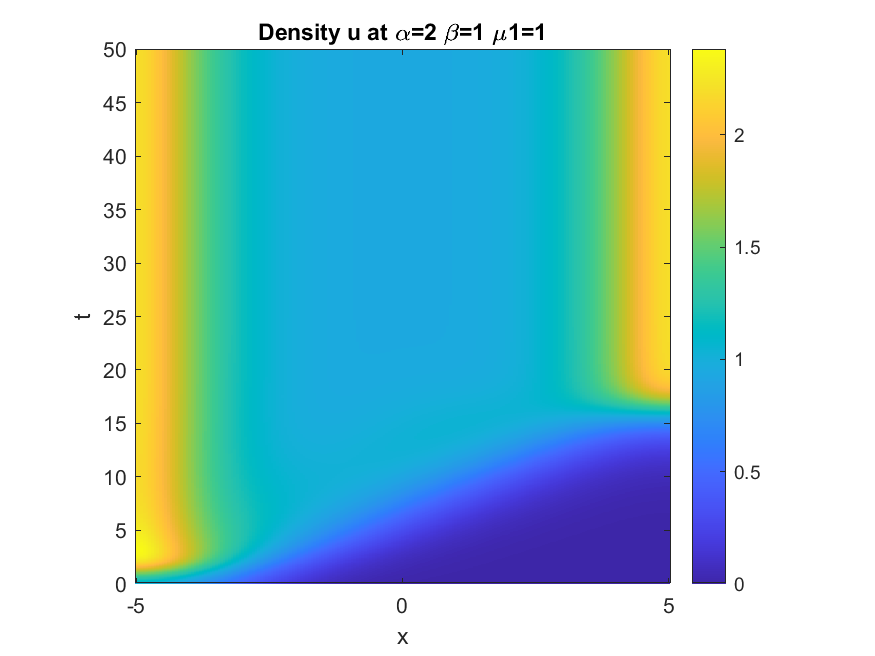}\  \includegraphics[width=0.24\linewidth]{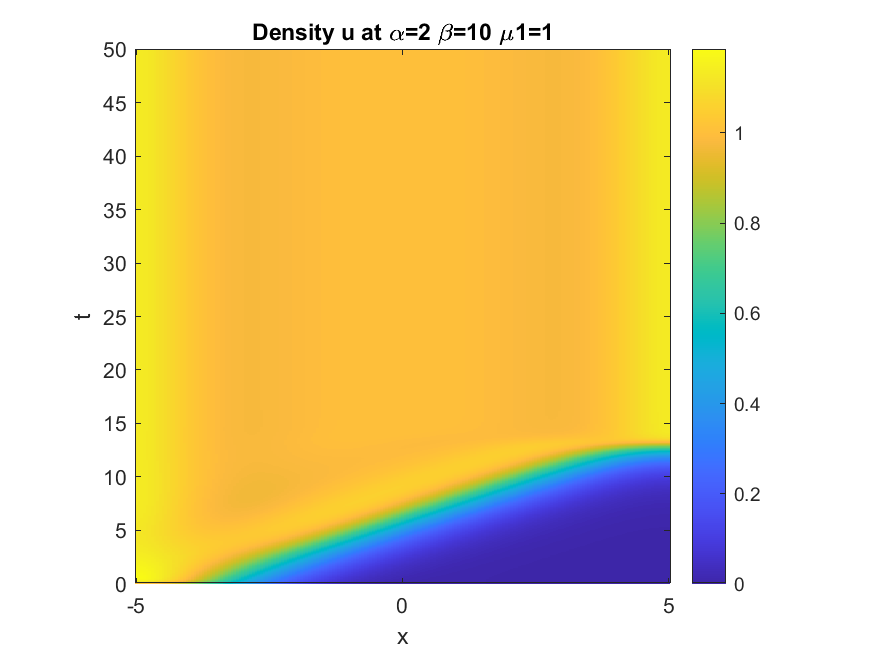}\  \includegraphics[width=0.24\linewidth]{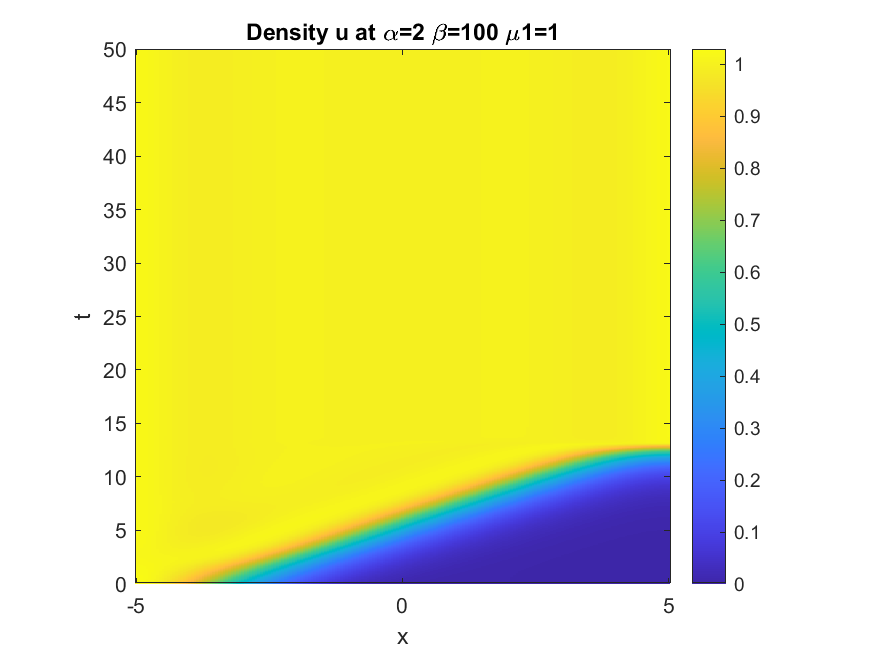}\ 
	\includegraphics[width=0.24\linewidth]{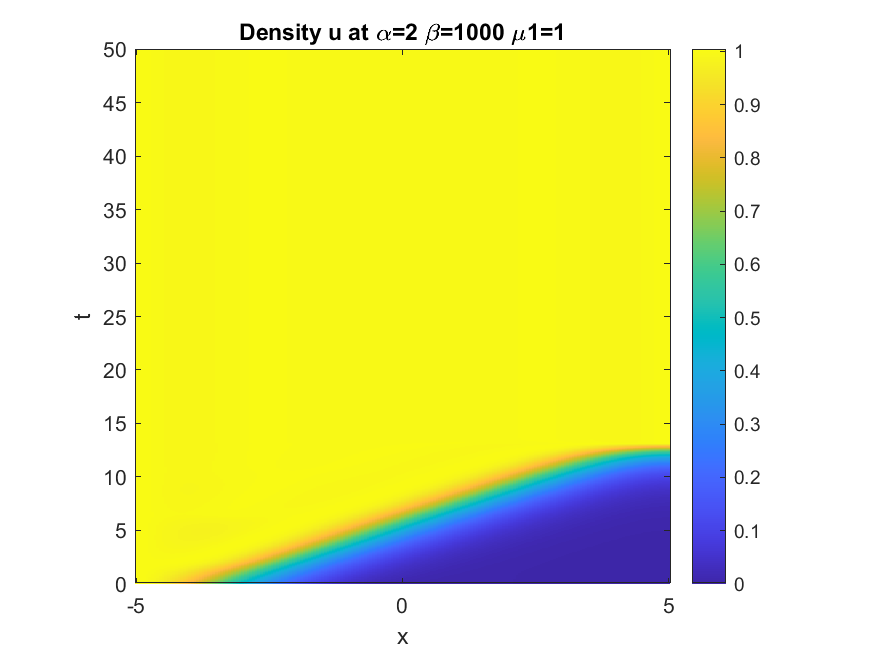}
	\caption{{ Simulations of models \eqref{IBVP} and \eqref{modelww} with $\alpha =2$, $\mu_1=1$, and (from left to right) $\beta =1,10, 100,1000$ and  $\gamma =1000$ in column 1, $\gamma =1$ in columns 2-4. First row: $J_1=J_2=J_U$; 2nd row: model without $w$, with $J=J_U$; 3rd row:  $J_1=J_2=J_L$; 4th row: model without $w$, with $J=J_L$}}.  \label{fig:6}
\end{figure}

\noindent
Moreover, as in \cite{EckardtSu}, increasing values of $\beta$ and $\mu_1$ in Figure \ref{fig:6} lead to patterns depending on the kernels $J_1$ and $J_2$. A high value of $\gamma$ does not seem to lead by itself to patterns, but further experiments suggest that $\gamma$ influences the height of the peaks, thus leading to less pronounced $u$-patterns. This is due to the stronger dampening of proliferation, which hinders stronger aggregates. To illustrate the effect of $\gamma $ we plot in Figure \ref{fig:6} two situations with very different values { ($\gamma =1000$ in the first column and $\gamma =1$ in the remaining columns)}.\\[-2ex] 

\noindent
The performed simulations are very similar to those of the reduced model 
\begin{align}\label{modelww}
	\begin{cases}
		\partial_t u = \nabla \cdot \left( \psi(h)\nabla u\right) + \mu_1 \ua \left(1- J(x,h)\ast \ub \right), & x\in \Omega,\ t>0,\\
		\partial_t h = D_H \Delta h + g(u) - \lambda h, & x \in \Omega,\ t>0
	\end{cases}
\end{align}
without inactive cells $w$ (compare rows 1 and 2 and rows 3 and 4 in Figure \ref{fig:6}, respectively, for uniform or logistic kernels).
System \eqref{modelww} is a simplification of the model considered in \cite{EckardtSu} without myopic diffusion and taxis, where Turing-like patterns for large values of $\beta \mu_1$ were induced by the nonlocal term. This also seems to be the case here in model \eqref{IBVP}. However, the calculation of a strictly positive steady-state $(u^*,w^*,h^*)$ already leads to analytical problems, since this requires even in the corresponding local model without diffusion and for $\mu_2,\mu_3,\mut$ independent from $h$ a solution to the nonlinear system 
\begin{align*}
	0 &= \mu_1 \left(u^*\right)^{\alpha} \left(1- \left(u^*\right)^{\beta} -  \left(w^*\right)^{\gamma} \right) + \mut \frac{w^*}{1+w^*},\\
	0 &= \mu_2(1-w^*)u^* - \mu_3\frac{w^*}{1+w^*}.
\end{align*}

\begin{figure}[h!]
	\includegraphics[width=0.24\linewidth]{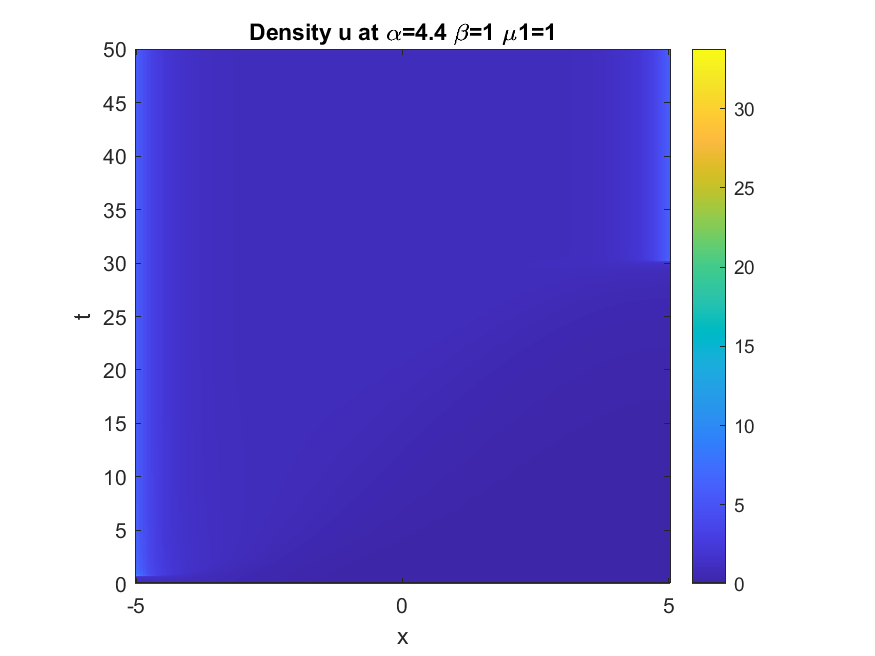}\  \includegraphics[width=0.24\linewidth]{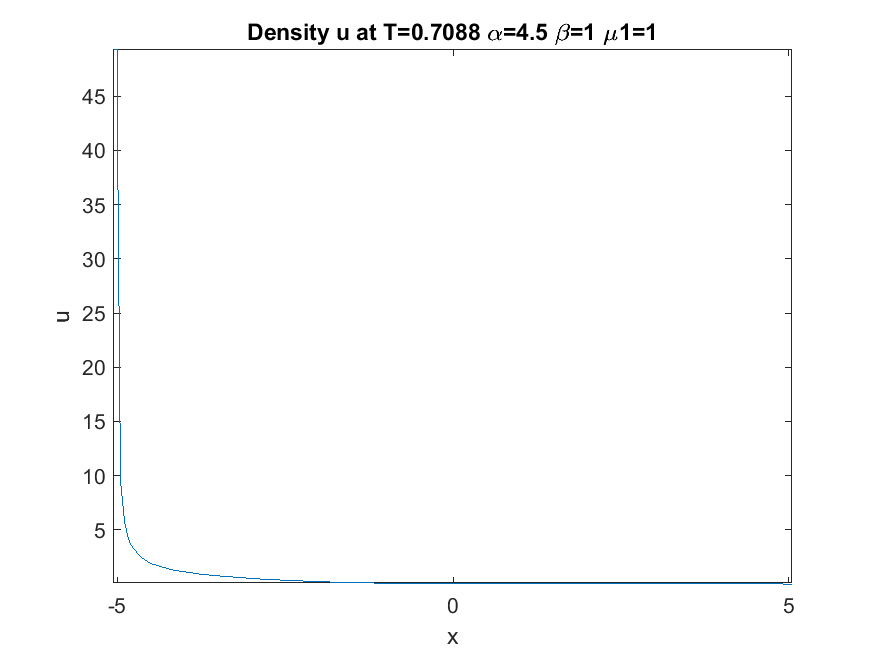}\
	\includegraphics[width=0.24\linewidth]{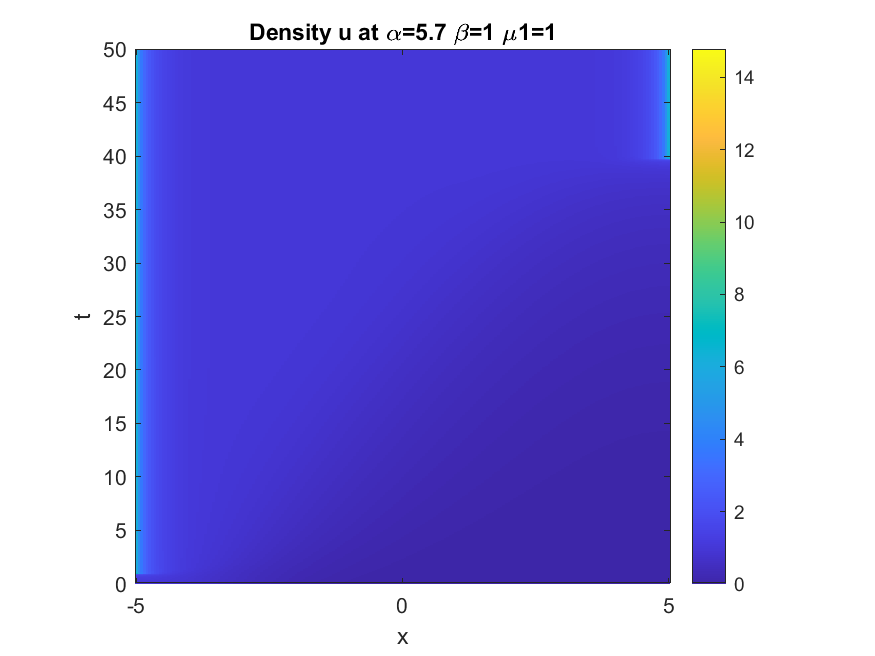}\ 
	\includegraphics[width=0.24\linewidth]{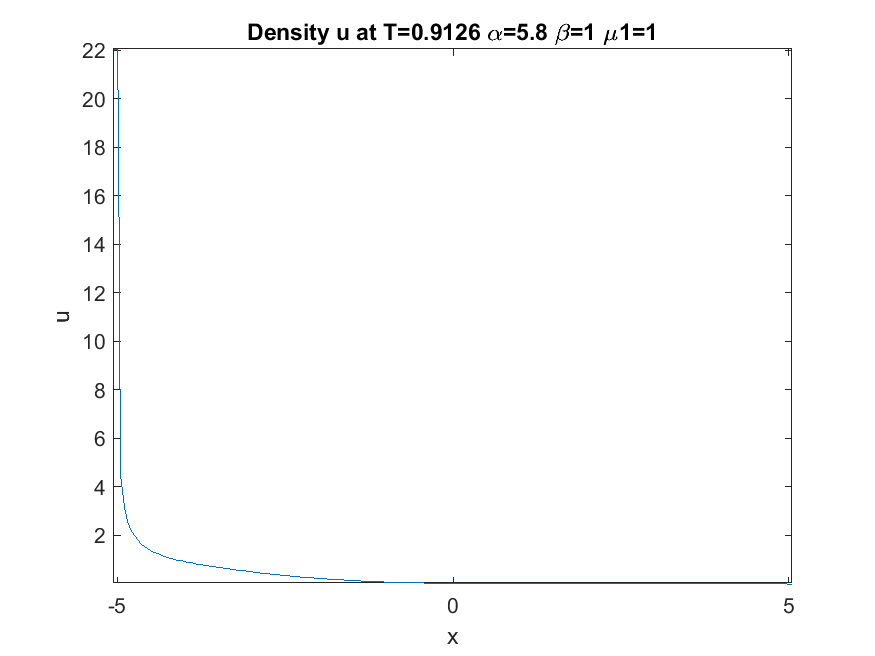}\\
	\includegraphics[width=0.24\linewidth]{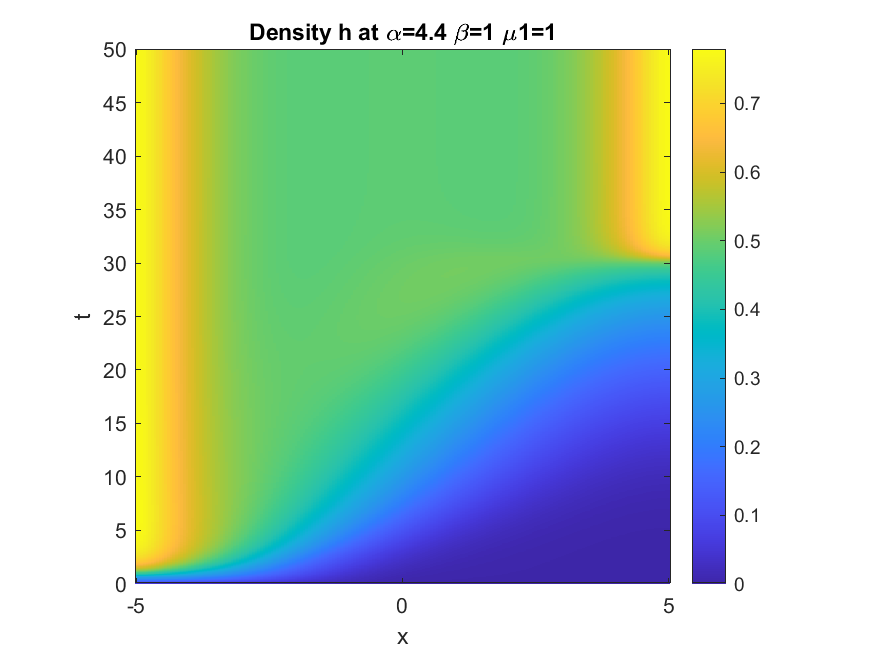}\  \includegraphics[width=0.24\linewidth]{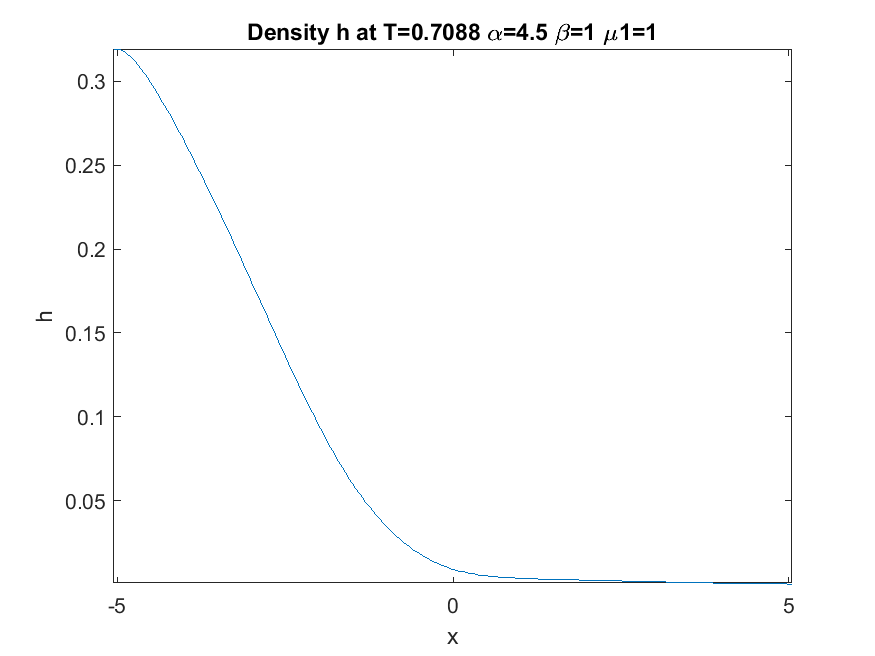}\
	\includegraphics[width=0.24\linewidth]{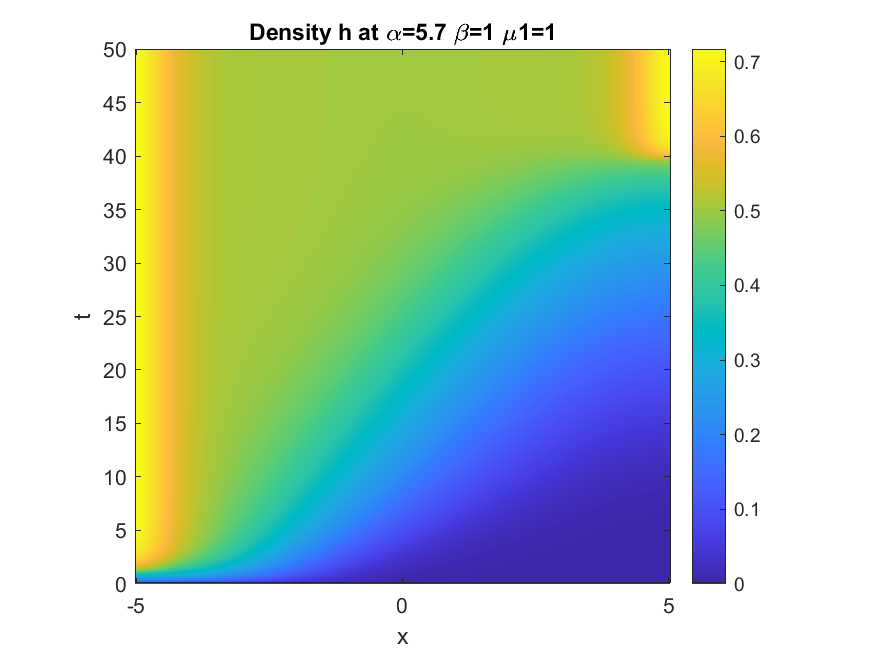}\ 
	\includegraphics[width=0.24\linewidth]{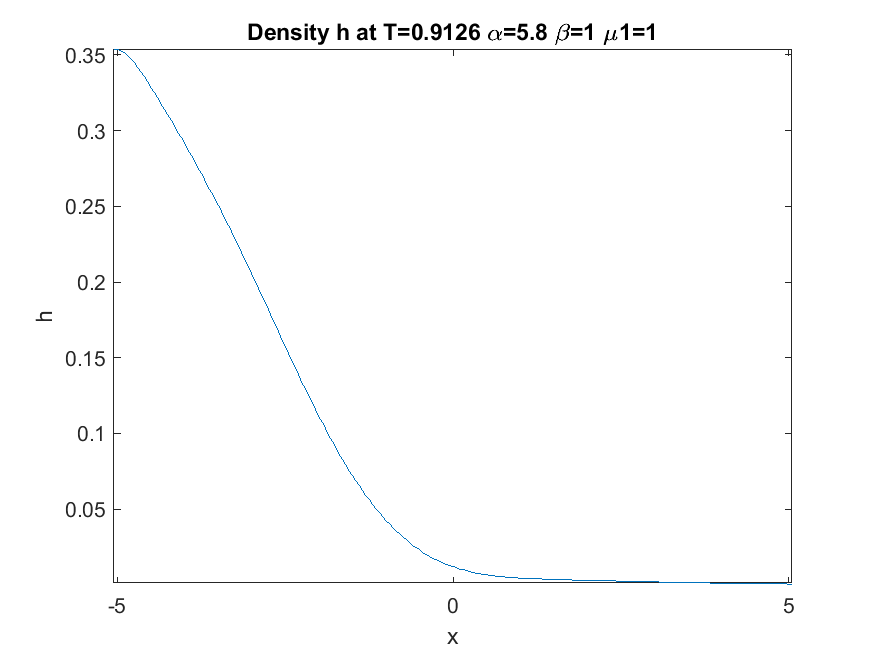}
	\caption{Simulations of model \eqref{modelww} without $w$ with $\beta = \mu = 1$. Columns 1 and 2: $J$ logistic and $\alpha = 4.4,4.5$. Columns 3 and 4: $J$ uniform and $\alpha = 5.7,5.8$. In the 2nd and 4th column a blow-up occurs in the next time step.}\label{fig:4}
\end{figure}

\begin{figure}[htb]
	\includegraphics[width=0.24\linewidth]{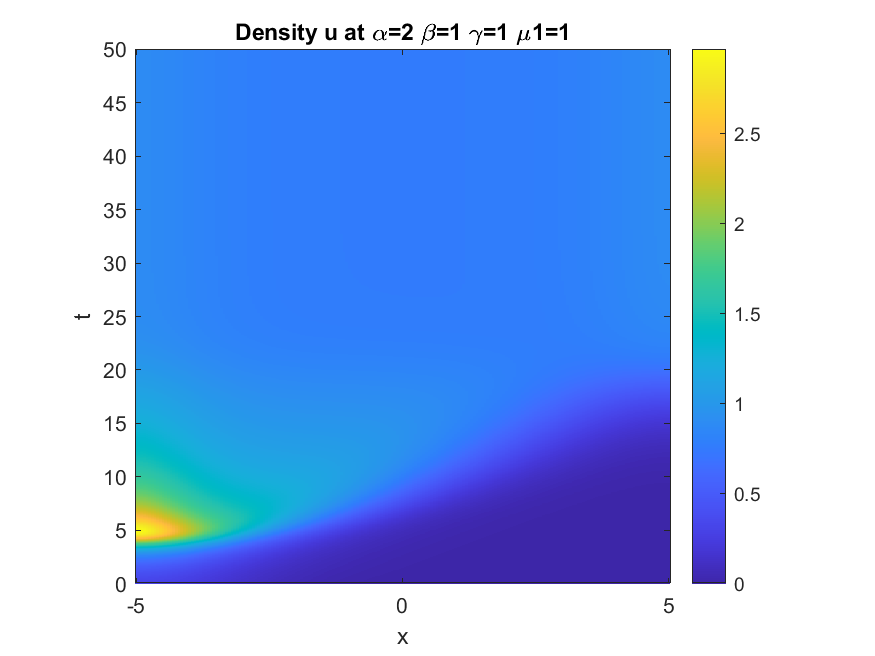}\  \includegraphics[width=0.24\linewidth]{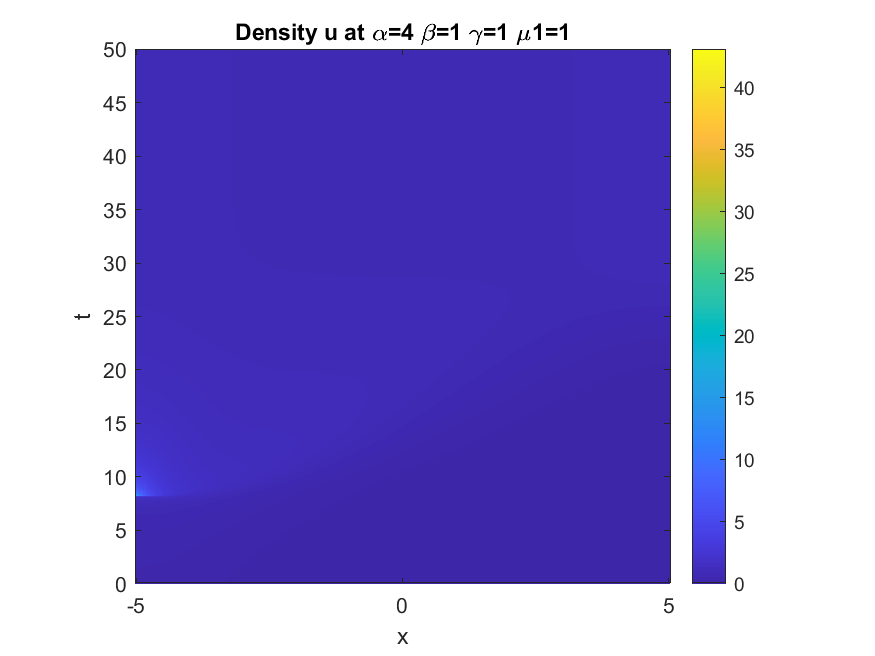}\  \includegraphics[width=0.24\linewidth]{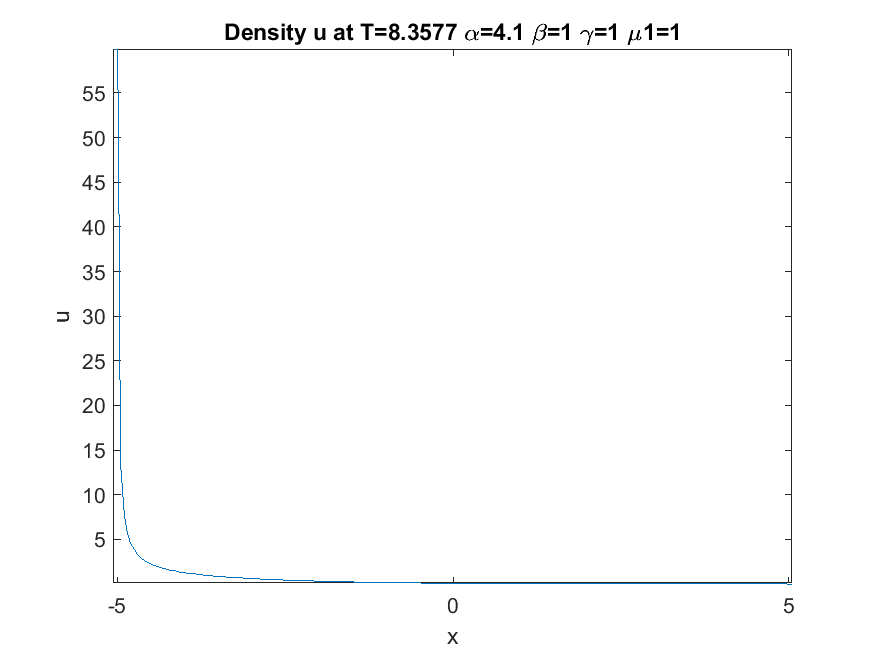}\ 
	\includegraphics[width=0.24\linewidth]{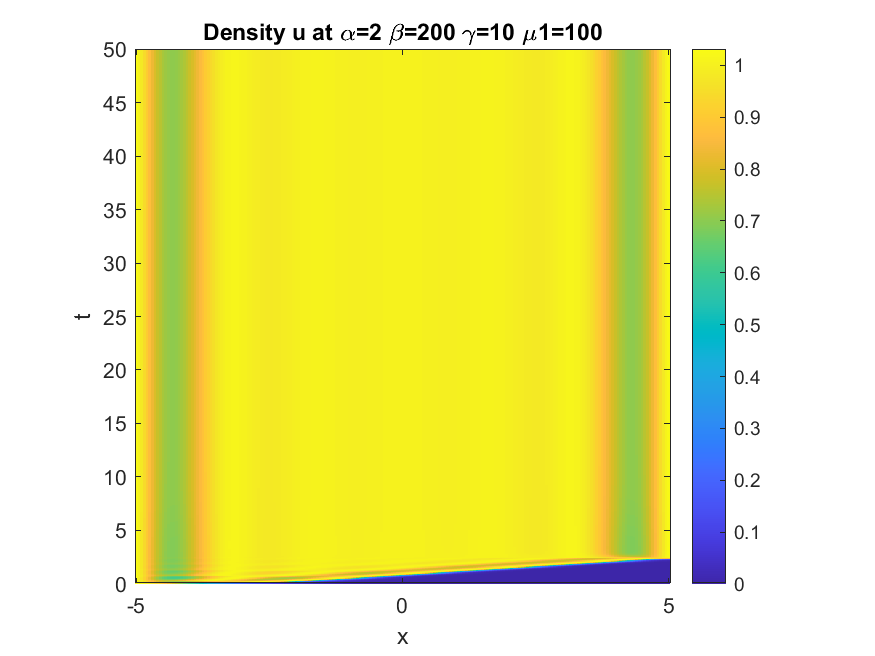}\\
	\includegraphics[width=0.24\linewidth]{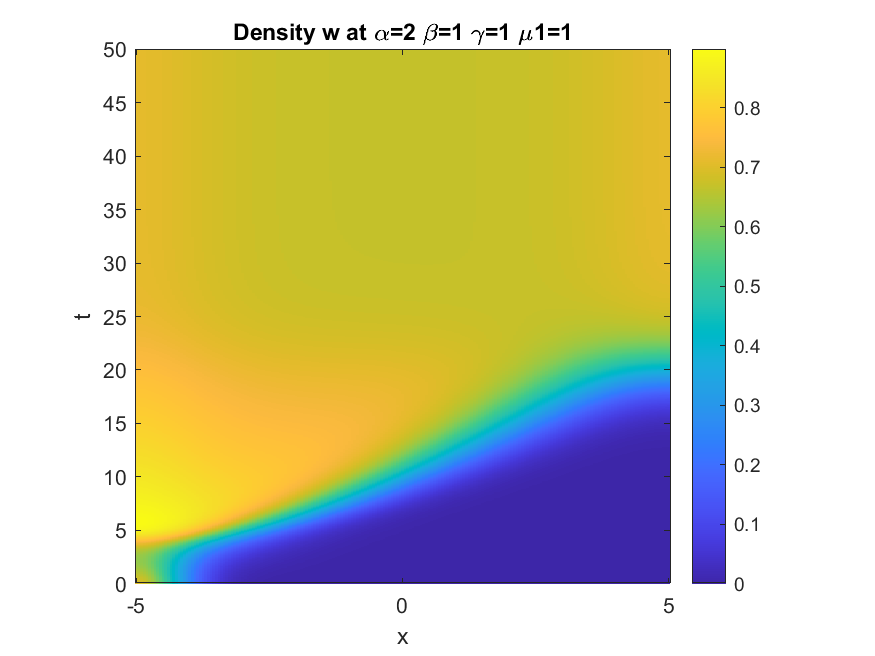}\  \includegraphics[width=0.24\linewidth]{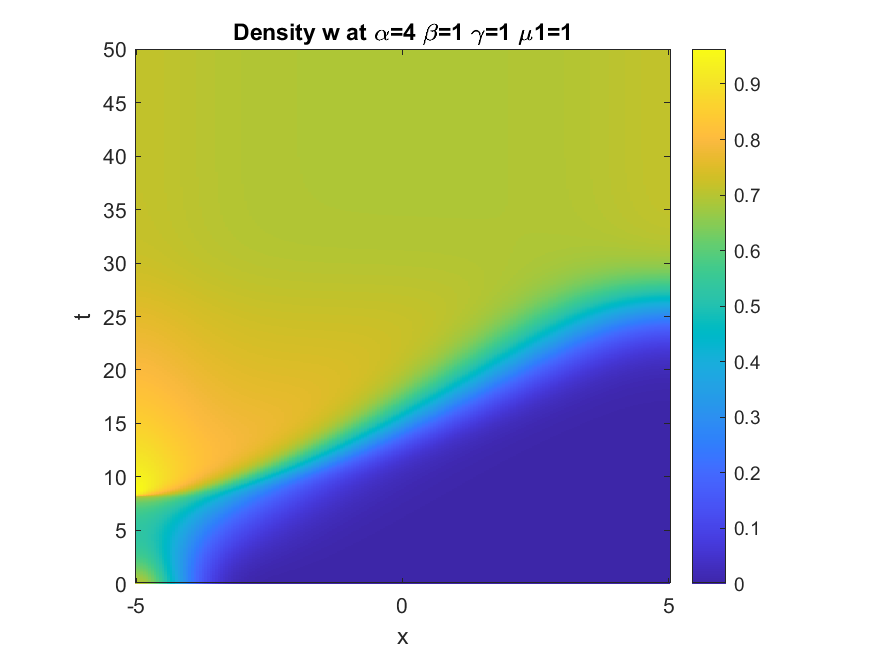}\  \includegraphics[width=0.24\linewidth]{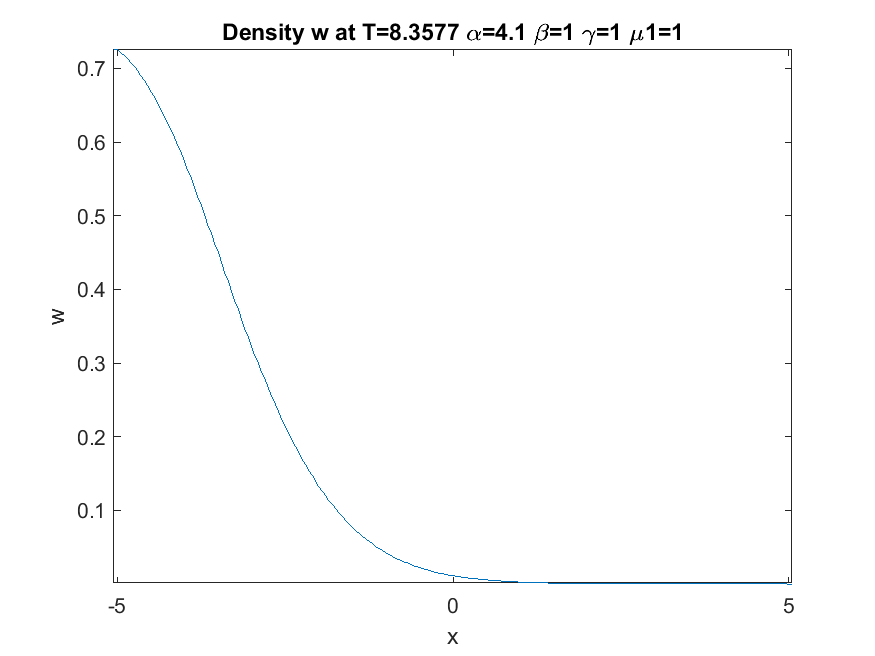}\ 
	\includegraphics[width=0.24\linewidth]{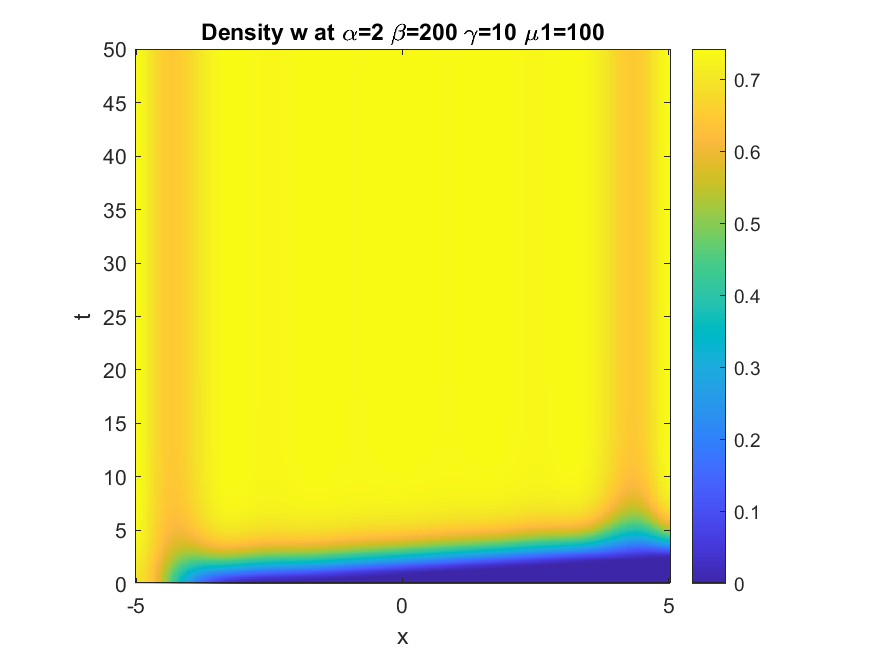}\\	\includegraphics[width=0.24\linewidth]{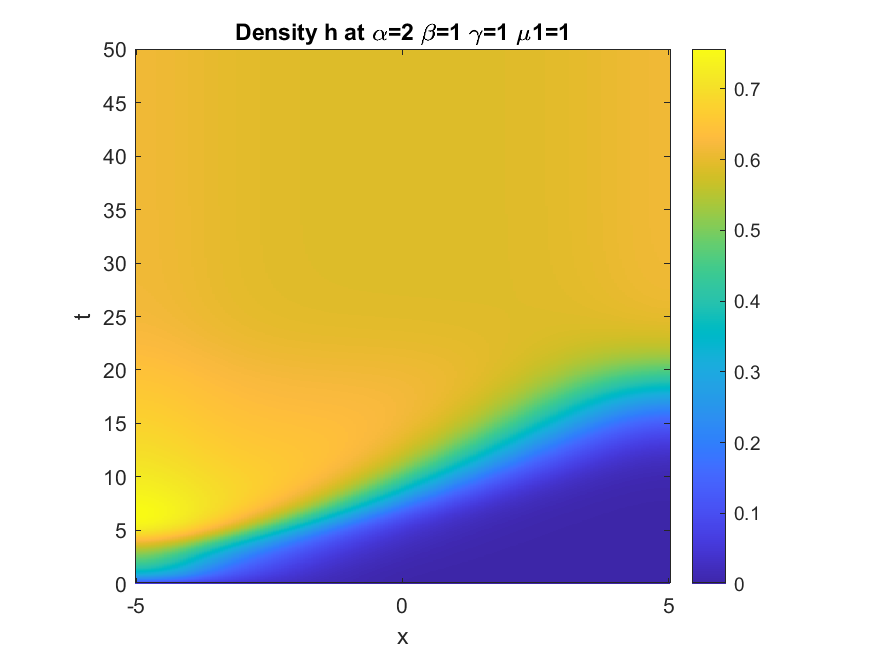}\  \includegraphics[width=0.24\linewidth]{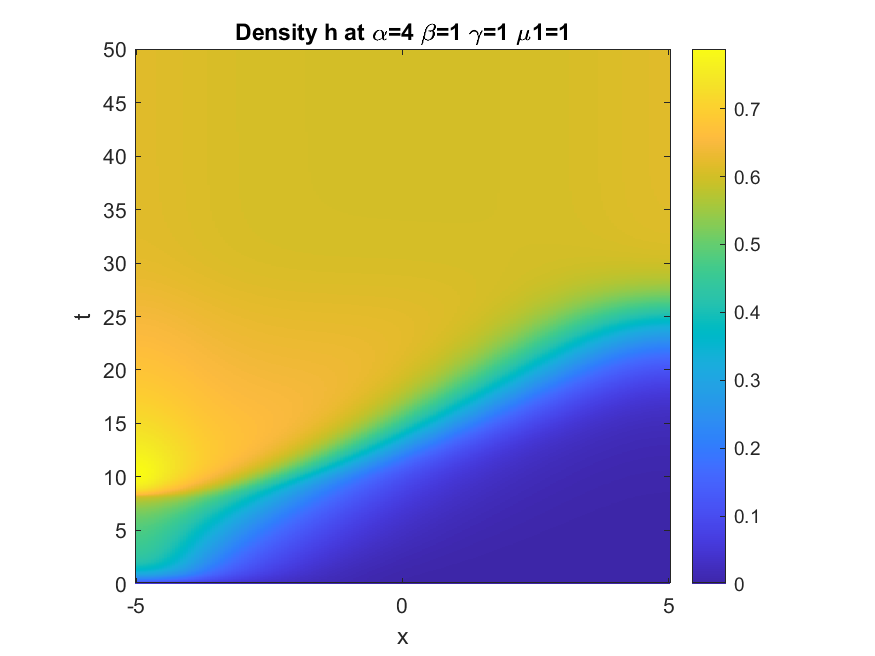}\  \includegraphics[width=0.24\linewidth]{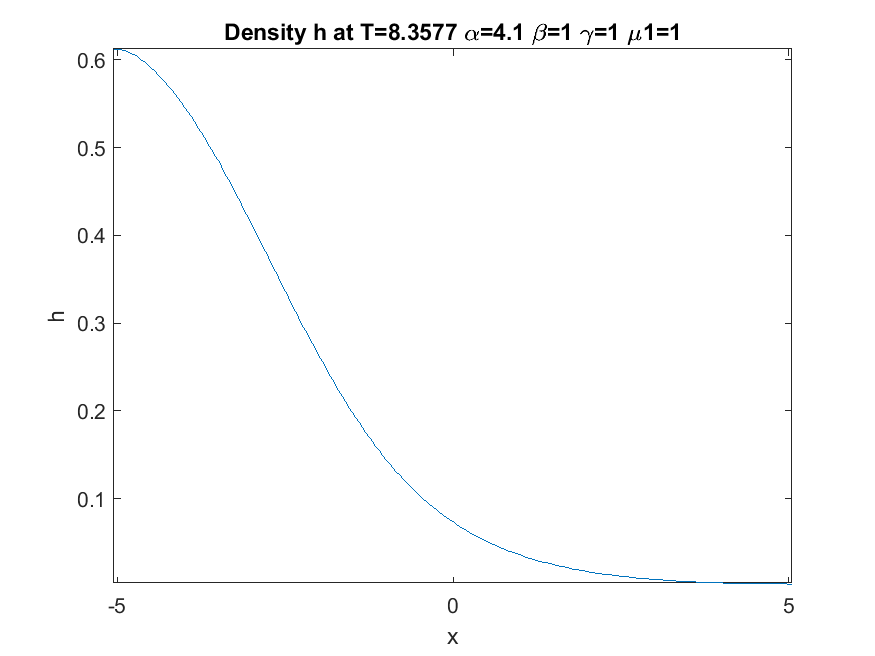}\ 
	\includegraphics[width=0.24\linewidth]{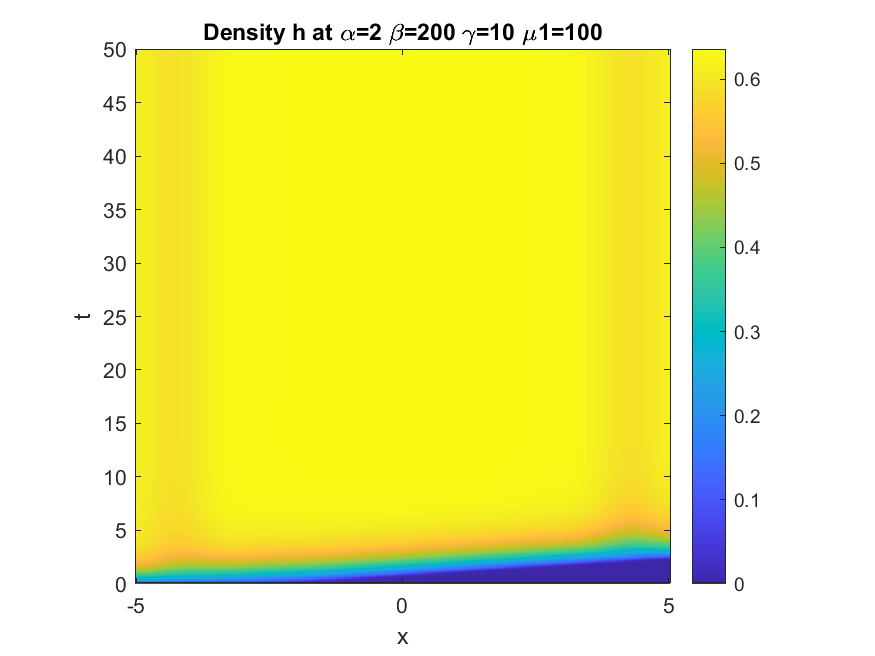}
	\caption{Simulations of model \eqref{IBVP} with $h$-dependent kernels $J_1$, $J_2$ from \eqref{j1} and \eqref{j2} for $\beta = \gamma = \mu_1=1$, $\alpha = 2,4,4.1$. In the 3rd column a blow-up occurs in the next time step. The 4th column shows patterns for $\alpha=2$, $\beta = 200$, $\gamma =10$, $\mu_1=100$.}\label{fig:5}
\end{figure}

\noindent
Comparing the height of the peaks in columns 3 and 4 of Figure \ref{fig:6} for uniform kernels (i.e., first two rows therein) clearly shows the dampening effect of interspecific interactions. Moreover, the solution of \eqref{modelww} reaches its maximum accumulation at the left boundary faster than the solution of \eqref{IBVP}, which is again due to the supplementary interspecific dampening in the latter model. Figure \ref{fig:4} shows that in model \eqref{modelww}
a blow-up already occurs for relatively smaller values of $\alpha$.\\[-2ex]


\noindent
For the $h$-dependent kernels $J_1$ and $J_2$ from \eqref{j1} and \eqref{j2} the solution $u$ in Figure \ref{fig:5} rapidly accumulates at the left boundary and then invades the whole domain aggregating much less at the boundaries than in Figures \ref{fig:2} and \ref{fig:3}. As mentioned in Table \ref{tab:1} the blow-up already occurs for $\alpha^* = 4.1$. The 4th column in Figure \ref{fig:5} shows one example of pattern formation for a certain choice of parameters.

\section{Discussion}\label{sec:discussion}

\noindent
As mentioned in Section \ref{sec:model}, the model introduced here extends previous settings \cite{LiChSu} and, in a certain sense, \cite{EckardtSu,SZYMASKA2009}. As in \cite{LiChSu,EckardtSu}, the main mathematical challenge comes from the interaction strengths $\alpha, \beta\ge 1$ present in the nonlocal terms; interspecific interactions did not add further difficulties as far as global existence and boundedness are concerned. In contrast to \cite{EckardtSu} we do not have here any myopic diffusion, nor taxis terms, which saves us the efforts otherwise needed to estimate first derivatives of the tactic signal. The model with interspecific interactions from \cite{SZYMASKA2009} involves haptotaxis, but there $\alpha =\beta =\gamma =1$ and the lack of transitions from one population to another, along with the assumptions made on initial data, convolution kernels, and coefficient functions render the analysis therein more accessible.\\[-2ex]
	
	\noindent
	The missing diffusion of $w$-cells required the construction of the approximate problem in Section \ref{sec:analysis1}. Introducing the term $-wu$ in the dynamics of $w$-cells helped ensure in that problem the boundedness of $w_\varepsilon$, with the aid of a comparison principle. Such term does have a biological motivation as well: it describes competition between active and inactive cells, which in our model is also triggered by the acidity profile, as both tumor cell phenotypes extrude protons in the interstitial space (the active ones more than their quiescent counterparts). \\[-2ex]
	
	\noindent
	As in \cite{LiChSu,EckardtSu}, the condition \eqref{bedalphbet} is not sharp: the numerical simulations suggest that the solution also exists globally for certain pairs $(\alpha, \beta)$ which do not satisfy that requirement. Interestingly, the critical value $\alpha ^*$ for which a solution ceases to exist does not seem to be an absolute $\alpha$-minimum, but can jump to higher or lower values, depending on the particular combination of the other parameters $\beta, \gamma, \mu_1$ (even for the same choice of kernels), as seen in Table \ref{tab:2}. Indeed, there seems to be an $\alpha ^{**}>\alpha^*$ such that the solution blows up in finite time for $\alpha \ge \alpha ^{**}$, but stays global for certain values $\alpha \in (\alpha^*, \alpha ^{**})$. A thorough mathematical investigation thereof remains, however, open. As Table \ref{tab:1} shows, $\alpha ^*$ also depends on the choice of convolution kernels (for fixed parameter values); this greatly complicates the analysis of blow-up behavior, due to the unlimited degrees of freedom one has for such choices, notwithstanding conditions \eqref{cond-kernels}. \\[-2ex]
	
	\noindent
	A rigorous mathematical stability and pattern analysis for this kind of PDE-ODE-PDE models seems to be out of reach with the established approaches (see, e.g., \cite{AndreguettoMaciel2021,Ni2018,Merchant2011,Simoy2023}), mainly due to the nonlinearities featuring weak Allee and overcrowding/competition effects with the respective interaction strengths $\alpha, \beta, \gamma$, which preclude from identifying nontrivial steady-states even in the absence of diffusion; the phenotypic switch terms only add difficulty to such attempts. The numerical simulations  performed in Section \ref{sec:simulations} give some insight into the long-term and patterning behavior of solutions, suggesting that the solution seems to be able to approach in the long term some stable state and to exhibit patterns, depending (as in \cite{EckardtSu}) on the choice of kernels and the parameter combination. The model extension with $w$-cells and interspecific interactions does not change substantially the type and shape of obtained oscillatory patterns, but does have an influence on the peaks of $u$-cell aggregates. The noticed dampening effect also contributes to detering solution blow-up or at least ensuring global boundedness for substantially larger $\alpha$ values which, again, depend on the choice of the convolution kernels and of the interaction strength $\gamma $.  \\[-2ex]
	
	\noindent
	Our analysis explicitly required the diffusion coefficient $\psi (w,h)$ to be nondegenerate. Alleviating this assumption leads to further mathematical challenges, when trying to obtain (as usually in such proofs) a bound on $\nabla u$ from the ODE for $w$. Indeed, the problem thereby relies on $u$ being involved in the growth, instead of the decay term. On the other hand, considering such nonlinear diffusion is motivated from a biological viewpoint, in order to account e.g., for chemokinesis \cite{Fok2006,Hayashi2008,Petrie2009}.

\appendix
\section{Auxiliary lemmas}\label{sec:appA}

\begin{Lemma} \label{lemholprod}
	Let $\vartheta,\kappa \in (0,1)$. As e.g., in \cite{GilbargTr}, we denote by $C^{\vartheta}$ the usual H\"older space $C^{0,\vartheta}$.\\
	Let $u \in  C^{\vartheta}(\oa)$ and $v \in C^{\kappa}(\oa)$. Then:
	\begin{enumerate}
		\item $uv \in C^{\min\{\vartheta,\kappa\}}(\oa)$,
		\item $u^r \in C^{\vartheta r}(\oa)$ if $r \in (0,1)$,
		\item $u^r \in C^{\vartheta}(\oa)$ if $r >1$,
		\item $\frac{1}{u} \in C^{\vartheta}(\oa)$ if $u\neq 0$ in $\oa$.
	\end{enumerate}
\end{Lemma}
\begin{proof}
	\begin{enumerate}
		\item See (4.7) in \cite{GilbargTr}.
		\item Clear.
		\item We estimate for $x,y \in \oa$ with the mean value theorem that
		\begin{align*}
			|u^r(x) - u^r(y)| \leq r \|u\|_{L^{\infty}(\Omega)}^{r-1} |x-y| \leq \C |x-y|^{\vartheta}.
		\end{align*}
		\item The function $u$ is continuous. Hence, there is $m := \min_{x\in\oa} |u(x)| > 0$ and we can estimate for $x,y \in \oa$ that
		\begin{align*}
			\left|\frac{1}{u(x)} - \frac{1}{u(y)}\right| = \frac{|u(y)-u(x)|}{|u(y)u(x)|} \leq 
			\frac{|u(y)-u(x)|}{m^2} \leq C|x-y|^{\vartheta}.
		\end{align*}
	\end{enumerate}
\end{proof}

\begin{Lemma} \label{lemholuu0}
	Let $\vartheta,\kappa\in (0,1)$ and $T\in (0,1)$.
	\begin{enumerate}
		\item  If $u \in \cet$, then it holds that 
		\begin{align} \label{umu0e}
			\|u-u(\cdot, 0)\|_{\ct } \leq \max\{1,\Cr{einbsob}\} T^{\frac{1+\vartheta}{2}}\|u\|_{\cet}.
		\end{align}
		\item If $u \in \czt$, then it holds that 
		\begin{align} \label{umu0}
			\|u-u(\cdot, 0)\|_{\cet } \leq \max\{1,\Cr{einbsob}\} T^{\frac{1}{2}\min\left\{\vartheta,1-\vartheta\right\}}\|u\|_{\czt}.
		\end{align}
		\item If $u \in \czkt$, then it holds that 
		\begin{align}
			\|u-u(\cdot, 0)\|_{\ct } \leq \max\{1,\Cr{einbsob}\} T^{\frac{1}{2}\min\left\{1+\kappa,1-\vartheta\right\}}\|u\|_{\czkt}.
		\end{align}
	\end{enumerate}
\end{Lemma}

\begin{proof}
	\begin{enumerate}
		\item Let $u \in \cet$. We estimate the norm term by term. Let $t,t' \in [0,T]$, $t \neq t'$, $x \in oa$. First, we estimate
		\begin{align}
			\frac{|u(x,t) - u(x,0)- (u(x,t') - u(x,0))|}{|t-t'|^{\frac{\vartheta}{2}}} &= \frac{|u(x,t) - u(x,t')|}{|t-t'|^{\frac{\vartheta}{2}}}\nonumber \\
			&\leq |t-t'|^{\frac{1}{2}} \frac{\|u(x,t) - u(x,t')\|}{|t-t'|^{\frac{1+\vartheta}{2}} } \leq T^{\frac{1}{2}} \langle u \rangle_{t,\ota}^{\frac{1+\vartheta}{2}} \label{normuinf}
		\end{align}
		Moreover, we use the continuous embedding of $W^{1,\infty}(\Omega)$ into $C^{\vartheta}(\oa)$ (see Theorem { 5 in \cite{Evans}, Section 5.6.2}) to estimate that
		\begin{align} 
			&\|u-u(\cdot,0)\|_{C(\ota)} + \langle u - u(\cdot,0) \rangle_{x,\ota}^{\vartheta} \nonumber\\
			\leq& \Cl{einbsob2} \left(\|u - u(\cdot,0)\|_{C(\ota)} + \sum_{i=1}^{d} \|u_{x_i}-u_{x_i}(\cdot,0)\|_{C(\ota)} \right) \nonumber\\
			\leq& \Cr{einbsob2} \left(T^{\frac{1+\vartheta}{2}}\langle u \rangle_{t,\ota}^{\frac{1+\vartheta}{2}} + T^{\frac{\vartheta}{2}} \sum_{i=1}^{d}\langle u_{x_i} \rangle_{t,\ota}^{\frac{\vartheta}{2}} \right) \label{anormest}
		\end{align}
		Putting this together with the supremum of \eqref{normuinf} and performing a similar calcution with the terms in \eqref{anormest} we obtain \eqref{umu0e} due to $T<1$.
		
		\item Let $u \in \czt$. Let $t,t' \in [0,T]$, $t \neq t'$, $x \in oa$. First, we conclude from the mean value theorem that
		\begin{align*} 
			\|u(x,t) - u(x,0)\|_{C(\ota)} + \langle u - u(x,0)\rangle_{t,\ota}^{\frac{1+\vartheta}{2}} \leq (T + T^{\frac{1-\vartheta}{2}})\|u_t\|_{C(\ota)}.
		\end{align*}
		Moreover, 
		\begin{align*}
			\langle u_{x_i} - u_{x_i}(\cdot,0) \rangle_{t,\ota}^{\frac{\vartheta}{2}} \leq T^{\frac{\alpha}{2}}\langle u_{x_i} \rangle_{t,\ota}^{\frac{1+\vartheta}{2}}
		\end{align*}
		As in \eqref{anormest} we use the continuous embedding of $W^{1,\infty}(\Omega)$ into $C^{\vartheta}(\oa)$ (see Theorem \textcolor{red}{??} in \cite{Evans}) to estimate 
		\begin{align*} 
			&\|u_{x_i}-u_{x_i}(\cdot,0)\|_{C{\ota}} + \langle u_{x_i} - u_{x_i}(\cdot,0) \rangle_{x,\ota}^{\vartheta}\\
			\leq& \Cr{einbsob2} \left(\|u_{x_i} - u_{x_i}(\cdot,0)\|_{C(\ota)} + \sum_{j=1}^{d} \|u_{x_ix_j}-u_{x_ix_j}(\cdot,0)\|_{C(\ota)} \right).
		\end{align*}
		Taking the supremum on the left side of 
		\begin{align*}
			|u_{x_i}(x,t) - u_{x_i}(x,0)| 
			= t^{\frac{1+\vartheta}{2}}\frac{|u_{x_i}(x,t) - u_{x_i}(x,0)|}{t^{\frac{1+\vartheta}{2}}}
			\leq T^{\frac{1+\vartheta}{2}} \langle u_{x_i} \rangle_{t,\ota}^{\frac{1+\vartheta}{2}}
		\end{align*}
		and analogously
		\begin{align*}
			|u_{x_ix_j}(x,t) - u_{x_ix_j}(x,0)| 
			\leq T^{\frac{\vartheta}{2}} \langle u_{x_ix_j} \rangle_{t,\ota}^{\frac{\vartheta}{2}}
		\end{align*}
		together with the above estimates we obtain \eqref{umu0}.
		\item We estimate as before with the help of the mean value theoren and the Sobolev embedding that \begin{align*}
			\langle u - u(\cdot,0) \rangle_{t,\ota}^{\frac{\vartheta}{2}} &\leq T^{\frac{1-\vartheta}{2}}\|u_t\|_{C(\ota)},\\
			\|u - u(\cdot,0)\|_{C(\ota)} + \langle u - u(\cdot,0) \rangle_{x,\ota}^{\vartheta} 
			&\leq \Cr{einbsob2} \left(T \|u_t\|_{C(\ota)} + T^{\frac{1+\kappa}{2}}\sum_{i=1}^d \langle u_{x_i} \rangle_{t,\ota}^{\frac{1+\kappa}{2}}  \right).
		\end{align*}
	\end{enumerate}
\end{proof}

\begin{Lemma} \label{lemholj}
	Let $T>0$, $\Omega \subset \R^d$ with sufficiently smooth boundary, $\vartheta\in (0,1)$, $\beta \geq 1$ and $J(x,h)\ast \ub = \io J_i(x-y,h(y))u(y) \td y$ and satisfying 
	\begin{align}
		|J(x,h_1) - J(x,h_2)| \leq L(x) |h_1-h_2|, \label{jprop1}\\
		L, \, J(\cdot,0) \in L^p(B) \label{jprop2}
	\end{align}
	for $L(\cdot)\geq 0$. $p\in(1,\infty)$, $h_1,h_2 \geq 0$.
	\begin{enumerate}
		\item If $u,h \in L^{\infty}(\ot)$ then, $J(x,h)\ast \ub \in L^{\infty}(\ot)$ and satisfies
		\begin{align*}
			\|J(x,h)\ast \ub\|_{L^{\infty}(\ot)} \leq \|u\|^{\beta}_{L^{\infty}(\ot)} \left(\|L\|_{L^1(B)} \|h\|_{L^{\infty}(\ot)} + \|J(x,0)\|_{L^1(B)} \right).
		\end{align*}
		\item If $u,h \in \ct$ then, $J(x,h)\ast \ub$ is in $C^{\kappa,\frac{\kappa}{2}}(\ota)$ for $\kappa := \min \{\vartheta,\frac{p-1}{p}\}$. 
	\end{enumerate}
\end{Lemma}
\begin{proof}
	\begin{enumerate}
		\item Let $u,h \in L^{\infty}(\ot)$, $x \in \Omega$ and $t\in [0,T]$. Using $\eqref{estj}$ and $\eqref{lpji}$ we can estimate 
		\begin{align*}
			\left|J(x,h)\ast u (t)\right| &\leq \left|\int_{\Omega} J(x-y,h(y,t)) \ub(y,t) \td x\right|\\
			&\leq \|u\|^{\beta}_{L^{\infty}(\ot)} \left|\io L(x-y)|h(y,t)| + |J(x,0)| \td x\right|\\
			&\leq \|u\|^{\beta}_{L^{\infty}(\ot)} \left(\|L\|_{L^1(B)} \|h\|_{L^{\infty}(\ot)} + \|J(x,0)\|_{L^1(B)} \right).
		\end{align*}
		\noindent Taking the supremum on the left-hand side, we conclude that $J(x,h)\ast \bu \in L^{\infty}(\ot)$.
		\item Let $u,h \in \ct$. Let $x\in \oa$ and $t_1,t_2 \in [0,T], t_1 \neq t_2$. Then, we can estimate using the properties of $J$, the H\"older continuity of $u$ and $h$ and the mean value theorem that
		\begin{align}
			&|J(x,h)\ast \bu (t_1) - J(x,h)\ast \bu (t_2)|\nonumber\\
			=& \left|\io J(x-y,h(y,t_1))\ub(y,t_1) - J(x-y,h(y,t_2))\ub(y,t_2)\td y\right|\nonumber\\
			\leq&  \left|\io \left(J(x-y,h(y,t_1)) - J(x-y,h(y,t_2))\ub(y,t_1)\right)\td y\right| 
			+\left|J(x-y,h(y,t_2))\left(\ub(y,t_1)- \ub(y,t_2)\right) \td y\right|\nonumber\\
			\leq& \|u\|_{L^{\infty}(\Omega)}^{\beta}\io L(x-y)|h(y,t_1)-h(y,t_2)| \td y \nonumber\\
			&+ \beta  \|u\|_{L^{\infty}(\Omega)}^{\beta-1} \io \left(L(x-y)|h(y,t_2)| + J(x-y,0)\right)|u(y,t_1)-u(y,t_2)| \td y\nonumber\\
			\leq& \Cl{constJh1}\left(\|u\|_{L^{\infty}},\|h\|_{L^{\infty}}, \|L\|_{L^1(B)}, \|J(\cdot,0)\|_{L^1(B)},\beta\right)\left(\langle h\rangle_{t,\ota}^{\frac{\vartheta}{2}} + \langle u\rangle_{t,\ota}^{\frac{\vartheta}{2}} \right). \label{estjht}
		\end{align}
		Now, let $x_1,x_2 \in \oa, x_1 \neq x_2$ and $t \in [0,T]$. Then, we can estimate
		\begin{align}
			&|J(x_1,h)\ast \bu (t) - J(x_2,h)\ast \bu (t)|\nonumber\\
			=& \left|\int_{\R^d} J(x_1-y,h(y,t)) \ub(y)\chi_{\Omega}(y) - J(x_2-y,h(y,t)) \ub(y)\chi_{\Omega}(y) \td y\right|\nonumber\\
			=& \left|\int_{\R^d} J(z,h(x_1-z,t)) \ub(x_1-z)\chi_{\Omega}(x_1-z) - J(z,h(x_2-z,t)) \ub(x_2-z)\chi_{\Omega}(x_2-z) \td z\right|\nonumber\\
			\leq& \left|\int_{\R^d} \left(J(z,h(x_1-z,t)) \ub(x_1-z) - J(z,h(x_2-z,t)) \ub(x_2-z)\right)\chi_{\Omega}(x_1-z)\chi_{\Omega}(x_2-z) \td z\right|\label{estjhx1}\\
			&+ \int_{\R^d}\left| J(z,h(x_1-z,t)) \ub(x_1-z)\chi_{\Omega}(x_1-z)(1-\chi_{\Omega}(x_2-z)) \right| \td y \label{estjhx2}\\ 
			&+\int_{\R^d} \left| J(z,h(x_2-z,t)) \ub(x_2-z)(1-\chi_{\Omega}(x_1-z))\chi_{\Omega}(x_2-z)  \right|\td z. \label{estjhx3}
		\end{align} 
		We can estimate \eqref{estjhx1} analogously to \eqref{estjht}. To estimate \eqref{estjhx2} we define the sets
		\begin{align*}
			S_{x_1x_2} &:= \left\{z\in\R^d : x_1-z \in \Omega, x_2-z \notin \Omega\right\},\\
			G_{x_1x_2} &:= \left\{z \in \R^d : x_1-z \in \Omega, \text{dist}\left(x_1-z,\partial \Omega\right)\leq |x_1-x_2|\right\}.
		\end{align*}
		Let $z \in	S_{x_1x_2}$. Then, $a:=x_1-z \in \Omega$ and $x_2-z = x_2-x_1+a \notin \Omega$. Assume $\text{dist}\left(x_1-z,\partial \Omega\right) > |x_1-x_1|$. Then, $\overline{B_{|x_1-x_2|}(a)} \subset \Omega$. Hence, $a + x_2 - x_1 \in \Omega$ which leads to a contradiction. Consequently, $S_{x_1x_2} \subset G_{x_1x_2}$ and $|S_{x_1x_2}| \leq |G_{x_1x_2}|\leq C_{\Omega}|x_1-x_2|$ holds for $\partial \Omega$ sufficiently smooth. Then, we can estimate
		\begin{align*}
			|\eqref{estjhx2}| 
			&= \left|\int_{S_{x_1x_2}} J(z,h(x_1-z,t)) \ub(x_1-z) \td z\right|\\
			&\leq \|u\|^{\beta}_{L^{\infty}(\Omega)} \int_{G_{x_1x_2}} L(z)h(x_1-z,t) + J(z,0) \td z\\
			&\leq \|u\|^{\beta}_{L^{\infty}(\Omega)} \max\{1, \|h\|^{\beta}_{L^{\infty}(\Omega)}\}\left(\|L\|_{L^p(B)} + \|J(\cdot,0)\|_{L^p(B)}\right)|G_{x_1x_2}|^{\frac{p-1}{p}}\\ 
			&\leq \Cl{constJh2}|x_1-x_2|^{\frac{p-1}{p}}.
		\end{align*}
		We can estimate \eqref{estjhx3} analogously. Putting this together with the estimates of the other terms, we conclude $J(x,h)\ast \ub \in C^{\kappa,\frac{\kappa}{2}}(\ota)$ for $\kappa := \min \{\vartheta,\frac{p-1}{p}\}$.
	\end{enumerate}
\end{proof}

\begin{Lemma} \label{lemodeest}
	Let $T>0$. If $f\in C([0,T];\R_0^+)\cap C^1((0,T);\R_0^+)$ satisfies on $(0,T)$ for $K_1,K_2 >0$ the inequality
	\begin{align*}
		f'(t) + K_1 f(t) \leq K_1K_2,
	\end{align*}
	then, for $t\in (0,T)$ it holds that
	\begin{align*}
		f(t) \leq K_2 + f(0).
	\end{align*}
\end{Lemma}
\begin{proof}
	Let $t\in (0,T)$. We differentiate
	\begin{align*}
		\left(f(t)e^{K_1t}\right)' = (f'(t) + K_1 f(t))e^{K_1t} \leq K_1K_2 e^{K_1t}.
	\end{align*}
	Integrating over $[0,t]$ we obtain
	\begin{align*}
		f(t) e^{K_1t} \leq K_1 K_2 \int_0^t e^{K_1s} \td s + f(0) = K_2 \left(e^{K_1t} - 1\right) +f(0).
	\end{align*}
	Consequently,
	\begin{align*}
		f(t) \leq K_2 \left(1 - e^{-K_1t}\right) + f(0)e^{K_1t} \leq K_2 + f(0).
	\end{align*}
\end{proof}

\section{Theorems from books/papers}\label{sec:appB}
\begin{Lemma} \label{lemLiSuChThm1}
	Let $u \in C^1(\oa)$. We consider $\alpha,\beta\geq 1$ satisfying \eqref{bedalphbet}, $q \geq \max\{1,\beta+\alpha-1\}$,
	\begin{align*}
		\max\left\{\frac{d(\alpha-1)}{q},\frac{2(\alpha-1)}{q},1\right\} < r \leq \frac{2(q+\alpha-1)}{q}
	\end{align*}
	and
	\begin{align*}
		s = \begin{cases}
			\infty, &d=1\\
			\left(\max\{\frac{2\beta}{\beta+1-\alpha},\frac{2qr}{qr-2(\alpha-1)}\}, \infty\right), &d=2\\
			\frac{2d}{d-2}, &d>2,
		\end{cases}
	\end{align*}
	Then, for $K_1>0$ it holds that
	\begin{align} \label{ineqLSC}
		\io u^{q+\alpha-1} \td x 
		&\leq \frac{2(q-1)}{q^2K_1}\io |\nabla u^{\frac{q}{2}}|^2 \td x
		+ \Cl{constLSC1}(K_1,q,r)\|u^{\frac{q}{2}}\|_{L^r(\Omega)}^{2r\frac{\frac{2(q+\alpha-1)}{s}-q}{qr\left(\frac{2}{s}-1\right)+2(\alpha-1)}}
		+ \Cl{constpp}(r)^{\frac{qr-2(q+\alpha-1)}{\frac{qr}{s}-1}}
	\end{align}
	where 
	\begin{align*}
		\Cr{constLSC1}(K_1,q,r) &:= 2\left(\frac{\Cr{constsp}^2q^2K_1}{q-1}\right)^{\frac{qr-2(q+\alpha-1)}{qr\left(1-\frac{2}{s}\right) +2(\alpha-1)}} + \Cr{constpp}(r)^{\frac{qr-2(q+\alpha-1)}{\frac{qr}{s}-1}},\\
		\Cr{constpp}(r) &:= 4C_S(s)|\Omega|^{\frac{r-2}{2r}},\\
		\Cr{constsp} &:= 2C_S(s)\left(1+2C_P\right).
	\end{align*}
	Here, $C_S(s)$ denotes the Sobolev embedding constant from $W^{1,2}(\Omega)$ into $L^s(\Omega)$  and $C_P>0$ denotes the constant from the Poincar\'e inequality.
	
	\noindent Moreover, for $K_1,K_2>0$ it holds that 
	\begin{align} \label{ineqLSC2}
		\io u^{q+\alpha-1} \td x \leq \frac{2(q-1)}{q^2K_1}\io |\nabla u^{\frac{q}{2}}|^2 \td x 
		+ \frac{1}{K_2} \io \ub \td x \io u^{q+\alpha-1} \td x + \Cr{constLSC2}\left(K_1,K_2,q\right),
	\end{align}
	where we consider $s\in \left(\max\{\frac{2\beta}{\beta+1-\alpha},\frac{2(q+\alpha-1+\beta)}{q-\alpha+1+\beta}\},\infty\right)$ for $d=2$ (if $d\neq 2$, we consider $s$ as above) and
	\begin{align*}
		\Cr{constLSC2}\left(K_1,K_2,q\right) :=& \left(2\left(\frac{\Cr{constsp}^2q^2K_1}{q-1}\right)^{\frac{q+\alpha-1-\beta}{q-\alpha+1+\beta-2\frac{q+\alpha-1+\beta}{s}}}  + \Cr{constpp2}(q)^{\frac{q+\alpha-\beta-1}{q-\frac{q+\alpha-1+\beta}{s}}}\right)^{\frac{q-\alpha+1+\beta - \frac{2(q+\alpha-1+\beta)}{s}}{\beta+1-\alpha - \frac{2\beta}{s}}}\\
		&\cdot K_2^{\frac{q-\frac{2(q+\alpha-1)}{s}}{\beta+1-\alpha-\frac{2\beta}{s}}}
		+ \Cr{constpp2}(q)^{\frac{q+\alpha-\beta-1}{q-\frac{q+\alpha-1+\beta}{s}}},
	\end{align*}
	and
	\begin{align*}
		\Cr{constpp2}(q) := 4C_S(s)|\Omega|^{\frac{1}{2}-\frac{q}{q+\alpha-1+\beta}}
	\end{align*}
\end{Lemma}
\begin{proof}
	Let
	\begin{align*}
		\lambda = \frac{\frac{q}{2(q+\alpha-1)} - \frac{1}{r}}{\frac{1}{s} - \frac{1}{r}}.
	\end{align*}
	Then, $\lambda \in [0,1)$ due to our choice of parameters. We state inequality (2.11) from the proof of Theorem 1 in \cite{LiChSu} (with $B(x,\delta)$ replaced by $\Omega$)
	\begin{align*}
		\io u^{q+\alpha-1} \td x 
		\leq 2\left(\Cr{constsp}\|\nabla u^{\frac{q}{2}} \|_{L^2(\Omega)} \right)^{\frac{2\lambda(q+\alpha-1)}{q}}
		\|u^{\frac{q}{2}}\|_{L^r(\Omega)}^{\frac{2(1-\lambda)(q+\alpha-1)}{q}} 
		+ \left(\Cr{constpp}(r)^{\lambda} \|u^{\frac{q}{2}}\|_{L^r(\Omega)}\right)^{\frac{2(q+\alpha-1)}{q}}.
	\end{align*}
	where 
	\begin{align*}
		\Cr{constpp}(r) := 4C_S(s)|\Omega|^{\frac{r-2}{2r}}, \Cr{constsp} := 2C_S(s)\left(1+2C_P\right). 
	\end{align*}
	We proceed as in the proof of (2.14) in \cite{LiChSu}. Applying Young's inequality twice and regrouping the terms we conclude that for $K_1 > 0$ it holds that
	\begin{align}
		\io u^{q+\alpha-1} \td x 
		\leq& \frac{2(q-1)}{q^2K_1}\|\nabla u^{\frac{q}{2}}\|_{L^2(\Omega)}^2 \\
		&+ 2 \left(\Cr{constsp}^{\frac{2\lambda(q+\alpha-1)}{q}} \|u^{\frac{q}{2}}\|_{L^r(\Omega)}^{\frac{2(1-\lambda)(q+\alpha-1)}{q}}\left(\frac{q^2K_1}{q-1}\right)^{\frac{\lambda(q+\alpha-1)}{q}}\right)^{\frac{q}{q-\lambda(q+\alpha-1)}}\|u^{\frac{q}{2}}\|_{L^r(\Omega)}^{\frac{2(1-\lambda)(q+\alpha-1)}{q-\lambda(q+\alpha-1)}}\nonumber\\
		&+ \Cr{constpp}(r)^{\frac{2\lambda(q+\alpha-1)}{q}}\left(\|u^{\frac{q}{2}}\|_{L^r(\Omega)}^{\frac{2(1-\lambda)(q+\alpha-1)}{q}} + 1\right)\nonumber\\
		\leq& \frac{2(q-1)}{q^2K_1}\|\nabla u^{\frac{q}{2}}\|_{L^2(\Omega)}^2
		+ \Cr{constLSC1}(K_1,q,r)\|u^{\frac{q}{2}}\|_{L^r(\Omega)}^{\frac{2(1-\lambda)(q+\alpha-1)}{q-\lambda(q+\alpha-1)}}
		+ \Cr{constpp}(r)^{\frac{2\lambda(q+\alpha-1)}{q}}, \label{ineqLSC3}
	\end{align}
	where
	\begin{align*}
		\Cr{constLSC1}(K_1,q,r) := 2\left(\frac{\Cr{constsp}^2q^2K_1}{q-1}\right)^{\frac{\lambda(q+\alpha-1)}{q-\lambda(q+\alpha-1)}} + \Cr{constpp}(r)^{\frac{2\lambda(q+\alpha-1)}{q}}.
	\end{align*}
	Inserting the definition of $\lambda$ we obtain inequality \eqref{ineqLSC}. Moreover, we state inequality (2.17) from the proof of Theorem 1 in \cite{LiChSu}, i.e. 
	\begin{align*}
		\|u^{\frac{q}{2}}\|_{L^{\frac{q+\alpha-1+\beta}{q}}(\Omega)}^{\frac{2(1-\lambda)(q+\alpha-1)}{q-\lambda(q+\alpha-1)}} 
		\leq \left(\|u^{\frac{q}{2}}\|_{L^{\frac{2\beta}{q}}(\Omega)}^{\frac{2\beta}{q}}\|u^{\frac{q}{2}}\|_{L^{\frac{2(q+\alpha-1)}{q}}(\Omega)}^{\frac{2(q+\alpha-1)}{q}}\right)^{\frac{q-\frac{2(q+\alpha-1)}{s}}{q-\alpha+1+\beta - \frac{2(q+\alpha-1+\beta)}{s}}},
	\end{align*}
	wbere due to our choice of $\alpha$ and $\beta$ in \eqref{bedalphbet} and $s$ it holds that
	\begin{align*}
		\frac{q-\frac{2(q+\alpha-1)}{s}}{q-\alpha+1+\beta - \frac{2(q+\alpha-1+\beta)}{s}}<1.
	\end{align*}
	Now, we estimate the term below from \eqref{ineqLSC3} for $r = \frac{q+\alpha-1+\beta}{q}$ as in (2.19) from \cite{LiChSu}. Young's inequality leads for $K_2>0$ to the estimate
	\begin{align*}
		&\Cr{constLSC1}(K_1,q,r)\|u^{\frac{q}{2}}\|_{L^{\frac{q+\alpha-1+\beta}{q}}(\Omega)}^{\frac{2(1-\lambda)(q+\alpha-1)}{q-\lambda(q+\alpha-1)}}\\
		\leq& \frac{1}{K_2} \io \ub \td x \io u^{q+\alpha-1} \td x + \Cr{constLSC1}(K_1,q,r)^{\frac{q-\alpha+1+\beta - \frac{2(q+\alpha-1+\beta)}{s}}{\beta+1-\alpha - \frac{2\beta}{s}}}
		K_2^{\frac{q-\frac{2(q+\alpha-1)}{s}}{\beta+1-\alpha-\frac{2\beta}{s}}}.
	\end{align*}
	Inserting this estimate and our choice of $r$ into \eqref{ineqLSC3} we arrive at \eqref{ineqLSC2}.
\end{proof}

\begin{Lemma} \label{lemLiSuChThm12}
	Consider $s$ as in Lemma \ref{lemLiSuChThm1}. Set $p_k := 2^k+h$ for $h := \frac{2(s-1)(\alpha-1)}{s-2}$, $k\in \N$. Then, for $k\geq 2$ it holds that
	\begin{align} \label{LSCgl1}
		\frac{\frac{2(p_k+\alpha-1)}{s}-p_k}{2q_{k-1}\left(\frac{2}{s}-1\right)+2(\alpha-1)} = 1,\\
		\frac{2q_{k-1}-2(p_k+\alpha-1)}{2q_{k-1}\left(1-\frac{2}{s}\right) +2(\alpha-1)} = \frac{s}{s-2} \label{LSCgl2}
	\end{align}
	and
	\begin{align}\label{LSCineq1}
		\frac{2q_{k-1}-2(p_k+\alpha-1)}{\frac{2q_{k-1}}{s}-p_k} \leq \alpha + 1.
	\end{align}
\end{Lemma}
\begin{proof}
	See part 2 of the proof of Theorem 1 in \cite{LiChSu}, where they show (in their notation) that $\frac{Q_k}{r_k} = 2$ in (2.28). Comparing this to the term in \eqref{LSCgl1} we obtain the desired equality. Moreover, they show (in their notation) in (2.32) that $\frac{\lambda_k(p_k + \alpha -1)}{p_k - \lambda_k(p_k + \alpha -1)} = \frac{s}{s-2}$ where the left-hand side equals $\frac{2q_{k-1}-2(p_k+\alpha-1)}{2q_{k-1}\left(1-\frac{2}{s}\right) +2(\alpha-1)}$. This gives us \ref{LSCgl2}. We obtain \eqref{LSCineq1} from (2.33) in \cite{LiChSu}, where $\frac{2q_{k-1}-2(p_k+\alpha-p_k)}{\frac{2q_{k-1}}{s}-1} = \frac{2\lambda_k(p_k +\alpha -1)}{p_k} \leq \alpha + 1$ was shown
\end{proof}

\begin{Lemma}\label{lemLSC21}
	We consider nonnegative $y_k \in C([0,\infty)) \cap C^1(0,\infty)$ for $k=0,1,2,\dot{3}$ that satisfies
	\begin{align*}
		y_k'(t) + c_k y_k(t) \leq c_k A_k \max\left\{1,\sup_{t\geq 0}y_{k-1}^2(t)\right\},
	\end{align*}
	where $A_k = \bar{a}2^{Dk}\geq 1$ and $c_k, \bar{a}, D >0$. We assume that there is $K > 0$ such that $y_k(0) \leq K$. Then, for all $m\geq 1$ it holds that
	\begin{align*}
		y_k(t) \leq (2\bar{a})^{2^{k-m+1}-1} 2^{D\left(2(2^{k-m}-1) + m 2^{k-m+1} -k \right)} \max\left\{\sup_{t \geq 0} y_{m-1}^{2^{k-m+1}}(t), K, 1 \right\}
	\end{align*}
\end{Lemma}
\begin{proof}
	See Lemma 2.1 in \cite{LiChSu} and adapt the proof with $K$ instead of $K^{2^k}$.
\end{proof}

\begin{Lemma} \label{lemMizuguchi}
	Let $\Omega \subset \R^d$ a bounded convex domain. Let $p = \infty$ if $d=1$, $p \in (2,\infty)$ if $d=2$ and $2 < p \leq\frac{2d}{d-2}$ if $d>2$. Then, it holds that
	\begin{align*}
		\| u \|_{L^p(\Omega)} \leq C_S(p)|u\|_{W^{1,2}(\Omega)},
	\end{align*}
	where
	\begin{align*}
		C_S(p) = \begin{cases}
			\max\left\{1, \frac{\diam(\Omega)|V|}{|\Omega|}\right\} & d=1,\\
			\sqrt{2}\max\left\{|\Omega|^{\frac{1}{p}-\frac{1}{2}},\frac{\diam(\Omega)^{1+\frac{p+2}{2p}d}\pi^{\frac{p+2}{4p}d}}{d|\Omega|}\frac{\Gamma\left(\frac{p-2}{4p}d\right)}{\Gamma\left(\frac{p+2}{4p}d\right)}\right\}\sqrt{\frac{\Gamma\left(\frac{d}{p}\right)}{\Gamma\left(\frac{p-1}{p}d\right)}}\left(\frac{\Gamma(d)}{\Gamma\left(\frac{d}{2}\right)}\right)^{\frac{p-2}{2p}} & d\geq 2.
		\end{cases}
	\end{align*}
	Here, $V := \bigcup_{x\in \Omega} \Omega_x$, where $\Omega_x := \left\{y-x : y\in\Omega\right\}$ for $x \in \Omega$, and $\Gamma$ denotes the Gamma function given by $\Gamma(x) = \int_0^{\infty} t^{x-1} e^{-t} \td t$ for $x > 0$.
\end{Lemma}
\begin{proof}
	See Theorems 2.1, 3.2 and 3.4 in \cite{Mizuguchi}.
\end{proof}

\phantomsection
\printbibliography
\end{document}